\documentclass[11pt,UTF8]{amsart}
\usepackage[top=28mm, bottom=28mm, left=31.5mm, right=31.5mm]{geometry}

\usepackage{amscd}
\usepackage{proof}
\usepackage[UKenglish]{babel}
\usepackage{amsmath}
\usepackage{amssymb}
\usepackage{latexsym}
\usepackage{amsthm}
\usepackage{amsfonts}
\usepackage{color}
\usepackage{hyperref}
\usepackage{url}

\hypersetup{
   colorlinks,
   citecolor=blue,
   filecolor=black,
   linkcolor=red,
   urlcolor=black
 }%

\usepackage[all]{xy}
\usepackage{tikz} 
\usepgflibrary{arrows}

\newtheorem{theorem}{Theorem}[section]
\newtheorem{lemma}[theorem]{Lemma}
\newtheorem{claim}[theorem]{Claim}
\newtheorem{proposition}[theorem]{Proposition}
\newtheorem{corollary}[theorem]{Corollary}
\newtheorem{fact}[theorem]{Fact}

\theoremstyle{definition}
\newtheorem{example}[theorem]{Example}
\newtheorem{remark}[theorem]{Remark}
\newtheorem{definition}[theorem]{Definition}
\newtheorem{notation}[theorem]{Notation}
\newtheorem{observation}[theorem]{Observation}
\newtheorem{preg}[theorem]{Question}

\def\Ind#1#2{#1\setbox0=\hbox{$#1x$}\kern\wd0\hbox to 0pt{\hss$#1\mid$\hss}
\lower.9\ht0\hbox to 0pt{\hss$#1\smile$\hss}\kern\wd0}
\def\Notind#1#2{#1\setbox0=\hbox{$#1x$}\kern\wd0\hbox to 0pt{\mathchardef
\nn=12854\hss$#1\nn$\kern1.4\wd0\hss}\hbox to
0pt{\hss$#1\mid$\hss}\lower.9\ht0 \hbox to
0pt{\hss$#1\smile$\hss}\kern\wd0}

\def\ind{\mathop{\mathpalette\Ind{}}}

\def\ind{\mathop{\mathpalette\Ind{}}}                              
\def\SU{\operatorname{SU}}                        

\newcommand{\bp}{\begin{proposition}}
\newcommand{\ep}{\end{proposition}}
\newcommand{\bd}{\begin{definition}}
\newcommand{\ed}{\end{definition}}
\newcommand{\bej}{\begin{ejem}}
\newcommand{\eej}{\end{ejem}}
\newcommand{\bl}{\begin{lemma}}
\newcommand{\el}{\end{lemma}}
\newcommand{\bh}{\begin{fact}}
\newcommand{\eh}{\end{fact}}
\newcommand{\bpreg}{\begin{preg}}
\newcommand{\epreg}{\end{preg}}
\newcommand{\bo}{\begin{obs}}
\newcommand{\eo}{\end{obs}}
\newcommand{\bcon}{\begin{conj}}
\newcommand{\econ}{\end{conj}}
\newcommand{\brmk}{\begin{remark}}
\newcommand{\ermk}{\end{remark}}
\newcommand{\bc}{\begin{corollary}}
\newcommand{\ec}{\end{corollary}}
\newcommand{\bconst}{\begin{const}}
\newcommand{\econst}{\end{const}}
\newcommand{\bdem}{\begin{proof}}
\newcommand{\edem}{\end{proof}}
\newcommand{\benum}{\begin{enumerate}}
\newcommand{\eenum}{\end{enumerate}}
\newcommand{\bitem}{\begin{itemize}}
\newcommand{\eitem}{\end{itemize}}
\newcommand{\be}{\begin{ejem}}
\newcommand{\ee}{\end{ejem}}
\newcommand{\bt}{\begin{theorem}}
\newcommand{\et}{\end{theorem}}
\newcommand{\bclaim}{\begin{claim}}
\newcommand{\eclaim}{\end{claim}}

\newcommand{\Le}{\mathcal{L}}
\newcommand{\U}{\mathcal{U}}
\newcommand{\Alg}{\operatorname{Alg}}
\newcommand{\Dim}{\operatorname{Dim}}
\def\tp{\operatorname{tp}}
\def\ldim{\operatorname{ldim}}

\def\acl{\textsf{acl}}
\def\scl{\textsf{scl}} 
\def\dcl{\textsf{dcl}}

\def\bigcupdot{\mathop{\hbox to 0pt{\hskip.43em$\dot{\ }$\hss}\bigcup}}

\definecolor{electricindigo}{rgb}{0.44, 0.0, 1.0}

\begin{document}
\title{Dimension and measure in pseudofinite $H$-structures}

\author{Alexander Berenstein}
\address{Alexander Berenstein\\ Departamento de Matemáticas. Universidad de los Andes\\ Carrera 1 No. 18A-10, Edificio H, Bogot\'a 111711, Colombia.}

\author{Dar\'io Garc\'ia}
\address{Dar\'io Garc\'ia\\ Departamento de Matemáticas. Universidad de los Andes\\ Carrera 1 No. 18A-10, Edificio H, Bogot\'a 111711, Colombia.}
\email{da.garcia268@uniandes.edu.co}

\author{Tingxiang Zou}
\address{Tingxiang Zou\\ Universit\'e Claude Bernard - Lyon 1. CNRS 5208, Institut Camille Jordan\\ 43 Blvd. du 11 Novembre 1918, F.69622 Villeurbanne Cedex, France.}
\curraddr{Fachbereich Mathematik und Informatik der Universit\"{a}t M\"{u}nster, Orl\'{e}ans-Ring 10, 48149 M\"{u}nster, Germany.}
\email{tzou@uni-muenster.de}
\thanks{The authors would like to thank Facultad de Ciencias Universidad de los Andes for its support through the project \emph{Expansiones densas de teor\'ias geom\'etricas: grupos y medidas}. The second author was supported by the Programa Estancias Posdoctorales de Colciencias, at Universidad de los Andes. Contract FP44842-178-2018, Convocatoria 784-2017. The third author is supported by the China Scholarship Council and partially supported by ValCoMo
(ANR-13-BS01-0006).
}
\date{\today}

\maketitle

\begin{abstract}
We study $H$-structures associated to $SU$-rank 1 measurable structures. We prove 
that the $SU$-rank of the expansion is continuous and that it is 
uniformly definable in terms of the parameters of the formulas. We also
introduce notions of dimension and measure for definable sets in the expansion and prove they are 
uniformly definable in terms of the parameters of the formulas.
\end{abstract}

\section{Introduction}
We say that a theory $T$ is \emph{geometric} if for any model $M\models
T$ the algebraic closure satisfies the exchange property and $T$
eliminates the quantifier $\exists^{\infty}$. When a theory is geometric and $M\models T$, the pair $(\mathcal{P}(M),\acl)$ is a pregeometry and elimination of $\exists^{\infty}$ guarantees that 
for every formula $\varphi(\bar{x},\bar{y})$ and every $k\leq |\bar{x}|$, the set 
$\{\bar a\in M^{|\bar{y}|}:\dim(\varphi(\bar{x},\bar{a}))=k\}$ is definable. 
There are many examples of geometric
theories, among them $SU$-rank 1 theories, dense o-minimal theories, the field of $p$-adic numbers in a single sort, etc.

In this paper we will deal with dense-codense expansions of geometric theories. 
Recall that if we let $H$ be a new unary predicate and
$M\models T$, the expansion $(M,H(M))$ is \emph{dense-codense} (see Definition
\ref{hstruc}) if it satisfies the
\emph{density property} (roughly, for any non-algebraic $\mathcal{L}$-formula
$\varphi(x,\bar{b})$ with parameters in $M$, $\varphi(H(M),\bar{b})\neq
\emptyset$) and $M$ satisfies the \emph{extension property} over
$H(M)$ (roughly, for any non-algebraic $\mathcal{L}$-formula $\varphi(x,\bar{b})$ with
parameters in $M$, $\varphi(M,\bar{b})\setminus \acl(H(M)\bar{b}) \neq
\emptyset$).

Lovely pairs are dense-codense expansions where
$H(M)\preceq M$, and $H$-structures are dense-codense expansions where
$H(M)$ is an algebraically independent set.

Many properties such as stability and NIP are preserved when going from a geometric theory to the corresponding theories of dense-codense pairs. Similarly, when we start with an $SU$-rank 1 theory, the algebraic dimension coincides with an abstract dimension coming from forking-independence, and the expansion corresponding to $H$-structures has a supersimple theory of $SU$-rank less than or equal to
$\omega$ (see \cite{BeVa2}). 

A \emph{measurable structure} (see \cite{MS}) is a structure equipped with a function that assigns a dimension and a measure to every definable set, which is uniformly definable in terms of their parameters and satisfies certain conditions similar to those satisfied by ultraproducts of finite fields. In this setting not only do we have control on the dimension of definable sets (as we do in the geometric setting), but the measure provides a finer account of the size of the corresponding sets. Some measurable structures can be built as ultraproducts of structures in a $1$-dimensional asymptotic class and are pseudofinite of $SU$-rank 1 (see \cite[Lemma 4.1]{MS}) and thus geometric. The work of the third author \cite{Zou} shows that in this special case, the corresponding $H$-structure is also pseudofinite. In contrast, the associated theory of a \emph{lovely pair} (cf. \cite{BeVa}, \cite{Va}) may not be pseudofinite (see \cite[Lemma 1.4]{Zou}).

It is a natural question to test whether these expansions, assuming the underlying theory is measurable of $SU$-rank 1, become some sort of \emph{generalized measurable structures}. Namely, can we construct appropriate notions of dimension and measure which depend definably on the parameters of the defining formulas for $H$-structures of measurable theories?

The goal of this paper is to give a positive solution to this problem. 
Our solution depends on a detailed study of $SU$-rank in $H$-structures, showing that it is continuous and Cantor additive. Drawing parallels with the notion of measurable structures, we use information associated to the $SU$-rank of a formula as an abstract notion of dimension and prove 
that for each formula it is uniformly definable in terms of its parameters. In addition, and using the fact that the base theory comes from a measurable structure, we introduce a notion of measure for definable sets in $H$-structures, proving also uniform definability.

We would like to point out that some expansions of measurable structures are already known to preserve 
measurability. For example, when $T$ is measurable, eliminates $\exists^\infty$ and $\acl=\dcl$, then adding a generic predicate does not change forking (see \cite[Theorem 2.7]{ChPi}), and the result gives a new measurable structure (see
\cite[Theorem 5.11]{MS}). There is even some liberty on the choice of the size of the new predicate. By contrast, $H$-structures of non-trivial theories (and in some cases lovely pairs) increase the $SU$-rank from finite to infinite. As measurable structures are essentially finite-dimensional, we could not expect H-structures/lovely pairs to stay within this family. Instead, the resulting notions of dimension and measure for $H$-structures satisfy the properties of a \emph{generalized measurable structure}, which is a generalization of the concept of measurable structures. This idea has been studied by several authors and will appear in the paper \cite{AMSW} by S. Anscombe, D. Macpherson, C. Steinhorn and D. Wolf, which is currently in preparation.

This paper is divided as follows. In Section \ref{section:preliminaries} we recall the definition of $H$-structures and show some of its properties. In particular we characterize the $SU$-rank of a tuple and prove it is Cantor-additive. In Section \ref{section:dimension} we show that the $SU$-rank of a formula is uniformly definable in terms of its parameters. In Section \ref{section:comparing-dimensions} we use these results to show, under certain additional hypothesis, that if the $SU$-rank of a tuple is $\omega\cdot n+k$, then the \emph{large dimension} $n$ corresponds to the coarse pseudofinite dimension with respect to $M$, while the number $k$ corresponds to the coarse dimension with respect to $H(M)$.  
Finally in Section \ref{section:measure} we introduce measures for formulas and prove they are uniformly definable in terms of their parameters.

\section{Preliminaries: $H$-structures of geometric theories}\label{section:preliminaries}

In this section, we review the notions of \emph{$H$-structures} of geometric theories from \cite{BeVa},\cite{BeVa2}  and their basic properties.

Recall that a theory $T$ in a language $\mathcal{L}$ is called \emph{geometric} if (1) it eliminates the quantifier $\exists^{\infty}$ and (2) in all models of $T$, the algebraic closure  satisfies the exchange property. Whenever $T$ is geometric and $M\models T$, 
 $\bar{a}$ a finite tuple in $M$ and $B, C\subset M$ we write $\dim(\bar{a}/B)$ for the length of a maximal subtuple of $\bar{a}$ that is algebraically independent over $B$, and we write $\bar{a} \ind_B C$ when $\dim(\bar{a}/B\cup C)=\dim(\bar{a}/B)$. If $A\subset M$, we write $A \ind_BC$ when $\bar{a} \ind_BC$ for all finite tuples $\bar{a}$ in $A$. We shall use the word \emph{independence} to mean algebraic independence inside models of the theory $T$.


\begin{definition} \label{hstruc}
Given a complete geometric theory $T$ in a language $\mathcal{L}$ and a model $M\models T$, let $H$ be a new unary predicate symbol and let $\mathcal{L}_H:=\mathcal{L}\cup \{ H \}$ be the extended language.
 Let $(M, H(M))$ denote an expansion of $M$ to $\mathcal{L}_H$, where $H(M): = \{ a\in M \mid H(a) \}$.
\begin{enumerate}
\item The pair $(M, H(M))$ has the \emph{density property} if for any $A\subseteq M$ of finite dimension and for any non-algebraic $\mathcal{L}$-type  $p(x)\in S_1(A)$,  $p(x)$ has a realization in $H(M)$. 
\item The pair $(M, H(M))$ has the \emph{extension property} if for any $A\subseteq M$ of finite dimension and for any non-algebraic $\mathcal{L}$-type $p(x)\in S_1(A)$, $p(x)$ has realizations in $M\setminus \acl(AH(M))$. 
\item An expansion $(M, H(M))$ which satisfies the density and the extension property is called  an \emph{$H$-structure} if, in addition, $H(M)$ is  an $\mathcal{L}$-algebraically independent subset of $M$.
\end{enumerate}
\end{definition}

It is easy to show by induction on the number of variables that the density property and the extension property also hold for types of independent tuples.

\begin{lemma} [\cite{BeVa2}, Lemma 2.3]
Let $T$ be a complete geometric theory and let $(M, H(M))$ be an expansion satisfying the density property and the extension property. Then whenever $A\subset M$ has a finite dimension and $p(x_1,\dots,x_n)\in S_n(A)$ 
is an $\mathcal{L}$-type of dimension $n$, $p(x_1,\dots,x_n)$ has realizations both in $H(M)^n$ and in 
$(M\setminus \acl(AH(M))^n$. Furthermore, we can choose $\bar b\in (M\setminus \acl(AH(M)))^n$ realizing $p(x_1,\dots,x_n)$ with $\dim(\bar b/\acl(AH(M)))=n$.
\end{lemma}

We will also use the expressions \emph{density property} and \emph{extension property} for the generalized versions of these properties. Although the class of $H$-structures is not a first order class, it corresponds to the class of sufficiently saturated models of a complete theory: 

\begin{theorem}[\cite{BeVa}, \cite{BeVa2}]
Given any geometric complete theory $T$, all the $H$-structures associated with $T$ are elementarily equivalent to one another. Furthermore, any $|T|^+$-saturated model of the common theory of $H$-structures is again an $H$-structure.
\end{theorem}

\begin{notation}
We will write $T^{ind}$ to denote the common complete $\Le_H$-theory of $H$-structures associated with $T$. Whenever $(M,H(M))\models T^{ind}$, and $\bar{a}$ is a tuple in $M$, we write $\tp_{\mathcal{L}}(\bar a)$ for the type in the
language $\mathcal{L}$ and $\tp_{\mathcal{L}_H}(\bar a)$ for the type in the language $\mathcal{L}_H$. Also, for $A\subset M$, we
write $H(A)$ for $A\cap H(M)$. We will also write $H(\bar{a})$ for the subtuple of $\bar{a}$ consisting of those elements that also belong to $H(M)$.

\end{notation}

\begin{definition} Let $(M, H(M))\models T^{ind}$.
A subset $A \subset M$ is called \emph{$H$-independent} if $A\ind_{H(A)}H(M)$.
\end{definition}

\begin{lemma}[\cite{BeVa}, \cite{BeVa2}] \label{fundlem1}
Let $(M, H(M))\models T^{ind}$. Then for any $H$-independent tuples $\bar{a}$ and $\bar{b}$, 
\[ \tp_{\Le_H}(\bar{a}) = \tp_{\Le_H}(\bar{b}) \text{ if and only if } \tp_{\mathcal{L}}(\bar{a} H(\bar{a})) = \tp_{\mathcal{L}}(\bar{b} H(\bar{b}))
\]
\end{lemma}

\begin{lemma}[\cite{BeVa}, \cite{BeVa2}] \label{fundlem2}
Any formula in $T^{ind}$ is equivalent
to a boolean combination of formulas of the form $\exists y_1\in H\cdots \exists y_k\in H \varphi(\bar{x},\bar{y})$, where $\varphi(\bar{x},\bar{y})$ is an $\mathcal{L}$-formula. 

Furthermore, definable subsets of
$H(M)^n$ are just traces of $\mathcal{L}$-formulas in $H(M)^n$: if $(M, H(M))\models T^{ind}$ and $X\subset H(M)^n$ is $\mathcal{L}_H$-definable, there is an $\mathcal{L}$-formula $\varphi(\bar{x},\bar{y})$ 
and there is  $\bar c\in M$ such that $X=\varphi(H(M)^n,\bar c)$.
\end{lemma}

\begin{definition}\label{def-small} Let $(M,H(M))\models T^{ind}$ be sufficiently saturated.
For any subset $A\subset M$, the set
\[ \scl(A):= \acl(A\cup H(M))
\]
is called the \emph{small closure} of $A$. Any subset $B\subset \scl(A)$ is called \emph{$A$-small}. A
definable set $X$ is \emph{small} if there is a finite set $A$ such that $X$ is $A$-small. For any finite tuple
$\bar{b}$ and any set $A$, we write $\ldim(\bar{b}/A)=\dim(\bar{b}/AH(M))$, and it is called the \emph{large
dimension of the tuple $\bar{b}$ over $A$} (see \cite{BeVa2}). It corresponds to the dimension
of the tuple $\bar{b}$ with respect to $A$ in the pregeometry localized at $H(M)$.
\end{definition}

The following properties and definitions come from \cite{BeVa}, \cite{BeVa2}; and they follow easily from the definitions.

\begin{proposition}\label{basicprop}
Let $(M, H(M))\models T^{ind}$.
\begin{enumerate}
\item For any finite tuple $\bar{a}$, there exists some finite tuple $\bar{h}$ in $H(M)$ such that $\bar{a}\ind_{\bar{h}}H(M)$, and for such $\bar{h}$, $\bar{a}\bar{h}$ is $H$-independent.
\item If $\bar{a}$ is any $H$-independent tuple, then for any finite tuple $\bar{h}$ in $H(M)$, 
$\bar{a}\bar{h}$ is also $H$-independent. 
\item If $(M,H(M))$ is an $H$-structure, $\bar{a}$ is a tuple and $A\subset M$ is $H$-independent, then there is a \emph{minimal} $\bar{h}$ in $H(M)$ such that $\bar{a}\ind_{A\bar{h}}H(M)$. This tuple is unique up to permutation, we call it the \emph{$H$-basis of $\bar{a}$ over $A$} and it is denoted by $\operatorname{HB}(\bar{a}/A)$.

\item Assume $(M, H(M))\models T^{ind}$. Then for any tuple $\bar{b}$ and for any finite tuple $\bar{h}$ in 
$H(M)$, we have $\operatorname{HB}(\bar{b}\bar{h})=\operatorname{HB}(\bar{b})\bar{h}$.
\item For any tuple $\bar{a}$, $\acl_{\mathcal{L}_H}(\bar{a})=
\acl(\bar{a}\operatorname{HB}(\bar{a}))$.
\end{enumerate}
\end{proposition}

\bc \label{interalg-a1a2} Let $(M,H(M))\models T^{ind}$ and $\bar{a}=\bar{a}_1\bar{a}_2$ be  tuple such that $\bar{a}_1$ is independent over $H(M)$ and $\bar{a}_2\in \acl(\bar{a}_1H(M))$. Then, $\bar{a}$ is $\Le_H$-interalgebraic with $\bar{a}_1\operatorname{HB}(\bar{a})$. 
\ec
\bdem
By Proposition \ref{basicprop} (5), $\operatorname{HB}(\bar{a})\subseteq \acl_{\Le_H}(\bar{a})$, so $\bar{a}_1\operatorname{HB}(\bar{a})\subseteq \acl_{\Le_H}(\bar{a}_1\bar{a}_2)=\acl_{\Le_H}(\bar{a})$. Using Proposition \ref{basicprop} (3) we have   $\bar{a}\ind_{\operatorname{HB}(\bar{a})}H(M)$ and since $\bar{a}_2$ belongs to $\acl(\bar{a}_1H(M))$ we obtain $\bar{a}_2\in \acl(\bar{a}_1\operatorname{HB}(\bar{a}))$, as desired.
\edem

\bc \label{triviality-algclosure-H} Let $(M,H(M))\models T^{ind}$ and $\bar{d}$ be a tuple from $M$. If $h_1,\ldots,h_n$ are elements in $H(M)$ then $H\cap \acl_{\Le_H}(h_1,\ldots,h_n,\bar{d})=\{h_1,\ldots,h_n\}\cup \operatorname{HB}(\bar{d})$.
\ec
\bdem By part (5) of Proposition \ref{basicprop} we have the right to left inclusion. Assume now that $h\in H(M)\cap \acl_{\Le_H}(h_1,\ldots,h_n,\bar{d})$. Using part (5) of Proposition \ref{basicprop}, we have $h\in\acl(h_1,\ldots,h_n,\bar{d},\operatorname{HB}(h_1,\ldots,h_n,\bar{d}))$, which by part (4) is equal to $\acl(h_1,\ldots,h_n,\bar{d},\operatorname{HB}(\bar{d}))$.

By definition of the $H$-basis, we have $\bar{d}\ind_{\operatorname{HB}(\bar{d})}H(M)$, and using transitivity and symmetry we obtain $h\ind_{\operatorname{HB}(\bar{d})h_1,\ldots,h_n}\bar{d}$.  So, we must have $h\in \acl(h_1,\ldots,h_n,\operatorname{HB}(\bar{d}))$, and since the elements in $H(M)$ are all $\Le$-algebraically independent, we conclude that $h\in \{h_1,\ldots,h_n\}\cup \operatorname{HB}(\bar{d})$.
\edem


Note that by part (5) of Proposition \ref{basicprop}, whenever  $(M, H(M))\models T^{ind}$ and $A\subset M$, $\acl_{\mathcal{L}_H}(A)$ is always $H$-independent. So if we are given tuples
$\bar{a}$, $\bar{b}$, the set $\operatorname{HB}(\bar{b}/\acl_{\mathcal{L}_H}(\bar{a}))$ is well defined. This observation can be used to prove the following additivity property:

\begin{proposition}[Additivity of H-basis - see Proposition 3.1 in \cite{Carmona}]\label{additivityHB}
Let $(M, H(M))\models T^{ind}$ and let $\bar{a}$, $\bar{b}$ be tuples. Then 
$\operatorname{HB}(\bar{a}\bar{b})=\operatorname{HB}(\bar a)\cup \operatorname{HB}(\bar b/\acl_{\mathcal{L}_H}(\bar{a}))$.
\end{proposition}

In particular, if in the proposition above the tuple $\bar{a}$ is $H$-independent, we have
$\operatorname{HB}(\bar{a}\bar{b})=H(\bar a)\cup \operatorname{HB}(\bar b/\bar{a})$.
Sometimes we will abuse the notation and write $\operatorname{HB}(\bar b/\bar{a})$ for
$\operatorname{HB}(\bar b/\acl_{\mathcal{L}_H}(\bar{a}))$.

\begin{definition}\label{defnowheretrivial} Let $T$ be an $SU$-rank 1 theory and let $M\models T$ be $|T|^+$-saturated. Let $A \subset M$ with 
$|A|\leq |T|$ and let $p(x)\in S_1(A)$ be non-algebraic. We say $p(x)$ is \emph{non-trivial} if for some (all) realization 
$a\models p(x)$ there are finite subsets $\bar{b},\bar{c}$ of $M$ such that $a\ind_A \bar{b}$, $a\ind_A \bar{c}$ and
$a\in \acl(A,\bar{b},\bar{c})$. Otherwise, we say $p(x)$ is trivial. 

We say $T$ is \emph{nowhere trivial} if for all $A \subset M$ with 
$|A|\leq |T|$ and $p(x)\in S_1(A)$, if $p(x)$ is non-algebraic, then it is non-trivial. 
\end{definition}

\begin{observation}
 Let $T$ be an $SU$-rank 1 theory and let $M\models T$ be $|T|^+$-saturated. Let $A \subset M$ with $|A|\leq |T|$ and let $p(x)\in S_1(A)$ be non-algebraic. Usually the type $p(x)$ is called trivial if the pregeometry $p(M)$ is trivial. We will now prove that if  $p(x)$ is \emph{non-trivial}, we can choose $\bar b, \bar c$ to be tuples in $p(M)$ and thus the two approaches coincide.
 Indeed, assume $p(x)$ is non-trivial in the sense of Definition  \ref{defnowheretrivial} and choose $a,\bar b,\bar c$ witnessing the definition. 
 Let $(a_i,\bar b_i)_i$ be a Morley sequence in $\tp(a,\bar b/\bar c A)$ with 
 $(a,\bar b)=(a_0,\bar b_0)$. By supersimplicity, there is $N$ such that
 $a_0,\bar b_0\ind_{A(a_i,\bar b_i)_{0<i<N}} \bar c$. So we still have $a_0\ind_A \bar b_0$, $a_0\ind_A (a_i,\bar b_i)_{0<i<N}$ and
 $a_0\in \acl(A \bar b_0 (a_i,\bar b_i)_{0<i<N})$. Let $\bar a'=(a_i)_{0\leq i<N}$ and let $(\bar a'_j)_j$ be a Morley sequence in $\tp(\bar a'/ (\bar b_i)_{0\leq i<N} A)$ with 
 $\bar a'_0=\bar a'$.  Again by supersimplicity, there is $N'$ such that
 $\bar a'\ind_{A(\bar a_j')_{0<j<N'}} (\bar b_i)_{0\leq i<N}$.
 We now get $a_0\ind_A(a_i)_{0<i<N}$, $a_0\ind_A(\bar a_j')_{0<j<N'}$
and $a_0\in \acl(A (a_i)_{0<i<N}(\bar a_j')_{0<j<N'})$ and all tuples under consideration are realizations of $p(x)$.

\end{observation}

\begin{example}There are several examples of $SU$-rank 1 nowhere trivial theories, for example any completion of the theory of pseudofinite fields, or the theory of torsion-free divisible abelian groups. In both cases there is a group structure, and using group generics it is easy to build algebraic triangles that witness nowhere triviality. 

Now consider the language $\mathcal{L}=\{R_{+},R_\times,R_{RG},P\}$, where $R_{+},R_\times$ are ternary
relations, $R_{RG}$ is a binary relation and $P$ is an unary relation. Let $M$ be a model where
we interpret $P(M)$ as a pseudofinite field with $R_+,R_\times$ being the graphs of addition and multiplication, $\neg P(M)$ is a random graph with $R_{RG}$ being the edge relation, and no other relations hold. Then $Th(M)$ is not nowhere trivial, since in $\neg P(M)$ all types are trivial.
\end{example}

Note that in Definition \ref{defnowheretrivial} the subsets $\bar{b},\bar{c}$ can be taken so that each of them is  
$A$-independent and $\bar{b}\ind_A \bar{c}$. 
It is proved in \cite{BeVa2} that whenever 
$T$ is an $SU$-rank 1 theory, the theory $T^{ind}$ is supersimple of $SU$-rank at most $\omega$,
and it is equal to $\omega$ whenever $T$ has a non-trivial type. To avoid confusion, we will always use $\dim(\bar{a})$ to denote the $\Le$-algebraic dimension of a tuple $\bar{a}$ with respect to the theory $T$, and $\operatorname{SU}(\bar{a})$ or $\operatorname{SU}(p)$ for the $SU$-rank of the type of the tuple $\bar{a}$ with respect to $T^{ind}$.

We will revisit the arguments from 
\cite{BeVa2} and calculate explicitly the $SU$-rank of types when $T$ is nowhere trivial. We will need the
following two results from \cite{BeVa2}:

\begin{proposition}\label{rankHfinite} Let $T$ be an $SU$-rank 1 theory.
Let $(M, H(M))\models T^{ind}$, let $B\subset M$ be $H$-independent, let $\bar{a}$ be a finite tuple in $H(M)$. Let $k=\dim(\bar a/B)$, then $\operatorname{SU}(\tp_{\mathcal{L}_H}(\bar a/B))=k$.
\end{proposition}
\bdem We do this proof by induction on the length of the tuple $\bar{a}$. 
When $\bar{a}=a$ is a singleton, it follows from Corollary 5.7 in \cite{BeVa2}. Now assume the result holds for tuples in $H(M)$ of length less than or equal to $n$, and let $\bar{a}=(a_1,\ldots,a_n,a_{n+1})$ be an $(n+1)$-tuple from $H(M)$. 

Let us fix $\bar{c}=(a_1,\ldots,a_n)$. Note that as $B$ is $H$-independent and $a_{n+1}\in H(M)$,  the set $B\cup \{a_{n+1}\}$ is also $H$-independent. So,   $\operatorname{SU}(\bar{c}/Ba_{n+1})=\dim(\bar{c}/Ba_{n+1})$ by induction hypothesis. Since Lascar inequalities become equalities when the ranks involved are finite and the dimension coming from a pregeometry is always additive, we have
\begin{align*}
\operatorname{SU}(\bar{a}/B)&=\operatorname{SU}(\bar{c}/Ba_{n+1})+SU(a_{n+1}/B)\\
&=\dim(\bar{c}/Ba_{n+1})+\dim(a_{n+1}/B)\\
&=\dim(\bar{c},a_{n+1}/B)=\dim(\bar{a}/B). \qedhere
\end{align*}
\edem

\begin{proposition}[\cite{BeVa2}, Corollary 5.8] \label{rankinfinite}  Let $T$ be an $SU$-rank 1 theory.
Let $(M, H(M))\models T^{ind}$, let $B\subset M$ be $H$-independent, let $a$ be a singleton and assume that $a \not \in \acl(BH(M))$ and that $\tp_{\Le}(a/B)$ is non-trivial. Then $\operatorname{SU}(\tp_{\mathcal{L}_H}(a))=\omega$.
\end{proposition}

\begin{proposition}\label{rankomega} Let $T$ be an $SU$-rank 1 theory which is nowhere trivial.
Let $(M, H(M))\models T^{ind}$, let $B\subset M$ be $H$-independent and let $\bar{a}$ be tuple and 
write $\bar a=\bar a_1\bar a_2$ so that
$\bar a_1$ is independent over $B\cup H(M)$ and $\bar a_2\in \acl(\bar a_1H(M)B)$. Let 
$n=\dim(\bar{a}/B\cup H(M))=|\bar a_1|$ and let
$k=|\operatorname{HB}(\bar a/B)|$. Then $\operatorname{SU}(\tp_{\mathcal{L}_H}(\bar a/B))=\omega\cdot n+k$.
\end{proposition}

\begin{proof} We may assume $(M, H(M))\models T^{ind}$ is sufficiently saturated. We will prove the
argument for $B=\emptyset$, it is an easy exercise to see how the proof carries out to the general case.

By Corollary \ref{interalg-a1a2}, $\bar{a}$ is interalgebraic with $\bar{a}_1 \operatorname{HB}(\bar{a})$. Therefore, $\operatorname{SU}(\tp_{\Le_H}(\bar{a}))=\operatorname{SU}(\tp_{\mathcal{L}_H}(\bar{a}_1 \operatorname{HB}(\bar{a}))$. By the Lascar inequalities,
\begin{align*}
&\operatorname{SU}(\tp_{\Le_H}(\bar{a}_1/\operatorname{HB}(\bar{a})))+ \operatorname{SU}(\tp_{\Le_H}(\operatorname{HB}(\bar{a})))\leq \operatorname{SU}(\tp_{\mathcal{L}_H}(\bar{a}_1 \operatorname{HB}(\bar{a})))\\
&\leq \operatorname{SU}(\tp_{\Le_H}(\bar{a}_1/\operatorname{HB}(\bar{a})))\oplus \operatorname{SU}(\tp_{\Le_H}(\operatorname{HB}(\bar{a}))).
\end{align*}
Suppose $\bar{a}_1=(a_1,\ldots,a_n)$ and consider an index $i$ with $1\leq i\leq n$. Then the tuple $\operatorname{HB}(\bar{a})a_1,\ldots,a_{i-1}$ is $H$-independent and $\operatorname{SU}(\tp_{\Le_H}(a_i/\operatorname{HB}(\bar{a})a_1,\ldots,a_{i-1}))=\omega$ by Proposition \ref{rankinfinite}. 
Hence, by Lascar inequalities we  have $\operatorname{SU}(\tp_{\Le_H}(\bar{a}_1/\operatorname{HB}(\bar{a})))=\omega\cdot n$. On the other hand, by Proposition \ref{rankHfinite} we have $\operatorname{SU}(\tp_{\Le_H}(\operatorname{HB}(\bar{a})))=k$. Thus, $\omega\cdot n+k\leq \operatorname{SU}(\tp_{\mathcal{L}_H}(\bar{a}_1 \operatorname{HB}(\bar{a})))\leq \omega\cdot n+k$, and we get $\operatorname{SU}(\tp_{\mathcal{L}_H}(\bar a))=\omega\cdot n+k$.

\end{proof}

The result above can be generalized to arbitrary $SU$-rank 1 theories, but the full description is more elaborated. 

\begin{observation}\label{rankomegageneral} Let $T$ be an $SU$-rank 1 theory.
Let $(M, H(M))\models T^{ind}$, let $B\subset M$ be $H$-independent and let $\bar{a}$ be tuple and 
write $\bar a=\bar a_0 \bar a_1 \bar a_2$ so that:
\begin{enumerate} 
\item The tuple $\bar{a}_0$ is independent over $H(M)$ and for each element $c\in\bar{a}_0$, $\tp(c/B)$ is trivial.
\item The tuple $\bar{a}_1$ is independent over $H(M)$ and for each element $c\in\bar{a}_1$, $\tp(c/B)$ is non-trivial.
\item $\bar a_2\in \acl(\bar a_0\bar a_1H(M)B)$.
\end{enumerate} 
Let $m=|\bar a_0|$, $n=|\bar a_1|$, $k=|\operatorname{HB}(\bar a/B)|$. Then $\operatorname{SU}(\tp_{\mathcal{L}_H}(\bar a/B))=\omega\cdot n+k+m$.
\end{observation}

In this paper we will deal with $SU$-rank 1 theories which are nowhere trivial, where Proposition \ref{rankomega} holds. The following result shows that the $SU$-rank of $T^{ind}$ is Cantor-additive when $T$ is nowhere trivial of $SU$-rank 1.\footnote{With the same proof and using Observation \ref{rankomegageneral}, one could also show that in the general case, the $SU$-rank of $T^{ind}$ is also Cantor-additive whenever $T$ has $SU$-rank 1.}

\begin{lemma}\label{lem-addSU} Let $T$ be an $SU$-rank 1 theory which is nowhere trivial.
Let $(M, H(M))\models T^{ind}$, let $C\subset M$ be $H$-independent and let $\bar{a}$
and $\bar{b}$ be tuples. Then $\operatorname{SU}(\tp_{\mathcal{L}_H}(\bar a\bar b/C))=\operatorname{SU}(\tp_{\mathcal{L}_H}(\bar a/C))\oplus \operatorname{SU}(\tp_{\mathcal{L}_H}(\bar b/C\bar a))$.
\end{lemma}

\begin{proof}
Write $\bar a=\bar a_1\bar a_2$ so that
$\bar a_1$ is independent over $H(M)C$ and $\bar a_2\in \acl(\bar a_1H(M)C)$. Let $n_1=|\bar a_1|$ and let
$k_1=|\operatorname{HB}(\bar a/C)|$, then by Proposition \ref{rankomega}, 
$\operatorname{SU}(\tp_{\mathcal{L}_H}(\bar a/C))=\omega\cdot n_1+k_1$. Similarly write $\bar b=\bar b_1\bar b_2$ so that
$\bar b_1$ is independent over $\bar a H(M)C$ and $\bar b_2\in \acl(\bar{b}_1\bar a H(M)C)$. Let $n_2=|\bar b_1|$ and let
$k_2=|\operatorname{HB}(\bar b/\bar a C)|$. Since $\acl_{\mathcal{L}_H}(\bar a C)$ is $H$-independent, we can apply Proposition \ref{rankomega} and get that $\operatorname{SU}(\tp_{\mathcal{L}_H}(\bar b/C\bar a))=\operatorname{SU}(\tp_{\mathcal{L}_H}(\bar b/\acl_{\mathcal{L}_H}(C\bar a)))=\omega\cdot n_2+k_2$. Then, $\operatorname{SU}(\tp_{\mathcal{L}_H}(\bar a/C))\oplus \operatorname{SU}(\tp_{\mathcal{L}_H}(\bar b/C\bar a))=\omega\cdot (n_1+n_2)+(k_1+k_2)$.

Now consider the tuple $\bar a\bar b$. The subtuple $\bar a_1 \bar b_1$ is independent over $CH(M)$
and $\bar a_2\bar b_2 \in \acl(\bar a_1\bar b_1H(M)C)$. By Proposition \ref{additivityHB}, we have that 
$\operatorname{HB}(\bar a \bar b/C)=\operatorname{HB}(\bar a/C)\cup \operatorname{HB}(\bar b/\bar a C)$, which is a disjoint union, so 
$|\operatorname{HB}(\bar a \bar b/C)|=k_1+k_2$. Again using Proposition \ref{rankomega} we get $\operatorname{SU}(\tp_{\mathcal{L}_H}(\bar a\bar b/C))=\omega\cdot(n_1+n_2)+(k_1+k_2)$ as we wanted.
\end{proof}

\begin{remark}
Cantor-additivity of $SU$-rank seems not to be an usual property of supersimple theories of infinite $SU$-rank. It does not hold for ACFA (see 4.17 in \cite{Ch}): if $a$ is transformally trascendental and $b_1=\sigma(a)-a$, then
$\operatorname{SU}(a,b_1)=\omega$, $\operatorname{SU}(a/b_1)=1$ and $\operatorname{SU}(b_1)=\omega$. A similar example can be build for $T=DCF$ taking $a$ to be 
generic and $b_1=\delta(a)-a$. Are there structural consequences of supersimple theories 
of infinite $SU$-rank that have Cantor-additivity?
\end{remark}

\section{Existential $H$-formulas and definability of $SU$-rank}\label{section:dimension}

Let $T$ be a nowhere trivial $SU$-rank 1 theory (which in particular is geometric) and let $(M,H)$ be a sufficiently saturated model of $T^{ind}$.\footnote{To simplify notation, we will write $(M,H)$ instead of $(M,H(M))$ from this section onwards.}  In this section we deal with $\mathcal{L}_H$-formulas of the form $\exists \bar{z}\in H^{|\bar{z}|}~\varphi(\bar{x},\bar{z},\bar{y})$, where $\varphi(\bar{x},\bar{z},\bar{y})$ is an $\mathcal{L}$-formula. We will call such formulas \emph{existential $H$-formulas}, and we will see that these formulas are fundamental in the analysis of ranks for definable sets in $H$-structures: they witness the continuity of the $SU$-rank, and any other 
$\mathcal{L}_H$-formula can be approximated uniformly by existential $H$-formulas, up to smaller $SU$-rank. We will use them to show that in $T^{ind}$ the $\operatorname{SU}$-rank of a formula is definable in terms of its parameters.

\begin{notation}
For a tuple $\bar{z}=(z_1,\ldots,z_n)$ and $1\leq i\leq n$, we will write $\hat{z}_i$ to denote the tuple $(z_1,\ldots,z_{i-1},z_{i+1},\ldots,z_n)$, where the coordinate $z_i$ is omitted. Given a partition $\bar{x}=(\bar{x}_1,\bar{x}_2)$ of $\bar{x}$ into two subtuples, we can consider the formula $\varphi_{\bar{x}_1}(\bar{x}_1;\bar{x}_2):=\varphi(\bar{x})$ which is the result of reorganizing variables of $\varphi$. In the particular case that $\bar{x}_1=x_j$ is a single variable, we simply write $\varphi_j(x_j;\widehat{x}_j)$ instead of $\varphi_{x_j}(x_j;\widehat{x}_j)$.\\

\end{notation}

\begin{definition}\label{Halg}
Let $T$ be a geometric theory in the language $\Le$ and let $\varphi(\bar{x})$ be an $\Le$-formula. 
By elimination of $\exists^\infty$, for every formula $\varphi(\bar{x})$ and every partition $\bar{x}=(\bar{x}_1,\bar{x}_2)$, there is a formula $\operatorname{Alg}_{\varphi,\bar{x}_1}(\bar{x}_2)$ such that for every $M\models T$ and every $\bar{b}\in M^{\bar{x}_2}$, $M\models \operatorname{Alg}_{\varphi.\bar{x}_1}(\bar{b})$ if and only if the definable set $\varphi_{\bar{x}_1}(M^{|\bar{x}_1|};\bar{b})$ is finite. Again, in the particular case when $\bar{x}_1=x_j$ is a single variable, we will denote this formula by $\operatorname{Alg}_{\varphi,j}(\widehat{x}_j)$.\\

Similarly, since in a geometric theory dimension is definable, whenever $\varphi(\bar{x})$ is an $\Le$-formula, $\bar{x}=(\bar{x}_1,\bar{x}_2)$ and $k\leq |\bar{x}|$, there is an $\Le$-formula denoted by $\operatorname{Dim}_{\varphi,\bar{x}_1,k}(\bar{x}_2)$ such that for every $M\models T$, $M\models \operatorname{Dim}_{\varphi,\bar{x}_1,k}(\bar{b})$ if and only $\operatorname{dim}(\varphi(M^{|\bar{x}_1|};\bar{b}))=k$. When $\bar{x}_1=x_j$ is a single variable, we denote this formula by $\operatorname{Dim}_{\varphi,j,k}(\widehat{x}_j)$. \\

Finally, for an $\Le$-formula $\varphi(\bar{x},\bar{z})$, let us define the $\Le$-formula
\[\varphi_{\bar{z},H}(\bar{x},\bar{z}):=\varphi(\bar{x},\bar{z})\wedge \bigwedge_{i\neq j}z_i\neq z_j\bigwedge_{0\leq j<|\bar{z}|}\operatorname{Alg}_{\varphi,j}(\bar{x},\hat{z}_j).\]
\end{definition}
\begin{lemma}\label{lemalgH}
Let $(M,H)$ be an $H$-structure, $\varphi(\bar{x},\bar{z})$ an $\Le$-formula and $\varphi_{\bar{z},H}(\bar{x},\bar{z})$ be defined as in Definition \ref{Halg}.
Then $\varphi_{\bar{z},H}(\bar{a},H^{|\bar{z}|})$ is finite for any $\bar{a}\in M^{|\bar{x}|}$.
\end{lemma}
\begin{proof}
Suppose $n=|\bar{z}|$. Let $\bar{a}\in M^{|\bar{x}|}$ and assume $\varphi_{\bar{z},H}(\bar{a},h_1,\ldots,h_{n})$ holds for some $h_1,\ldots,h_n\in H$. Then, the elements $h_1,\ldots,h_n$ are all distinct, and for each $i\leq n$, $\operatorname{Alg}_{\varphi,i}(\bar{a},\widehat{h}_i)$ holds, so $h_i\in\acl(h_1,\ldots,h_{i-1},h_{i+1},\ldots,h_n,\bar{a})$. Thus, by Corollary \ref{triviality-algclosure-H} we have $h_i\in \operatorname{HB}(\bar{a})$ for every $1\leq i\leq n$, and thus $\varphi_{\bar{z},H}(\bar{a},H^n)$ has size at most $|\operatorname{HB}(\bar{a})|^n$, which is finite.
\end{proof}

In the following, we will use existential $H$-formulas to define $\operatorname{HB}(\bar{a})$ for a given tuple $\bar{a}$. Since $\operatorname{HB}(\bar{a})$ is a set and not a tuple, we may need to consider all permutations of any enumeration of $\operatorname{HB}(\bar{a})$.

\begin{notation} For $k\geq 1$, we write $\mathbb{S}_k$ for the permutation group on $k$ elements, and whenever $\bar{z}=(z_1,z_2,\ldots,z_k)$ is a tuple of variables and $\sigma\in\mathbb{S}_k$, we let $\sigma(\bar{z})$ denote the tuple $(z_{\sigma(1)},z_{\sigma(2)},\ldots, z_{\sigma(k)})$. 

For a formula $\varphi(\bar{x},\bar{z})$ in a complete theory $T$, we say that $\varphi$ is \emph{invariant under permutations of $\bar{z}$} if \[T\models \forall\bar{x}\forall \bar{z}\bigwedge_{\sigma\in\mathbb{S}_{|\bar{z}|}}\Big{(}\varphi(\bar{x},\bar{z})\leftrightarrow \varphi(\bar{x},\sigma(\bar{z}))\Big{)}.\]
\end{notation}

\bd\label{mainformula} 
Let $(M,H)$ be a sufficiently saturated model of $T^{ind}$. Let $\bar{c}\in M$ be $H$-independent
and let $p(\bar{x})$ be a complete $\mathcal{L}_H$-type of $SU$-rank $\omega\cdot r+k$ over $\bar{c}$. Our goal is to define an existential $H$-formula $\exists\bar{z}\in H^k\Psi_p(\bar{x},\bar{z},\bar{c})$ which witnesses the $SU$-rank of $p$ over $\bar{c}$.

Let $\bar{a}\models p(\bar{x})$. By Proposition \ref{rankomega}, $\dim(\bar{a}/H\bar{c})=r$ and $|\operatorname{HB}(\bar{a}/\bar{c})|=k$.
Choose $\bar{a}_1=(a_{i_1},\ldots,a_{i_r})$ which is maximal independent subtuple of $\bar{a}$ over $H\bar{c}$ and let $\bar{a}_2=(a_{j_1},\ldots,a_{j_t})$ be the rest of the tuple. Also let $\bar h=(h_1,\dots,h_{k})$ be an enumeration of $\operatorname{HB}(\bar{a}/\bar{c})$. Then $\bar{a}_2\in \acl(\bar{a}_1,\bar h,\bar{c})$ (as $\bar{c}$ is $H$-independent and $\bar{a}\ind_{\bar{h},
\bar{c}} H$ by definition of $\operatorname{HB}(\bar{a}/\bar{c})$), so we can fix an $\mathcal{L}$-formula $\psi_{p}(\bar{x}_{2p},\bar{x}_{1p},\bar{z},\bar{y})$ with $\bar{x}_{1p}=(x_{i_1},\ldots,x_{i_r})$ and $\bar{x}_{2p}=(x_{j_1},\ldots,x_{j_t})$ 
such that \[M\models \psi_{p}(\bar{a}_{2},\bar{a}_{1},\bar h,\bar{c})\wedge \operatorname{Alg}_{\psi_p,\bar{x}_{2p}}(\bar{a}_1,\bar{h},\bar{c}).\] 

We can now consider the formula

\[\Psi_p(\bar{x},\bar{z},\bar{y}):=\bigvee_{\sigma\in\mathbb{S}_k}\left(\psi_{p}(\bar{x}_{2p},\bar{x}_{1p},\sigma(\bar{z}),\bar{y}))\land \operatorname{Alg}_{\psi_p,\bar{x}_{2p}}(\bar{x}_{1p},\sigma(\bar{z}),\bar{y}) \wedge \operatorname{Dim}_{\varphi,(\bar{x},\bar{z}),r+k}(\bar{y})\right)\]

Note that by construction, $\Psi_p$ only depends on the type $p$ and not on the choice of $\bar{a}\models p$ in the sense that whenever $\bar{a}'\models p$ and $\bar{h}'$ is any enumeration of the $\operatorname{HB}(\bar{a}'/\bar{c})$, we will have $(M,H)\models \Psi_p(\bar{a}',\bar{h}',\bar{c})$.
\ed

\begin{lemma}\label{unique-HB}
Let $(M,H)$ be an $H$-structure and let $\bar{c}$ be an $H$-independent tuple. Let $p(\bar{x})$ be a complete type over $\bar{c}$ of $SU$-rank $\omega\cdot n+k$. Let $\Psi_p$ be as described in Definition \ref{mainformula}. For any $\bar{a}$ and $\bar{c}'$ with 
$\operatorname{SU}(\bar{a}/\bar{c}')=\omega\cdot n+k$ and any $\bar{h}\in H^k$, if $(M,H)\models\Psi_p(\bar{a},\bar{h},\bar{c}')$, then $\bar{h}$ is an enumeration of $\operatorname{HB}(\bar{a}/\bar{c}',\operatorname{HB}(\bar{c}'))$. 

Moreover, if $(M,H)\models \exists \bar{z}\in H^k\Psi_p(\bar{d},\bar{z},\bar{c}')$ for some $\bar{d},\bar{c}'$, then $\operatorname{SU}(\bar{d}/\bar{c}')\leq \omega\cdot n+k$.
\end{lemma}

\begin{proof} Suppose $(M,H)\models \Psi_p(\bar{a},\bar{h},\bar{c}')$. Let $\psi_p(\bar{x}_{2p},\bar{x}_{1p},\bar{z},\bar{y})$ be the subformula of $\Psi_p$ as given by Definition \ref{mainformula} and let $\bar{a}_1,\bar{a}_2$ be subtuples of $\bar{a}$ partitioned according to $\bar{x}_{1p},\bar{x}_{2p}$. Then $|\bar{a}_1|=n$ and there is $\sigma\in\mathbb{S}_k$ such that $(M,H)\models \psi_p(\bar{a}_{2},\bar{a}_{1},\sigma(\bar{h}),\bar{c}')\land \Alg_{\psi_p,\bar{x}_{2p}}(\bar{a}_1,\sigma(h),\bar{c}')$ and hence $\bar{a}_{2}\in\acl(\bar{a}_{1},\sigma(\bar{h}),\bar{c}')$. Since $\operatorname{SU}(\bar{a}/\bar{c}')\geq \omega\cdot n$, we must have $\dim(\bar{a}_{1}/H,\bar{c}')=n$.
Therefore, $\bar{a}\ind_{\sigma(\bar{h}),\bar{c}'}H$ and $\bar{a}\ind_{\sigma(\bar{h}),\bar{c}',\operatorname{HB}(\bar{c}')}H$. Since $|\operatorname{HB}(\bar{a}/\bar{c}',\operatorname{HB}(\bar{c}'))|=k$ and $|\sigma(\bar{h})|\leq k$, we must have that $\sigma(\bar{h})$ is minimal. Hence, $\sigma(\bar{h})$ is an enumeration of $\operatorname{HB}(\bar{a}/\bar{c}',\operatorname{HB}(\bar{c}'))$, and $\bar{h}$ is another enumeration.

For the moreover part, suppose $(M,H)\models \Psi_p(\bar{d},\bar{h}',\bar{c}')$ for some $\bar{h}'$ in $H$. Note that $\dim(\bar{d},\bar{h}'/\bar{c}')\leq n+k$ and $\dim(\bar{d}/\bar{h}',\bar{c}')\leq n$ by the construction of $\Psi_p$. If $\dim(\bar{d}/H,\bar{c}')<n$, then $\operatorname{SU}(\bar{d}/\bar{c}')<\omega\cdot n$ and we are done. Otherwise $\dim(\bar{d}/H,\bar{c}')=n$, then $\dim(\bar{d}/H,\bar{c}')=\dim(\bar{d}/\bar{h},\bar{c}')$ and $\operatorname{HB}(\bar{d}/\operatorname{HB}(\bar{c}'),\bar{c}')\subseteq\{\bar{h}\}$. Hence $\operatorname{SU}(\bar{d}/\bar{c}')=\operatorname{SU}(\bar{d}/\bar{c}',\operatorname{HB}(\bar{c}'))\leq \omega\cdot n+k$.
\end{proof}

\begin{lemma}\label{lem-HB-base}
Let $(M,H)$ be an $H$-structure. Suppose $\bar{a}$ is a tuple with $\operatorname{SU}(\bar{a})=\omega\cdot n+k$, and let $p(\bar{x})=\tp_{\Le_H}(\bar{a})$. Then there is an $\mathcal{L}$-formula $\xi_{p}(\bar{x},\bar{z})$ such that the following hold:
\benum
\item If $\bar{h}$ is an enumeration of $\operatorname{HB}(\bar{a})$, then $(M,H)\models \xi_{p}(\bar{a},\bar{h})$.
\item For any $\bar{a}'$ with $\operatorname{SU}(\bar{a}')=\omega\cdot n+k$ and any $\bar{h}'\in H^k$, if $(M,H)\models \xi_{p}(\bar{a}',\bar{h}')$ then $\bar{h}'$ is an enumeration of $\operatorname{HB}(\bar{a}')$.
\item $\xi_{p}(\bar{x},\bar{z})$ is invariant under permutations of $\bar{z}$.
\item For all $\bar{d}\in M^{|\bar{a}|}$, the set $\xi_{p}(\bar{d},H^k)$ is finite.
\eenum
\end{lemma}
\begin{proof}
Let $p=p(\bar{x}):=\tp_{\Le_H}(\bar{a})$ and let $\Psi_p(\bar{x},\bar{z})$ be given as in Definition \ref{mainformula}. Let $\bar{h}=(h_1,\ldots,h_k)$ be any enumeration of $\operatorname{HB}(\bar{a})$, then $(M,H)\models \Psi_{p}(\bar{a},\bar{h})$.

Recall that $\hat{z}_j$ is a tuple of variables obtained by omitting $z_j$ from $\bar{z}$. Recall that there is an $\Le$-formula $\Alg_{\Psi_p,z_j}(\bar{x},\hat{z}_j)$ stating that $\Psi_p$ is an algebraic formula in the variable $z_j$.

\textbf{Claim:} \emph{For every $j=1,\ldots,k$, $\Alg_{\Psi_p,z_j}(\bar{a},h_1,\ldots,h_{j-1},h_{j+1},\ldots,h_{k})$ holds.}

Otherwise, $\Psi_p(\bar{a},h_1,\ldots,h_{j-1},M,h_{j+1},\ldots,h_{k})$ is infinite, and by the density property we have that
$\Psi_p(\bar{a},h_1,\ldots,h_{j-1},H,h_{j+1},\ldots,h_{k})$ is also infinite. However, $\Psi_p(\bar{a},H^k)$ is a set of enumerations of $\operatorname{HB}(\bar{a})$ by Lemma \ref{unique-HB}, a contradiction.

Let \[\xi_{p}(\bar{x},\bar{z}):=(\Psi_p)_{\bar{z},H}(\bar{x},\bar{z})=\Psi_p(\bar{x},\bar{z})\land\bigwedge_{1\leq j\leq k}z_i\neq z_j\bigwedge_{1\leq j\leq k}\Alg_{\Psi_p,z_j}(\bar{x},\hat{z}_j)\] be defined as in Definition \ref{Halg}. Since $\Psi_p$ is invariant under permutations of $\bar{z}$, so is $\xi_{p}$, and condition (3) is satisfied. Moreover, $(M,H)\models \bigwedge_{\sigma\in\mathbb{S}_k}\xi_{p}(\bar{a},\sigma(\bar{h}))$ by the Claim above and by condition (3). Hence condition (1) also holds.\\

Condition (2) follows from Lemma \ref{unique-HB} applied to the subformula $\Psi_p$ in $\xi_{p}$ and condition (4) follows from Lemma \ref{lemalgH}.

\end{proof}

\begin{remark}{\label{remSubHB}} Combining Corollary \ref{triviality-algclosure-H} with part (4) of the previous lemma, we know that for any $\xi_{p}$ defined as in Lemma \ref{lem-HB-base}, we have that for any tuple $\bar{d}$, $\xi_{p}(H^n,\bar{d})$ is a collection of different enumerations of subsets of $\operatorname{HB}(\bar{d})$, although the containment may be strict as the following example shows.
\end{remark}
\begin{example} \label{example-H-vs}

Let $T$ be the theory of infinite vector spaces over $\mathbb{F}_p$, which is strongly minimal and nowhere trivial. Let $(M,H)$ be a saturated model of $T^{ind}$ and choose $c_1$ independent from $H$ and $c_2=c_1+h$ for some $h\in H$. Let $p(x_1,x_2)=\tp_{\Le_H}(c_1,c_2)$. In this case, $\xi_p(z;x_1,x_2)\equiv x_2=x_1+z$ and $\xi_p(H,c_1,c_2):=\{h\}$. However, if $d_1=h_1$ is another element from $H$ and $d_2=h_1+h$ we would have $\xi_p(H,d_1,d_2)=\{h\}\subsetneq \{h_1,h\}=\operatorname{HB}(d_1,d_2)$.
\end{example}

Recall that the $SU$-rank of a formula $\varphi(\bar{x})$, which is defined over a set of parameters $\bar{c}$, 
is given by: $$\operatorname{SU}(\varphi):=\sup\{\operatorname{SU}(p): \varphi\in p,p\in S_{T^{ind}}(\bar{c})\}.$$

We will now prove that this supremum can be attained as a maximum in the theory $T^{ind}$ using the formula $\Psi_p$ as defined before.

\begin{lemma}\label{lem-continuitySU}
Let $(M,H)$ be a sufficiently saturated $H$-structure.
Let $X\subset M^n$ be $\Le_H$-definable over $\bar{c}$ where $\bar{c}$ is
$H$-independent. Then there is a type $p\in S_{T^{ind}}(\bar{c})$ such that $\operatorname{SU}(X/\bar{c})=\operatorname{SU}(p)$. That is, 
 \[\operatorname{SU}(X/\bar{c})=\max\{\operatorname{SU}(p): X\in p,p\in S_{T^{ind}}(\bar{c})\}.\]
\end{lemma}
\begin{proof}
Let $\varphi(\bar{x},\bar{c})$ be an $\mathcal{L}_H$-formula defining $X$ and let $p(\bar{x})$ be a complete $\mathcal{L}_H$-type over $\bar{c}$ extending $\varphi(\bar{x},\bar{c})$. Let $\Psi_p(\bar{x},\bar{z}_p,\bar{c})$ be the $\Le$-formula defined as in Definition \ref{mainformula}, then $\Theta_p(\bar{x}):=\exists \bar{z}_p\in H^{|\bar{z}_p|}\Psi_p(\bar{x},\bar{z}_p,\bar{c})\in p(\bar{x})$. Hence $\operatorname{SU}(\Theta_p)\geq \operatorname{SU}(p)$. On the other hand,  by Lemma \ref{unique-HB}, for any $\bar b$ satisfying $\Theta_{p}(\bar{x})$, we have $\operatorname{SU}(\bar b/\bar{c})\leq \operatorname{SU}(p)$. Therefore, $\operatorname{SU}(\Theta_p)=\operatorname{SU}(p)$.

Notice that $\Theta_p(\bar{x})$ is an existential $H$-formula and $\{\Theta_p(\bar{x}):p\in S_{T^{ind}}(\bar{c}),X\in p\}$ is an open covering of $\varphi(\bar{x},\bar{c})$. By compactness, there are finitely many types $p_1,\ldots,p_N$ containing the formula $\varphi(\bar{x},\bar{c})$ and formulas $\Theta_{p_1}(\bar{x}),\ldots,\Theta_{p_N}(\bar{x})$ such that $\varphi(\bar{x},\bar{c})\models \bigvee_{i\leq N}\Theta_{p_i}(\bar{x})$. Hence, 
$$\operatorname{SU}(X)=\operatorname{SU}(\varphi(\bar{x},\bar{c}))\leq \operatorname{SU}\left(\bigvee_{i\leq N}\Theta_{p_i}(\bar{x})\right)=\max\{\operatorname{SU}(p_i):i\leq N\},$$ and since by definition we have $\varphi(\bar{x},\bar{c})\in p_i$ for each $i\leq N$, we also have $$\operatorname{SU}(X)=\operatorname{SU}(\varphi(\bar{x},\bar{c})):=\sup\{\operatorname{SU}(p): \varphi(\bar{x},{c})\in p,p\in S_{T^{ind}}(\bar{c})\}\geq \max\{\operatorname{SU}(p_i):i\leq N\}.$$ Therefore, $\displaystyle{\operatorname{SU}(\varphi(\bar{x},\bar{c}))=\max\{\operatorname{SU}(p): \varphi(\bar{x},\bar{c})\in p,p\in S_{T^{ind}}(\bar{c})\}= \operatorname{SU}\left(\bigvee_{i\leq N}\Theta_{p_i}(\bar{x})\right).}$ 
\end{proof}

\begin{lemma}\label{lem-almostExistential}
Let $(M,H)$ be a sufficiently saturated $H$-structure and let $\bar{c}$ be a tuple. Let $\bar{h}$ be an enumeration of $\operatorname{HB}(\bar{c})$. Let $X\subset M^m$ be $\mathcal{L}_H$-definable over $\bar{c}$ and $\operatorname{SU}(X/\bar{h},\bar{c})=\omega\cdot n+k$. Then there is
$Y\subset M^m$ defined by an existential $H$-formula $\exists\bar{z}\in H^k\eta(\bar{x},\bar{z},\bar{h},\bar{c})$, such that 
\begin{enumerate}
\item $\eta(\bar{x},\bar{z},\bar{t},\bar{y})$ is an $\mathcal{L}$-formula invariant under permutations of $\bar{z}$ and permutations of $\bar{t}$. And for any $\bar{a}',\bar{h}',\bar{c}'$, the set $\eta(\bar{a}',H^k,\bar{h}',\bar{c}')$ is finite.
\item There are finitely many complete types $p_1,\ldots,p_t$ of $SU$-rank $\omega\cdot n+k$ such that 
$$M\models \forall\bar{x}\forall\bar{z}\left( \eta(\bar{x} ,\bar{z},\bar{h},\bar{c})\to \bigvee_{i\leq t ,\tau\in\mathbb{S}_{|\bar{h}|}}\Psi_{p_i}(\bar{x},\bar{z},\tau(\bar{h}),\bar{c})\right),$$
where $\Psi_{p_i}(\bar{x},\bar{z},\bar{h},\bar{c})$ is defined as in Definition \ref{mainformula}.
\item
$\operatorname{SU}(Y\triangle X)<\omega\cdot n+k$.
\end{enumerate}
\end{lemma}
\brmk
It is important to make the formula $\eta$ invariant under permutations of $\bar{z}$. The reason is that in the next section, we will use $Y$ to define the measure for $X$ and $Y$ is the projection of the set  $\{(\bar{a},\bar{h}_a):\bar{h}_a\in H^{k},M\models\eta(\bar{a},\bar{h}_a,\bar{h},\bar{c})\}$, where each of the fibers is finite. When $\eta$ is invariant under permutations of $\bar{z}$, the size of the fibers is generically uniform. More precisely, we will see later that whenever $\bar{a}\in Y$ has maximal $SU$-rank and $\bar{h}_a$ is an enumeration of $\operatorname{HB}(\bar{a}/\bar{c},\bar{h})$, the fiber over $\bar{a}$ is of the form $\{(\bar{a},\sigma(\bar{h}_a)):\sigma\in\mathbb{S}_k\}$, hence the fiber has size $|\mathbb{S}_{k}|$ uniformly. The reason for making $\eta$ invariant under permutations of $\bar{t}$ is similar, and it will be essential for proving the Fubini condition for measures (see Theorem \ref{th-Fubini}).
\ermk

\begin{proof} Let us consider the set $E:=\{p\in S_{T^{ind}}(\bar{h},\bar{c}):X\in p,\operatorname{SU}(p)= \omega\cdot n+k\}$, which is non-empty by Lemma \ref{lem-continuitySU}. For every $p\in E$, we can consider the open set $\Theta_p(\bar{x}):=\exists\bar{z}\in H^k\Psi_p(\bar{x},\bar{z},\bar{h},\bar{c})$ where $\Psi_p$ is defined in Definition \ref{mainformula}, which contains $p$. Notice also that $E=\{p\in S_{T^{ind}}(\bar{h},\bar{c}):X\in p,\operatorname{SU}(p)\geq  \omega\cdot n+k\}$, so it is closed, and by compactness there are finitely many types $p_1,\ldots,p_N$ in $E$ such that $E$ is covered by $\Theta_{p_1}(\bar{x})\vee \cdots \vee \Theta_{p_N}(\bar{x})$. Note that 
$\operatorname{SU}(p_i)= \omega\cdot n+k$ for all $i\leq N$. 

Let $\chi_0(\bar{x},\bar{z},\bar{t},\bar{y}):=\bigvee_{i\leq N,\tau\in\mathbb{S}_{|\bar{h}|}}\Psi_{p_i}(\bar{x},\bar{z},\tau(\bar{t}),\bar{y}).$ Then $\chi_0$ is invariant under permutations of $\bar{z}$ and of $\bar{t}$ and the existential $H$-formula $\exists\bar{z}\in H^k\chi_0(\bar{x},\bar{z},\bar{h},\bar{c})$ covers $\Theta_{p_1}(\bar{x})\vee \cdots \vee \Theta_{p_N}(\bar{x})$, hence also covers $E$. Let $Y_0$ be the set defined by $\exists\bar{z}\in H^k\chi_0(\bar{x},\bar{z},\bar{h},\bar{c})$.
Then $\operatorname{SU}(Y_0)= \omega\cdot n+k$. Since all types $p$ with $X\in p$ and $\operatorname{SU}(p)= \omega\cdot n+k$ belong to $Y_0$, we get that $\operatorname{SU}(X\setminus Y_0)<\omega\cdot n+k$. However it maybe the case that $\operatorname{SU}(Y_0\setminus X)= \omega\cdot n+k$.

\textbf{Claim:} \emph{For any $q\ni Y_0\setminus X$ with $\operatorname{SU}(q)=\omega\cdot n+k$ there is
an $\Le$-formula $\chi_q(\bar{x},\bar{z},\bar{t},\bar{y})$ such that:
\benum
\item[(a)]
$\chi_q(\bar{x},\bar{z},\bar{t},\bar{y})$ is invariant under permutations of $\bar{z}$ and of $\bar{t}$.
\item[(b)] The set $Z_q$ defined as $\exists \bar{z}\in H^k(\chi_q(\bar{x},\bar{z},\bar{h},\bar{c}))$ extends $E$ and does not contain $q$.
\item[(c)] There are finitely many types $p_1,\ldots,p_{m_q}$ (which depend on $q$) contained in $E$ such that
\[M\models \forall \bar{x}\,\forall \bar{z}\left(\chi_q(\bar{x},\bar{z},\bar{h},\bar{c})\rightarrow\bigvee_{i\leq m_q,\tau\in\mathbb{S}_{|\bar{h}|}}\Psi_{p_i}(\bar{x},\bar{z},\tau(\bar{h}),\bar{c})\right).\]
\eenum}

    First, let us fix a type $p\in X$ with $\operatorname{SU}(p)= \omega\cdot n+k$ and choose realizations $\bar{a}\models p$ and $\bar{b}\models q$. Let  $\bar{h}_a,\bar{h}_b$ be enumerations of $\operatorname{HB}(\bar{a}/\bar{h},\bar{c}),\operatorname{HB}(\bar{b}/\bar{h},\bar{c})$ respectively. Note that both tuples $\bar{h}_a,\bar{h}_b$ have length $k$. Since $(\bar{h},\bar{c})$ is $H$-independent, by additivity of HB we have that both $(\bar{a},\bar{h}_a,\bar{h},\bar{c})$ and $(\bar{b},\bar{h}_b,\bar{h},\bar{c})$ are H-independent. Since $p=\tp_{\Le_H}(\bar{a}/\bar{h},\bar{c})\neq \tp_{\Le_H}(\bar{b}/\bar{h},\bar{c})=q$, by Lemma \ref{fundlem1} 
we get that $\tp_{\Le}(\bar{a},\bar{h}_a,\bar{h},\bar{c})\neq \tp_{\Le}(\bar{b},\bar{h}_b,\bar{h},\bar{c})$. Similarly, for any $\sigma\in\mathbb{S}_k$ and $\tau\in\mathbb{S}_{|\bar{h}|}$ , $\tp_{\Le}(\bar{a},\bar{h}_a,\bar{h},\bar{c})\neq \tp_{\Le}(\bar{b},\sigma(\bar{h}_b),\tau(\bar{h}),\bar{c})$, so there is an $\Le$-formula $\xi_{pq}^{\sigma\tau}(\bar{x},\bar{z},\bar{t},\bar{c})\in \tp_{\Le}(\bar{a},\bar{h}_a,\bar{h},\bar{c})\setminus \tp_{\Le}(\bar{b},\sigma(\bar{h}_b),\tau(\bar{h}),\bar{c})$. By taking conjunctions over all $\sigma,\tau$, we can find an $\Le$-formula $\xi_{pq}(\bar{x},\bar{z},\bar{t},\bar{c})$ such that for all $\sigma\in\mathbb{S}_k, \tau\in\mathbb{S}_{|\bar{h}|}$, $\xi_{pq}(\bar{x},\bar{z},\bar{t},\bar{c})\in \tp_{\Le}(\bar{a},\bar{h}_a,\bar{h},\bar{c})\setminus \tp_{\Le}(\bar{b},\sigma(\bar{h}_b),\tau(\bar{h}),\bar{c})$.\\

Let $\eta_{pq}(\bar{x},\bar{z},\bar{t},\bar{y}):=\bigvee_{\sigma\in\mathbb{S}_k,\tau\in\mathbb{S}_{|\bar{h}|}}\xi_{pq}(\bar{x},\sigma(\bar{z}),\tau(\bar{t}),\bar{y})$. Notice that $\eta_{pq}$ is invariant under permutations of $\bar{z}$ and of $\bar{t}$ (separately), and consider now the set $Z_{pq}$ defined by the existential $H$-formula \[\exists \bar{z}\in H^k\left(\eta_{pq}(\bar{x},\bar{z},\bar{h},\bar{c})\wedge \bigvee_{\tau\in\mathbb{S}_{|\bar{h}|}}\Psi_p(\bar{x},\bar{z},\tau(\bar{h}),\bar{c})\right).\] Clearly $Z_{pq}\in p=\tp_{\Le_H}(\bar{a}/\bar{h},\bar{c})$, since $M\models \xi_{pq}(\bar{a},\bar{h}_a,\bar{h},\bar{c})\land \Psi_p(\bar{a},\bar{h}_a,\bar{h},\bar{c})$. On the other hand, if $Z_{pq}\in q=\tp_{\Le_H}(\bar{b}/\bar{h},\bar{c})$, then there is a tuple $\bar{h}_b'\in H^k$ and $\sigma\in \mathbb{S}_k,\tau,\tau'\in\mathbb{S}_{|\bar{h}|}$ such that $M\models \xi_{pq}(\bar{b},\sigma(\bar{h}_b'),\tau(\bar{h}),\bar{c})\wedge \Psi_p(\bar{b},\bar{h}_b',\tau'(\bar{h}),\bar{c})$. By Lemma \ref{unique-HB}, this would imply that $\bar{h}_b'$ is an enumeration of $\operatorname{HB}(\bar{b}/\operatorname{HB}(\bar{c},\tau'(\bar{h})),\bar{c},\tau'(\bar{h}))=\operatorname{HB}(\bar{b}/\bar{h},\bar{c})$. By construction, $M\not\models\xi_{pq}(\bar{b},\sigma(\bar{h}_b'),\tau(\bar{h}),\bar{c})$, a contradiction.
Hence, $Z_{pq}$ is an existential $H$-formula that separates $p$ from $q$.

Note that the set $\{Z_{pq}:p\in E\}$ is an open cover of $E$ and does not contain $q$. Since $E$ is closed, we obtain again by compactness that there are types $p_{q,1},\ldots,p_{q,N_q}\in E$ such that 
\begin{align*}
Z_q&=\bigvee_{i=1}^{N_q} Z_{p_{q,i}q}=\bigvee_{i=1}^{N_q} [\exists \bar{z}\in H^k(\eta_{p_{q,i}q}(\bar{x},\bar{z},\bar{h},\bar{c})\wedge \bigvee_{\tau\in\mathbb{S}_{|\bar{h}|}}\Psi_{p_{q,i}}(\bar{x},\bar{z},\tau(\bar{h}),\bar{c}))]\\
&=\exists \bar{z}\in H^k\bigvee_{i=1}^{N_q}\left( \eta_{p_{q,i}q}(\bar{x},\bar{z},\bar{h},\bar{c})\wedge\bigvee_{\tau\in\mathbb{S}_{|\bar{h}|}} \Psi_{p_{q,i}}(\bar{x},\bar{z},\tau(\bar{h}),\bar{c})\right)=:\exists \bar{z}\in H^k \chi_{q}(\bar{x}, \bar{z}, \bar{h},\bar{c}).
\end{align*}
 is a single existential $H$-formula $Z_{q}$ that separates $E$ from $q$, as desired. \qed$_{\text{Claim}}$

Given a type $q\in Y_0\setminus X$ with $\operatorname{SU}(q)=\omega \cdot n+k$, consider the $\Le$-formula $\chi_{q}(\bar{x}, \bar{z}, \bar{h},\bar{c})$ given by the previous claim. Note that $\chi_{q}(\bar{x}, \bar{z}, \bar{t},\bar{y})$ is invariant under permutations of $\bar{z}$ and of $\bar{t}$ by construction. 
And $M\models\forall\bar{x}\forall\bar{z}\left(\chi_{q}(\bar{x}, \bar{z}, \bar{h},\bar{c}) \to \bigvee_{i\leq N_q,\tau\in\mathbb{S}_{|\bar{h}|}}\Psi_{p_{q,i}}(\bar{x},\bar{z},\tau(\bar{h}),\bar{c})\right)$ for some $N_q$, where the types $p_{q,i}$ all satisfy $\operatorname{SU}(p_{q,i}/\bar{h},\bar{c})=\omega \cdot n+k$.

Consider now the $\Le_H$-type-definable (closed) set \[E_1=\{q\in S_{\Le_H}(\bar{h},\bar{c})
:q\in Y_0\setminus X \text{\ \ and\ \ }\operatorname{SU}(q)\geq \omega\cdot n+k\}.\] The collection $\{Z_q^c: q\in E_1\}$ forms an open cover of $E_1$, where $Z_q^c$ is the complement of $Z_q$. By compactness there are finitely many  types $q_1,\ldots,q_s$ such that $E_1\subseteq Z_{q_1}^c\cup \cdots \cup Z_{q_s}^c$. Moreover, the sets $Z_{q_1},\ldots,Z_{q_s}$ extend $E$ and the formulas defining them have the form $\exists \bar{z}\in H^k \chi_{q_i}(\bar{x},\bar{z},\bar{h},\bar{c})$, for $\Le$-formulas $\chi_{q_i}(\bar{x},\bar{z},\bar{h},\bar{c})$, $i=1,\ldots,s$. 
To simplify the notation, for $1\leq i\leq s$, write $\chi_{i}(\bar{x},\bar{z},\bar{h},\bar{c}):=\chi_{q_i}(\bar{x},\bar{z},\bar{h},\bar{c})$, $p_{i,j}:=p_{q_i,j}$ and $N_i:=N_{q_i}$.

Let $Y':=Y_0\cap Z_{q_1}\cap\cdots \cap Z_{q_s}$ (and if $E_1=\emptyset$, we set $Y':=Y_0$).
Then $\operatorname{SU}(Y'\triangle X)<\omega\cdot n+k$. Recall that $\chi_0(\bar{x},\bar{z},\bar{h},\bar{c}):=\bigvee_{i\leq N,\tau\in\mathbb{S}_{|\bar{h}|}}\Psi_{p_i}(\bar{x},\bar{z},\tau(\bar{h}),\bar{c})$. Let $N_0:=N$ and $p_{0,j}:=p_j$ for $j\leq N_0$. Then $Y_0$ is defined by $\exists \bar{z}\in H^k\,\chi_0(\bar{x},\bar{z},\bar{h},\bar{c})$, and $Y'$ is defined by the conjunction 
\[\hspace{0.3cm}\bigwedge_{0\leq i\leq s}\exists \bar{z}\in H^k \chi_{i}(\bar{x},\bar{z},\bar{h},\bar{c}).\]

To finish the proof, we need two more steps. The first one is to make $Y'$ definable by an existential $H$-formula, in order to do this we will move the existential quantifier $\exists \bar{z}\in H^k$ out of the conjunction. The second step is to apply Lemma \ref{lemalgH} in order to satisfy condition (1) in the statement of the lemma.  \\

Define the $\Le$-formula $\varphi(\bar{x},\bar{z},\bar{t},\bar{y}):=\bigwedge_{0\leq i\leq s} \chi_{i}(\bar{x},\bar{z},\bar{t},\bar{y})$ and let 
\begin{align*}
\eta(\bar{x},\bar{z},\bar{t},\bar{y})&:=\varphi_{\bar{z},H}(\bar{x},\bar{z},\bar{t},\bar{y})\\
&=\varphi(\bar{x},\bar{z},\bar{t},\bar{y})\land\left(\bigwedge_{1\leq i<j\leq k}z_i\neq z_j\right)\land\left(\bigwedge_{1\leq j\leq k}\Alg_{\varphi,z_j}(\bar{x},\bar{y},\bar{t},\hat{z}_j)\right)
\end{align*}
be the formula given as in Definition \ref{Halg}.

Let $Y$ be the definable set  given by the existential $H$-formula $\exists \bar{z}\in H^k\eta(\bar{x},\bar{z},\bar{h},\bar{c})$.

\textbf{Claim:} The definable set $Y$ and the $\Le$-formula $\eta$ satisfies all the conditions of the lemma. 

Condition (1): By construction $\varphi(\bar{x},\bar{z},\bar{t},\bar{y})$ is invariant under permutations of $\bar{z}$ and $\bar{t}$ since $\{\chi_i,0\leq i\leq s\}$ are. And $\eta(\bar{a}',H^k,\bar{h}',\bar{c}')$ is finite for any $\bar{a}',\bar{h}',\bar{c}'$ by Lemma \ref{lemalgH}.

Condition (2): Note that for each $0\leq i\leq s$ there is $N_{i}$ such that 
\[(*)\quad M\models \forall\bar{x}\,\forall\bar{z}\left(\chi_{i}(\bar{x},\bar{z},\bar{h},\bar{c})\to\bigvee_{\ell\leq N_i,\tau\in\mathbb{S}_{|\bar{h}|}}\Psi_{p_{i,\ell}}(\bar{x},\bar{z},\tau(\bar{h}),\bar{c})\right).\]
Let $\{p_1,\ldots,p_t\}=\{p_{i,\ell}:0\leq i\leq s, \ell\leq N_i\}$, then condition (2) is satisfied.

Condition (3): We first show that $\operatorname{SU}(Y'\triangle Y)<\omega\cdot n+k$. We know that $Y\subset Y'$, so it remains to show that $\operatorname{SU}(Y'\setminus Y)<\omega\cdot n+k$. That is, it suffices to show that every $\bar{a}\in Y'$ with $\operatorname{SU}(\bar{a}/\bar{h},\bar{c})=\omega\cdot n+k$ belongs to $Y$.

Let $\bar a\in Y'$ with $\operatorname{SU}(\bar a/\bar{h},\bar{c})=\omega\cdot n+k$. There are $\bar h_i\in H^k$ such that $M\models \chi_{i}(\bar a,\bar h_i,\bar{h},\bar{c})$ for $0\leq i\leq s$. By $(*)$, there is $\ell\leq N_i$ and $\tau_i\in \mathbb{S}_{|\bar{h}|}$ such that $M\models \Psi_{p_{i,\ell}}(\bar a,\bar h_i,\tau_i(\bar{h}),\bar{c})$. Hence $\bar h_i$ is an enumeration of $\operatorname{HB}(\bar a/\bar{h},\bar{c})$ by Lemma \ref{unique-HB}. Let $\bar h_a=(h_{1,a},\ldots,h_{k,a})$ be a fixed enumeration of $\operatorname{HB}(\bar a/\bar{h},\bar{c})$. Then $\chi_{i}(\bar a,\bar{h}_a,\bar{h},\bar{c})$ holds, as $\chi_{i}(\bar x,\bar z,\bar{h},\bar c)$ is invariant under permutations of $
\bar z$, so $M\models \varphi(\bar{a},\bar{h}_a,\bar{h},\bar{c})$. Clearly, $M\models\bigwedge_{1\leq i<i'\leq k}h_{i,a}\neq h_{i',a}$. \\

Let $\hat{h}_{a,j}=(h_{1,a},\ldots,h_{j-1,a},h_{j+1,a},\ldots,h_{k,a}).$ We need to check $M\models\Alg_{\varphi,z_j}(\bar{a},\bar{c},\bar{h},\hat{h}_{a,j})$. Suppose not. Let $\varphi_j(z_j,\bar{a},\hat{h}_{a,j},\bar{h},\bar{c}):=\varphi(\bar{a},h_{1,a},\ldots,h_{j-1,a},z_j,h_{j+1,a},\ldots,h_{k,a},\bar{h},\bar{c})$. Then $\varphi_j(M,\bar{a},\hat{h}_{a,j},\bar{h},\bar{c})$ is infinite. Therefore, by the density property, $\varphi_j(H,\bar{a},\hat{h}_{a,j},\bar{h},\bar{c})$ is infinite. However by condition (2), any $\bar{h}'\in H^k$ satisfying $\varphi(\bar{a},\bar{h}',\bar{h},\bar{c})$ should also satisfy $\Psi_{p_i}(\bar{a},\bar{h}',\tau(\bar{h}),\bar{c})$ for some $\tau\in\mathbb{S}_{|\bar{h}|}$ and $p_i$ of $SU$-rank $\omega\cdot n+k$ over $\bar{h},\bar{c}$. Hence $\bar{h}'$ is enumeration of $\operatorname{HB}(\bar{a}/\bar{h},\bar{c})$ by Lemma \ref{unique-HB} and $\varphi_j(H,\bar{a},\hat{h}_{a,j},\bar{h},\bar{c})\subseteq \operatorname{HB}(\bar{a}/\bar{h},\bar{c})$ must be finite, a contradiction. Therefore, $M\models\eta(\bar{a},\bar{h},\bar{h},\bar{c})$ and $\bar{a}\in Y$.

In conclusion, since \[Y\triangle X=(Y\triangle Y')\triangle (Y'\triangle X)\subseteq (Y\triangle Y')\cup (Y'\triangle X),\] we get $\operatorname{SU}(Y\triangle X)=\max\{\operatorname{SU}(Y\triangle Y'),\operatorname{SU}(Y'\triangle X)\}<\omega\cdot n+k$.
\end{proof}

In the following, we will show definability of $SU$-rank. Note that in all the results we have obtained so far we are using that the definable sets are defined by formulas with parameters that are $H$-independent. To obtain definability of $SU$-rank, we need to be able to uniformly define the $H$-basis, which we have partially achieved in Lemma \ref{lem-HB-base}.

\begin{theorem}\label{thm-def-SU}

Let $(M,H)$ be a sufficiently saturated $H$-structure. Then for any $\mathcal{L}_H$-formula $\varphi(\bar{x};\bar{y})$ over $\emptyset$, there is a finite subset $D_\varphi=\{(n_i,k_i):i\leq N\}\subseteq \mathbb{N}\times \mathbb{N}$, $\Le_H$-formulas $\{\psi_i(\bar{y}):i\leq N\}$ over $\emptyset$ such that the following hold:
\begin{enumerate}
\item The formulas $\{\psi_i(\bar{y}):i\leq N\}$ form a cover of $M^{|\bar{y}|}$.
\item For any $i\leq N$ and $\bar{c}\in M^{|\bar{y}|}$, $(M,H)\models \psi_i(\bar{c})$ implies that $\operatorname{SU}(\varphi(\bar{x};\bar{c}))=\omega\cdot n_i+k_i$.
\end{enumerate}
\end{theorem}

\begin{proof} Let us first fix a tuple $\bar{c}\in M^{|\bar{y}|}$. Throughout this proof we will define several formulas that depend on $\tp(\bar{c})$, but to simplify the notation we will index them with $\bar{c}$ instead of the more accurate index $\tp(\bar{c})$. By Lemma \ref{lem-HB-base}, there is an $\Le$-formula $\xi_{\bar{c}}(\bar{t},\bar{y})$ invariant under permutations of $\bar{t}$ such that $M\models\xi_{\bar{c}}(\bar{h},\bar{c})$ for any enumeration $\bar{h}$ of $\operatorname{HB}(\bar{c})$ and $\xi_{\bar{c}}(H^{|\bar{t}|},\bar{c}')$ is finite for any $\bar{c}'$. (If $\operatorname{HB}(\bar{c})=\emptyset$, we set $\xi_{\bar{c}}(\bar{y}):=\bar{y}=\bar{y}$.) 

Suppose $\operatorname{SU}(\varphi(\bar{x};\bar{c}))=\omega\cdot n+k$. 
By Lemma \ref{lem-almostExistential}, there is an existential $\Le$-formula $\eta_{\bar{c}}(\bar{x},\bar{z},\bar{t},\bar{y})$ invariant under permutations of $\bar{z}$ and of $\bar{t}$, such that the following holds: 
Let $\bar{h}$ be a fixed enumeration of $\operatorname{HB}(\bar{c})$, then
\[\operatorname{SU}\left(\big(\exists \bar{z}\in H^k \eta_{\bar{c}}(\bar{x},\bar{z};\bar{h},\bar{c})\big)\triangle\varphi(\bar{x};\bar{c})\right)<\operatorname{SU}(\varphi(\bar{x};\bar{c})),\]  $\dim(\eta_{\bar{c}}(\bar{x},\bar{z};\bar{h},\bar{c}))=n+k$ and there are complete types $\{p_i:i\leq K_{\bar{c}}\}$ over $\bar{h},\bar{c}$ with $SU$-rank $\omega\cdot n+k$, such that
$$M\models \forall\bar{x}\forall\bar{z}\left(\eta_{\bar{c}}(\bar{x},\bar{z};\bar{h},\bar{c})\to\bigvee_{i\leq K_{\bar{c}},\tau\in\mathbb{S}_{|\bar{t}|}}\Psi_{p_i}(\bar{x},\bar{z};\tau(\bar{h}),\bar{c})\right).$$
Recall that, by the definability of dimension in the geometric structure $M$, there is an $\Le$-formula $\Dim_{\eta_{\bar{c}},(\bar{x},\bar{z}),n+k}(\bar{t},\bar{y})$ stating that $\dim(\eta_{\bar{c}}(\bar{x},\bar{z};\bar{t},\bar{y}))=n+k$. Since $\eta_{\bar{c}}(\bar{x},\bar{z};\bar{t},\bar{y})$ is invariant under permutations of $\bar{t}$, so is $\Dim_{\eta_{\bar{c}},(\bar{x},\bar{z}),n+k}(\bar{t},\bar{y})$.

Let $\phi_{\bar{c}}(\bar{x},\bar{h},\bar{c}):=\exists \bar z\in H^k \eta_{\bar{c}}(\bar{x},\bar{z},\bar{h},\bar{c})\triangle\varphi(\bar{x};\bar{c})$. Notice that 

\[\{\exists\bar{z}_q\in H^{k_q}\Psi_q(\bar{x},\bar{z}_q,\bar{h},\bar{c}): q(\bar{x})\in S_{T^{ind}}(\bar{h},\bar{c}), \phi_{\bar{c}}(\bar{x},\bar{h},\bar{c})\in q\}\] forms an open cover of $\phi_{\bar{c}}(\bar{x},\bar{h},\bar{c})$. Since $\operatorname{SU}(\phi_{\bar{c}}(\bar{x},\bar{h},\bar{c}))<\omega\cdot n+k$, all types $q$ considered in the cover above have SU-rank strictly less than $\omega\cdot n+k$. Thus, by compactness, there are finitely many types $\{q_j:j\leq N_{\bar{c}}\}$ with $\operatorname{SU}(q_j/\bar{h},\bar{c})<\omega\cdot n+k=\operatorname{SU}(\varphi(\bar{x},\bar{c}))$ such that 
$$(M,H)\models\forall \bar{x}\left(\phi_{\bar{c}}(\bar{x},\bar{h},\bar{c})\to \bigvee_{j\leq N_{\bar{c}}}\exists\bar{z}_j\in H^{k_j}\Psi_{q_j}(\bar{x},\bar{z}_j,\bar{h},\bar{c})\right).$$
In particular,
$$(M,H)\models\forall \bar{x}\left(\phi_{\bar{c}}(\bar{x},\bar{h},\bar{c})\to \bigvee_{j\leq N_{\bar{c}},\tau\in\mathbb{S}_{|\bar{t}|}}\exists\bar{z}_j\in H^{k_j}\Psi_{q_j}(\bar{x},\bar{z}_j,\tau(\bar{h}),\bar{c})\right).$$



Let us consider the formula
\begin{align*} 
\psi_{\bar{c}}'(\bar{t},\bar{y}):=&\xi_{\bar{c}}(\bar{t},\bar{y})\land\Dim_{\eta_{\bar{c}},(\bar{x},\bar{z}),n+k}(\bar{t},\bar{y})\\
& \wedge \forall\bar{x}\forall\bar{z}\left(\eta_{\bar{c}}(\bar{x},\bar{z};\bar{t},\bar{y})\to\bigvee_{i\leq K_{\bar{c}},\tau\in\mathbb{S}_{|\bar{t}|}}\Psi_{p_i}(\bar{x},\bar{z};\tau(\bar{t}),\bar{y})\right)\\
&\wedge \forall \bar{x}\left(\phi_{\bar{c}}(\bar{x},\bar{t},\bar{y})\to \bigvee_{j\leq N_{\bar{c}},\tau\in\mathbb{S}_{|\bar{t}|}}\exists\bar{z}_j\in H^{k_j}\Psi_{q_j}(\bar{x},\bar{z}_j,\tau(\bar{t}),\bar{y})\right).
\end{align*}
(In case $\omega\cdot n+k=0$, we set $\bigvee_{j\leq N_{\bar{c}},\tau\in\mathbb{S}_{|\bar{t}|}}\exists\bar{z}_j\in H^{k_j}\Psi_{q_j}(\bar{x},\bar{z}_j,\tau(\bar{t}),\bar{y}):=\bot$.)
Then $\psi'_{\bar{c}}(\bar{t},\bar{y})$ is invariant under permutations of $\bar{t}$ and $(M,H)\models\psi'_{\bar{c}}(\bar{h},\bar{c})$.
Let $\psi_{\bar{c}}(\bar{y}):=\exists\bar{t}\in H^{|\bar{t}|}\psi_{\bar{c}}'(\bar{t},\bar{y})$. 

\begin{claim}  \label{claim-main-inside} For any $\bar{c}'$ such that $(M,H)\models \psi_{\bar{c}}(\bar{c}')$ we have $$\operatorname{SU}(\varphi(\bar{x};\bar{c}'))=\omega\cdot n+k.$$
\end{claim}

We will prove this claim later. Let us conclude now the proof of the theorem assuming Claim \ref{claim-main-inside}.

Assume Claim \ref{claim-main-inside} holds. Then $\{\psi_{\bar{c}}(\bar{y}):\bar{c}\in M^{|\bar{y}|}\}$ is a cover of $M^{|\bar{y}|}$, and by compactness, there are finitely many formulas $\psi_1(\bar{y}):=\psi_{\bar{c}_1}(\bar{y}),\ldots,\psi_N(\bar{y}):=\psi_{\bar{c}_N}(\bar{y})$ that cover $M^{|\bar{y}|}$, and for each $i\leq N$ there are $(n_i,k_i)\in\mathbb{N}\times\mathbb{N}$ such that, whenever $(M,H)\models \psi_i(\bar{c}')$, we have $\operatorname{SU}(\varphi(\bar{x},\bar{c}'))=\omega\cdot n_i+k_i$. By taking $D_\varphi=\{(n_i,k_i):i\leq N\}$, the theorem follows. \qedhere$_{\text{Theorem \ref{thm-def-SU}}}$
\end{proof}

We show now an improved version of Claim \ref{claim-main-inside} which is also essential for determine the measure and dimension of $\varphi(\bar{x};\bar{c})$ in Section \ref{section:measure}.

\begin{claim}\label{claim-main}
We already know the following holds:
\begin{itemize}
    \item The $\Le_H$-formula $\psi_{\bar{c}}'(\bar{t},\bar{y})$ and the $\Le$-formula $\eta_{\bar{c}}(\bar{x},\bar{z},\bar{t},\bar{y})$ are invariant under permutations of $\bar{t}$ and $\eta_{\bar{c}}(\bar{x},\bar{z},\bar{t},\bar{y})$ is invariant under permutations of $\bar{z}$ as well.
    \item
   If $M\models \psi'_{\bar{c}}(\bar{h}',\bar{c}')$ for some $\bar{h}',\bar{c}'$, then \[M\models\forall\bar{x}\forall\bar{z}\left(\eta_{\bar{c}}(\bar{x},\bar{z};\bar{h}',\bar{c}')\to\bigvee_{i\leq K_{\bar{c}},\tau\in\mathbb{S}_{|\bar{t}|}}\Psi_{p_i}(\bar{x},\bar{z};\tau(\bar{h}'),\bar{c}')\right),\] where $\{p_i:i\leq K_{\bar{c}}\}$ are types over $(\bar{h},\bar{c})$ of $SU$-rank $\omega\cdot n+k$.
    \item
     By Lemma \ref{lem-HB-base} (2) and the fact that $\psi'_{\bar{c}}$ implies $\xi_{\bar{c}}$, we know that if $\bar{c}'$ has the same $SU$-rank as $\bar{c}$ and $M\models \psi'_{\bar{c}}(\bar{h}',\bar{c}')$ for some $\bar{h}'\in H^{|\bar{t}|}$, then $\bar{h}'$ is a permutation of $\operatorname{HB}(\bar{c}')$.
     
\end{itemize}

Moreover, we claim the following: whenever $\bar{c}',\bar{h}'$ are tuples such that $\psi'_{\bar{c}}(\bar{h}',\bar{c}')$ holds and $\bar{h}'\in H^{|\bar{t}|}$ (note that $\bar{h}'\subseteq \operatorname{HB}(\bar{c}')$ by Remark \ref{remSubHB} applied to the subformula $\xi_{\bar{c}}(\bar{t},\bar{y})$), we have 
\begin{enumerate}
\item $\operatorname{SU}(\exists\bar{z}\in H^k\eta_{\bar{c}}(\bar{x},\bar{z};\bar{h}',\bar{c}'))=\omega\cdot n+k$;
\item The definable set $\eta_{\bar{c}}(\bar{a},H^k;\bar{h}',\bar{c}')$ is finite for any tuple $\bar{a}$ and $$\operatorname{SU}(\exists\bar{z}\in H^k\eta_{\bar{c}}(\bar{x},\bar{z};\bar{h}',\bar{c}')\triangle\varphi(\bar{x};\bar{c}'))<\omega\cdot n+k;$$

\item For any $\bar{e},\bar{b}$ satisfying $\eta_{\bar{c}}(\bar{e},\bar{b},\bar{h}',\bar{c}')$ with $\dim(\bar{e},\bar{b}/\bar{c}',\operatorname{HB}(\bar{c}'))=n+k$, there are $\bar{a},\bar{d}$ such that $$\tp_{\mathcal{L}}(\bar{e},\bar{b}/\bar{c}',\operatorname{HB}(\bar{c}'))=\tp_{\mathcal{L}}(\bar{a},\bar{d}/\bar{c}',\operatorname{HB}(\bar{c}')),$$ and $\operatorname{SU}(\bar{a}/\bar{c}',\operatorname{HB}(\bar{c}'))=\omega\cdot n+k$ with $\bar{d}$ an enumeration of $\operatorname{HB}(\bar{a}/\bar{c}',\operatorname{HB}(\bar{c}'))$. 
\item
For any $\bar{a}$ satisfying $\varphi(\bar{a},\bar{c}')$ such that $\operatorname{SU}(\bar{a}/\bar{c}',\operatorname{HB}(\bar{c}'))=\omega\cdot n+k$, and for $\bar{d}\in H^{k}$, we have $\eta_{\bar{c}}(\bar{a},\bar{d},\bar{h}',\bar{c}')$ holds if and only if $\bar{d}$ is an enumeration of $\operatorname{HB}(\bar{a}/\bar{c}',\operatorname{HB}(\bar{c}'))$.
\end{enumerate}
\end{claim}

\begin{remark}
Let us show that Claim \ref{claim-main} implies Claim \ref{claim-main-inside}. 
Suppose $(M,H)\models\psi_{\bar{c}}(\bar{c}')$, then there is $\bar{h}'\in H^{|\bar{t}|}$ such that $\psi_{\bar{c}'}'(\bar{h}',\bar{c}')$ holds. By item (1) we have $\operatorname{SU}(\exists\bar{z}\in H^k\eta_{\bar{c}}(\bar{x},\bar{z};\bar{h}',\bar{c}'))=\omega\cdot n+k$, and by item (2),  
$\operatorname{SU}(\exists\bar{z}\in H^k\eta_{\bar{c}}(\bar{x},\bar{z};\bar{h}',\bar{c}')\triangle\varphi(\bar{x};\bar{c}'))<\omega\cdot n+k$. Therefore, $\operatorname{SU}(\varphi(\bar{x};\bar{c}'))=\omega\cdot n+k$ as desired. 
\end{remark}

\begin{proof}[Proof of Claim \ref{claim-main}]
Given $\bar{c}'$, let $\bar{h}'\in H^{|\bar{t}|}$  be
such that $\psi_{\bar{c}}'(\bar{h}',\bar{c}')$ holds. 
We need to show that items (1)-(4) hold.

We start by showing item (2):
By Lemma \ref{lem-almostExistential} (1), we have that $\eta_{\bar{c}}(\bar{a},H^k;\bar{h}',\bar{c}')$ is finite for any $\bar{a},\bar{h}',\bar{c}'$. 

By definition of $\psi_{\bar{c}}'(\bar{h}',\bar{c}')$, $$(M,H)\models\forall \bar{x}\left(\phi_{\bar{c}}(\bar{x},\bar{h}',\bar{c}')\to \bigvee_{j\leq N_{\bar{c}},\tau\in\mathbb{S}_{|\bar{t}|}}\exists\bar{z}_j\in H^{k_j}\Psi_{q_j}(\bar{x},\bar{z}_j,\tau(\bar{h}'),\bar{c}')\right),$$ where $\phi_{\bar{c}}(\bar{x},\bar{h}',\bar{c}')=\exists\bar{z}\in H^k\eta_{\bar{c}}(\bar{x},\bar{z};\bar{h}',\bar{c}')\triangle\varphi(\bar{x};\bar{c}')$ and $\Psi_{q_j}(\bar{x},\bar{z}_j,\bar{h},\bar{c})$ is defined as in Definition \ref{mainformula} and $\operatorname{SU}(q_j/\bar{h},\bar{c})<\omega\cdot n+k$. By Lemma \ref{unique-HB}, we have \[\operatorname{SU}(\exists\bar{z}_j\in H^{k_j}\Psi_{q_j}(\bar{x},\bar{z}_j,\tau(\bar{h}'),\bar{c}'))\leq \operatorname{SU}(q_j/\bar{h},\bar{c})<\omega\cdot n +k.\] Hence, $\operatorname{SU}\left((\exists\bar{z}\in H^k\eta_{\bar{c}}(\bar{x},\bar{z};\bar{h}',\bar{c}'))\triangle\varphi(\bar{x};\bar{c}')\right)<\omega\cdot n+k$ over $\bar{c}',\operatorname{HB}(\bar{c}')$ (as $\bar{h}'\subseteq\operatorname{HB}(\bar{c}')\subseteq \acl_{\Le_H}(\bar{c}')$). 

Now we show item (3) of the claim. Suppose $M\models \eta_{\bar{c}}(\bar{e},\bar{b},\bar{h}',\bar{c}')$ and $\dim(\bar{e},\bar{b}/\bar{c}',\operatorname{HB}(\bar{c}'))=n+k$. By definition of $\psi_{\bar{c}}'(\bar{h}',\bar{c}')$, there is $p_i$ and $\tau\in\mathbb{S}_{|\bar{t}|}$ such that 
$M\models \Psi_{p_i}(\bar{e},\bar{b};\tau(\bar{h}'),\bar{c}')$.
By definition of $\Psi_{p_i}$, we have $\dim(\bar{e}/\bar{b},\bar{h}',\bar{c}')\leq n$, and since $\bar{h}'\subseteq HB(\bar{c}')$, we have $\dim(\bar{e}/\bar{b},\bar{c}',\operatorname{HB}(\bar{c}'))\leq n$. Since $|\bar{b}|=k$, we must have $\dim(\bar{b}/\bar{c}',\operatorname{HB}(\bar{c}'))=k$
and $\dim(\bar{e}/\bar{b},\bar{c}',\operatorname{HB}(\bar{c}'))=n$. 
Consider $\tp_{\mathcal{L}}(\bar{b}/\bar{c}',\operatorname{HB}(\bar{c}'))$, which has dimension $k$. By the density property in $H$-structures, we can choose $\bar{d}\in H^k$ such that 
$$\tp_{\mathcal{L}}(\bar{d}/\bar{c}',\operatorname{HB}(\bar{c}'))=\tp_{\mathcal{L}}(\bar{b}/\bar{c}',\operatorname{HB}(\bar{c}')).$$

Suppose $\bar{e}=\bar{e}_1\bar{e}_2$, where $|\bar{e}_1|=\dim(\bar{e}_1/\bar{b},\bar{c}',\operatorname{HB}(\bar{c}'))=n$ and 
$\bar{e}_2\in\acl(\bar{e}_1,\bar{b},\bar{c}',\operatorname{HB}(\bar{c}'))$.
Consider $q:=\tp_{\mathcal{L}}(\bar{e}_1/\bar{b},\bar{c}',\operatorname{HB}(\bar{c}'))$, it has dimension $n$. 
Let $q'$ be the type over $\bar{d},\bar{c}',\operatorname{HB}(\bar{c}')$ obtained by replacing $\bar{b}$ by $\bar{d}$. 
Since $\bar{b}$ and $\bar{d}$ has the same $\mathcal{L}$-type over $\bar{c}',\operatorname{HB}(\bar{c}')$, we also have $q'$ has dimension $n$.
By the extension property in $H$-structures, we can choose $\bar{a}_1\in M^n$ such that $\bar{a}_1\models q'$ and $\dim(\bar{a}_1/H,\bar{c}')=n$.
Let $\bar{a}_2$ be such that $$\tp_{\mathcal{L}}(\bar{a}_2,\bar{a}_1,\bar{d}/\bar{c}',\operatorname{HB}(\bar{c}'))=\tp_{\mathcal{L}}(\bar{e}_2,\bar{e}_1,\bar{b}/\bar{c}',\operatorname{HB}(\bar{c}'))$$ 
 and let $\bar{a}=\bar{a}_1\bar{a}_2$.
Then by construction, we have 
$$\tp_{\mathcal{L}}(\bar{a},\bar{d}/\bar{c}',\operatorname{HB}(\bar{c}'))=\tp_{\mathcal{L}}(\bar{e},\bar{b}/\bar{c}',\operatorname{HB}(\bar{c}')).$$

We claim that $\operatorname{SU}(\bar{a}/\bar{c}',\operatorname{HB}(\bar{c}'))=\omega\cdot n+k$. 
We know that
\[
(**)\quad\bar{a}\ind_{\bar{d},\bar{c}',\operatorname{HB}(\bar{c}')}H.
\]

We want to show that $\bar{d}=\operatorname{HB}(\bar{a}/\bar{c}',\operatorname{HB}(\bar{c}'))$. We only need to show that $\bar{d}$ is a minimal tuple in $H$ for which $(**)$ holds.
Suppose not, then there is $d_i\in\bar{d}=(d_1,\ldots, d_{k})$ such that 
$\bar{a}_{2}\in\acl(\bar{a}_{1},\bar{d},\bar{c}',\operatorname{HB}(\bar{c}'))\setminus\{d_i\}$. 
Thus, 
$$
\dim(\bar{a},\hat{d}_i/\bar{c}',\operatorname{HB}(\bar{c}'))=
\dim(\bar{a}_{1},\hat{d}_i/\bar{c}',\operatorname{HB}(\bar{c}'))=n+k-1.
$$
Note that by the construction of $\eta_{\bar{c}}$ in Lemma \ref{lem-almostExistential} (which uses Definition \ref{Halg}), we must have that 
$d_i\in \acl(\bar{a},\hat{d}_i,\bar{h}',\bar{c}').$
Therefore, 
\begin{align*}
\dim(\bar{a},\bar{d}/\bar{c}',\operatorname{HB}(\bar{c}'))=&\dim(d_i/\bar{a},\hat{d}_i,\bar{c}',\operatorname{HB}(\bar{c}'))
+\dim(\bar{a},\hat{d}_i/\bar{c}',\operatorname{HB}(\bar{c}'))\\
=&0+n+k-1=n+k-1.
\end{align*}
This contradicts that $\dim(\bar{a},\bar{d}/\bar{c}',\operatorname{HB}(\bar{c}'))=n+k$.

Now we can prove item (1):
By the definition of $\psi_{\bar{c}'}$, \[M\models \forall\bar{x}\forall\bar{z}\left(\eta_{\bar{c}}(\bar{x},\bar{z};\bar{h}',\bar{c}')\to\bigvee_{i\leq K_{\bar{c}},\tau\in\mathbb{S}_{|\bar{t}|}}\Psi_{p_i}(\bar{x},\bar{z};\tau(\bar{h}'),\bar{c}')\right),\] where $p_i$ are $\Le_H$-types of $SU$-rank $\omega\cdot n+k$.
Let $\bar{a}_0$ be such that $(M,H)\models \exists\bar{z}\in H^k\eta_{\bar{c}}(\bar{a}_0,\bar{z},\bar{h}',\bar{c}')$.
Therefore, $(M,H)\models \exists\bar{z}\in H^k\Psi_{p_i}(\bar{a}_0,\bar{z};\tau(\bar{h}'),\bar{c}')$ for some $p_i$ and $\tau$.
Then by Lemma \ref{unique-HB}, we have $\operatorname{SU}(\bar{a}_0/\tau(\bar{h}'),\bar{c}')\leq \omega\cdot n+k$.
We conclude \[\operatorname{SU}(\exists\bar{z}\in H^k\eta_{\bar{c}}(\bar{x},\bar{z};\bar{h}',\bar{c}'))\leq \omega\cdot n+k.\]
We only need to show $\operatorname{SU}(\exists\bar{z}\in H^k\eta_{\bar{c}}(\bar{x},\bar{z};\bar{h}',\bar{c}'))\geq \omega\cdot n+k$.

Since $\Dim_{\eta_{\bar{c}},\bar{x},\bar{z},n+k}(\bar{h}',\bar{c}')$ holds by definition of $\psi_{\bar{c}}'(\bar{h}',\bar{c}')$, we have $\dim(\eta_{\bar{c}}(\bar{x},\bar{z};\bar{h}',\bar{c}'))=n+k$. 
By saturation of $M$, there are $\bar{e},\bar{b}$ be such that 
$M\models \eta_{\bar{c}}(\bar{e},\bar{b},\bar{h}',\bar{c}')$ and $\dim(\bar{e},\bar{b}/\bar{c}',\operatorname{HB}(\bar{c}'))=n+k$. By item (3), there are $\bar{a},\bar{d}$ such that $\bar{d}\in H^k$, $\operatorname{SU}(\bar{a}/\bar{c}',\operatorname{HB}(\bar{c}'))=\omega\cdot n+k$ and $M\models\eta_{\bar{c}}(\bar{a},\bar{d},\bar{h}',\bar{c}')$. Hence $\operatorname{SU}(\exists\bar{z}\in H^k\eta_{\bar{c}}(\bar{x},\bar{z};\bar{h}',\bar{c}'))\geq \operatorname{SU}(\bar{a}/\bar{c}',\operatorname{HB}(\bar{c}'))=\omega\cdot n+k$ and we are done.

Now we prove item (4) of the claim. 
Suppose $(M,H)\models \varphi(\bar{a},\bar{c}')$ and $\operatorname{SU}(\bar{a}/\bar{c}',\operatorname{HB}(\bar{c}'))=\omega\cdot n+k$. 
Since $$\operatorname{SU}(\exists\bar{z}\in H^k\eta_{\bar{c}}(\bar{x},\bar{z};\bar{h}',\bar{c}')\triangle\varphi(\bar{x};\bar{c}'))<\omega\cdot n+k,$$ 
we must have $M\models \eta_{\bar{c}}(\bar{a},\bar{d}';\bar{h}',\bar{c}')$ for some $\bar{d}'\in H^k$.
By assumption, there is some $p_i$ complete type of $SU$-rank $\omega\cdot n+k$ over $\bar{h},\bar{c}$ and $\tau\in\mathbb{S}_{|\bar{t}|}$ such that 
$M\models\Psi_{p_i}(\bar{a},\bar{d}';\tau(\bar{h}'),\bar{c}')$.
By Lemma \ref{unique-HB}, $\bar{d}'$ is an enumeration of $\operatorname{HB}(\bar{a}/\bar{c}',\tau(\bar{h}'),\operatorname{HB}(\bar{c}',\tau(\bar{h}')))=\operatorname{HB}(\bar{a}/\bar{c}',\operatorname{HB}(\bar{c}'))$.
On the other hand, if $\bar{d}$ be an enumeration of $\operatorname{HB}(\bar{a}/\bar{c}',\operatorname{HB}(\bar{c}'))$, 
then it is a permutation of $\bar{d}'$.
By construction, the formula $\eta_{\bar{c}}(\bar{x},\bar{z};\bar{h}',\bar{c}')$ is invariant under permutations of $\bar{z}$, therefore, $M\models \eta_{\bar{c}}(\bar{a},\bar{d};\bar{h}',\bar{c}')$.
\end{proof}

\bc Let $T$ be a nowhere trivial SU-rank 1 theory. Then $T^{ind}$ eliminates $\exists^{\infty}$.
\ec
\bdem 
Theorem \ref{thm-def-SU} establishes that the $SU$-rank in $T^{ind}$ is definable, in the sense that given a formula $\varphi(\bar{x},\bar{y})$ the relation $\operatorname{SU}(\varphi(\bar{x},\bar{a}))=\omega\cdot n+k$ is a definable condition on $\bar{a}$. In particular, when $n=0$ and $k=0$, this yields elimination of $\exists^{\infty}$.
\edem

\section{Comparing two notions of dimension}\label{section:comparing-dimensions}

Let $T$ be the theory of an infinite ultraproduct $M=\prod_{\U}M_n$ of finite $\Le$-structures in a 1-dimensional asymptotic class. Note that $T$ is an $SU$-rank one theory, thus geometric. Suppose furthermore that $T$ is nowhere trivial. We will study the pseudofinite $H$-structure $(M,H)=\prod_{\U}(M_n,H_n)$ built in \cite{Zou}. For each $X\subset M^m$ definable in the extended language $\mathcal{L}_H$, we will consider two counting dimensions associated to the ultraproduct which correspond to the pseudofinite coarse dimensions with respect to $M$ and $H$. Formally, we define $\delta_H(X)=\lim_{n,\U}\log |X(M_n)|/\log |H_n|$ and let $\delta_M(X)=\lim_{n,\U}\log |X(M_n)|/\log |M_n|$.
Note that we may build $(M,H)$ in such a way that $\delta_M(H)=0$, and equivalently, $\delta_H(M)=\infty$ (see Remark 2.9 of \cite{Zou}).

The goal of this section is to show that, $\delta_M(X)=k$ corresponds to $\operatorname{SU}(X)=\omega\cdot k+\ell$ for some $\ell\in\mathbb{N}$.
We also show that, under extra
assumptions, $\delta_H(X)=n$ corresponds to $\operatorname{SU}(X)=n$. 

\begin{lemma}\label{prop-deltaH}
Let $X\subset M^m$ be a definable set defined over an $H$-independent tuple $\bar{c}$. Assume that $\operatorname{SU}(X)=n$. Then $\delta_H(X)\leq n$.
\end{lemma}

\begin{proof}  Let $q(\bar{x})$ be a complete type over $\bar{c}$ extending $X$ and let $\bar{b} \models q(\bar{x})$. Then $\operatorname{SU}(q)\leq\operatorname{SU}(X)=n$.
Let $\Psi_q(\bar{x},\bar{z})$ be the $\Le$-formula over parameters in $\bar{c}$ defined as in Definition $\ref{mainformula}$. Then $|\bar{z}|=\operatorname{SU}(q)\leq n$, $\exists\bar{z}\in H^{|\bar{z}|}\Psi_q(\bar{x},\bar{z})\in q(\bar{x})$ and $M\models \forall \bar{z} \exists^{\leq k_q} \bar{x} \Psi_q(\bar{x};\bar{z})$ for some $k_q\in\mathbb{N}$. 
By compactness, there are finitely many formulas $\{\Psi_{q_i}(\bar{x};\bar{z}_i): 1\leq i\leq \ell\}$ and a single $k$ such that \[X\subseteq \bigcup_{1\leq i\leq \ell}\exists \bar{z}_i\in H^{|\bar{z}_i|}\Psi_{q_i}(M^m;\bar{z}_i) \text{\hspace{0.5cm}and\hspace{0.5cm}}M\models \forall \bar{z_i}\exists^{\leq k}\bar{x}(\Psi_{q_i}(\bar{x},\bar{z}_i))\] for $i=1,\ldots,\ell$. 

Then,
\begin{align*}
\delta_H(X)&\leq \delta_H\left(\bigcup_{1\leq i\leq \ell}\exists \bar{z}_i\in H^{|\bar{z}_i|}\Psi_{q_i}(M^m;\bar{z}_i) \right)\\
&= \max_{1\leq i\leq \ell}\left\{\delta_H\big(\exists \bar{z}_i\in H^{|\bar{z}_i|}\Psi_{q_i}(M^m;\bar{z}_i)\big) \right\}\leq \delta_H(H^{|\bar{z}_i|})\leq \delta_H(H^n)= n.\qedhere
\end{align*}
\end{proof}

To proof the converse we will need an extra assumption. By \cite[Proposition 3.5]{BeVa2}, the definable subsets of $H^n$ are just the traces of $\mathcal{L}$-formulas with parameters in $M$. Thus we expect that, whenever $Y\subset H^n$ has $SU$-rank $n$ and $Y=Y_0\cap H^n$ for an $\Le$-definable subset $Y_0\subseteq M^n$ of dimension $n$, the size of $|Y|/|H^n|$ should be roughly that of  $|Y_0|/|M^n|$. Note that the standard part of $|Y_0|/|M^n|$ is the measure $\mu$ of $Y_0$, which is positive since $Y_0$ has dimension $n$.\footnote{Although it is natural to expect the size of definable subsets of $H^n$ to satisfy this property, it can be difficult to prove it in general. For example, if $M$ is a pseudofinite field and $Y_a$ is the set defined by $\exists z(z^2=x-a)$, we need $H\cap Y_a$ to have size approximately $\frac{1}{2}|H|$ for all $a\in M$. This appears to be a non-trivial condition in the construction of $H$ carried in \cite{Zou}.} 

\textbf{Assumption $(\star)$}: whenever $Y\subset H^n$ is $\mathcal{L}_H$-definable of $SU$-rank $n$, there is $\mu_0\in\mathbb{R}^{>0}$ such that $|Y|\geq \mu_0|H|^n$. 

\begin{lemma}
Let $X\subset M^m$ be definable over an $H$-independent tuple of parameters $\bar{c}$. Assume that $\operatorname{SU}(X)=n$ and that the condition $(\star)$ holds. Then $\delta_H(X)\geq n$.
\end{lemma}

\begin{proof} By Claim \ref{claim-main} inside the proof of Theorem \ref{thm-def-SU}, we can find a set $X_1$ defined by a formula of the form $\exists \bar{h}\in H^n\eta_{\bar{c}}(\bar{x},\bar{h})$ such that $\operatorname{SU}(X\triangle X_1/\bar{c})<n$, where $\eta_{\bar{c}}$ is an $\mathcal{L}$-formula such that $\eta_{\bar{c}}(\bar{a},H^n)$ is finite for any $\bar{a}\in M^{m}$ and $\eta_{\bar{c}}(M^m,\bar{d})$ is finite for any $\bar{d}\in M^{|\bar{d}|}$. By the previous Lemma, $\delta_H(X_1\triangle X)<n$ and $\delta_H(X)\leq n$. If we show that $\delta_H(X_1)\geq n$, then 
$\delta_H(X)=\delta_H(X_1\triangle(X_1\triangle X))\geq \delta_H(X_1\setminus (X_1\triangle X))\geq n$. Thus it suffices to prove $\delta_H(X_1)\geq n$. 
Let $Y:=\exists \bar{x} \eta_{\bar{c}}(\bar{x},H^n)$. Clearly $Y\subseteq H^n$ is $\Le_H$-definable. Since $\operatorname{SU}(X_1)=n$ and $\eta_{\bar{c}}(\bar{a},H^n)$ is finite and non-empty for any $\bar{a}\in X_1$, we know that $\SU(Y)=n$.  By condition $(\star)$, we get 
$|Y|\geq \mu_0|H|^n$ for some $\mu_0>0$. On the other hand, by compactness and $\omega$-saturation of $(M,H)$, there is $k$ such that $|\eta_{\bar{c}}(\bar{a},H^n)|\leq k$ for any $\bar{a}\in M^{m}$. Hence, $k|X_1|\geq \mu_0|H|^n$ and we have

\[\delta_H(X_1)= \lim_{n,\mathcal{U}}\dfrac{\log|X_1|}{\log|H|}\geq \lim_{n,\mathcal{U}}\dfrac{n\log|H|+\log\left(\frac{\mu_0}{k}\right)}{\log |H|}=n,\] as desired.
\end{proof}

Now we discuss the coarse dimension $\delta_M$, for which we will \textbf{not} require the assumption ($\star$).

\begin{proposition} \label{0-1dim} Let $\varphi(x,\bar{y})$ be an $\Le_H$-formula over $\emptyset$ with $x$ a single variable. Then there are $\Le_H$-formulas $\theta_0(\bar{y})$ and $\theta_1(\bar{y})$ over $\emptyset$ such that 
\begin{enumerate}
    \item 
    $\theta_0(M^{|\bar{y}|})$ and $\theta_1(M^{|\bar{y}|})$ form a partition of $M^{|\bar{y}|}$;
    \item
    For any $\bar{b}\in M^{|\bar{y}|}$ and $i\in\{0,1\}$ we have \[\delta_M(\varphi(M;\bar{b}))=i \text{ if and only if }(M,H)\models\theta_i(\bar{b}).\]
    \item
   $(M,H)\models \theta_0(\bar{b})$ if and only if $\operatorname{SU}(\varphi(x;\bar{b}))<\omega$ for any $\bar{b}\in M^{|\bar{y}|}$.
\end{enumerate}
\end{proposition}

\begin{proof}
By Theorem \ref{thm-def-SU}, and the fact that $x$ is a single variable, there is a finite set $\{\alpha_i:i\leq N\}$ of ordinals with $\alpha_i\leq \omega$, and there are $\Le_H$-formulas $\{\psi_i(\bar{y}):i\leq N\}$ over $\emptyset$ such that
    $\{\psi_i(M^{|\bar{y}|}):i\leq N\}$ is a cover of $M^{|\bar{y}|}$ and
    for any $\bar{b}\in M^{|\bar{y}|}$ and $i\leq N$, we have $(M,H)\models\psi_i(\bar{b})$ implies $\operatorname{SU}(\varphi(x,\bar{b}))=\alpha_i$.
Let $\theta_0(\bar{y}):=\bigvee_{i\leq N, \alpha_i<\omega}\psi_i(\bar{y})$ and $\theta_1(\bar{y}):=\bigvee_{i\leq N,\alpha_i=\omega}\psi_i(\bar{y})$. Clearly $\theta_0(M^{|\bar{y}|})$ and $\theta_1(M^{|\bar{y}|})$ form an partition of $M^{|\bar{y}|}$. If $(M,H)\models\theta_0(\bar{b})$, then $\operatorname{SU}(\varphi(x,\bar{b}))<\omega$ and by Lemma \ref{prop-deltaH}, $\delta_H(\varphi(M,\bar{b}))<\infty$. So, since $\delta_M(H)=0$, we conclude that $\delta_M(\varphi(M,\bar{b}))=0$.\\

Otherwise, $(M,H)\models \theta_1(\bar{b})$, so $\operatorname{SU}(\varphi(x,\bar{b}))=\omega$. By Claim \ref{claim-main}, there are $\Le$-formulas $\psi'(\bar{t},\bar{y})$ and $\eta(x,\bar{t},\bar{y})$ (note that $\bar{z}=\emptyset$ since the $SU$-rank is $\omega$) such that $(M,H)\models\exists\bar{t}\in H^{|\bar{t}|}\psi'(\bar{t},\bar{b})$, and for any $\bar{h}\in H^{|\bar{t}|}$ with $M\models \psi'(\bar{h},\bar{b})$, we have 
 $\bar{h}\subseteq \operatorname{HB}(\bar{b})$ and
\[\operatorname{SU}(\eta(x,\bar{h},\bar{b})\triangle\varphi(x,\bar{b}))<\omega.\] 

As $\operatorname{SU}(\varphi(x,\bar{b}))=\omega$, the $\Le$-formula $\eta(x,\bar{h},\bar{b})$ has infinitely many realizations, so $\dim(\eta(x,\bar{h},\bar{b}))=1$ in the geometric theory $T$. Since $M$ is an ultraproduct of structures in a 1-dimensional asymptotic class, there is some $\mu>0$ such that $|\eta(M,\bar{h},\bar{b})|\geq \mu|M|$. Hence $\delta_M(\eta(M,\bar{h},\bar{b}))=1$. By the previous argument, since $\operatorname{SU}(\eta(x,\bar{h},\bar{b})\triangle\varphi(x,\bar{b}))<\omega$, we have $\delta_M(\eta(M,\bar{h},\bar{b})\triangle\varphi(M,\bar{b}))=0$. Notice that $\eta(M,\bar{h},\bar{b})$ is contained in the union $[\eta(M,\bar{h},\bar{b})\cap\varphi(M,\bar{b})]\cup [\eta(M,\bar{h},\bar{b})\triangle\varphi(M,\bar{b})]$, so $\delta_M(\eta(M,\bar{h},\bar{b})\cap\varphi(M,\bar{b}))=1$ and we conclude that $\delta_M(\varphi(M,\bar{b}))=1$.  
\end{proof}

By \cite[Corollary 2.1]{ZouDifference} (using induction on $|\bar{x}|$ and projections), Proposition \ref{0-1dim} can be generalized to formulas $\varphi(\bar{x},\bar{y})$ in several variables:


\begin{corollary}
Let $\varphi(\bar{x},\bar{y})$ be an $\mathcal{L}_H$-formula without parameters. For any $\bar{b}\in M^{|\bar{y}|}$, we have $$\delta_M(\varphi(M^{|\bar{x}|},\bar{b}))\in\{0,\cdots,|\bar{x}|\}.$$ And $\delta_M(\varphi(M^{|\bar{x}|},\bar{b}))=k$ if and only if $SU(\varphi(\bar{x},\bar{b}))=\omega\cdot k+\ell$ for some $\ell\in\mathbb{N}$.

Moreover, $$\{\bar{y}\in M^{|\bar{y}|}:\delta_M(\varphi(M,\bar{y}))=i\}$$ is $\mathcal{L}_H$-definable without parameters for each $i\in\{0,1,\cdots,|\bar{x}|\}$.
In particular, the coarse dimension $\delta_M$ is definable and additive.
\end{corollary}

Let $\bar{a}$ be a tuple in $(M,H)$ and $A$ be a countable subset of $M$. Suppose $\operatorname{SU}(\bar{a}/A)=\omega\cdot k+n$ for some $0\leq k\leq |\bar{a}|$ and $n\in\mathbb{N}$. Recall from Definition \ref{def-small} that $\ldim(\bar{a}/A)=k$ is the \emph{large dimension} of $\bar{a}$ over $A$.

\begin{theorem}
For any tuple $\bar{a}$ in $(M,H)$ and any countable subset $A\subseteq M$, we have $$\ldim(\bar{a}/A)=\delta_{M}(\bar{a}/A).$$
\end{theorem}
\begin{proof}
By Lemma \ref{0-1dim}, we have $\ldim(a/A)=\delta_{M}(a/A)$ for $a$ an element. The Theorem follows by additivity of both $\ldim$ and $\delta_M$.
\end{proof}

\section{Measure and dimension}\label{section:measure}

In this section, we will define notions of dimension and measure for $H$-structures coming from theories that are \emph{measurable} of $SU$-rank one. We first recall the definition of a measurable structure. The formulation here is slightly different (but equivalent) to Definition 5.1 in \cite{MS}. See the remark after the definition.

\bd \label{MS-measurable} A structure $\mathcal{M}$ is \emph{measurable}\footnote{In some parts of the literature these structures are called \emph{MS-measurable} to distinguish this concept from other notions of measure.} if there is a function $h:\operatorname{Def}(M)\to \mathbb{N}\times \mathbb{R}^{>0}\cup\{(0,0)\}$ (where we denote $h(X)=(\operatorname{dim}(X),\operatorname{meas}(X))$) satisfying the following conditions:
\benum
\item If $X$ is finite, then $h(X)=(0,|X|)$.
\item (Finitely many values) For every formula $\varphi(\bar{x},\bar{y})$ with $|\bar{x}|=n$, $|\bar{y}|=m$ there is a finite set $D_\varphi\subseteq \mathbb{N}\times \mathbb{R}^{>0}\cup\{(0,0)\}$ so that for all $\bar{a}\in M^m$, $h(\varphi(M^n;\bar{a}))\in D_\varphi$.
\item (Definability condition) For every $\emptyset$-definable formula $\varphi(\bar{x},\bar{y})$ and each $(d,\mu)\in D_\varphi$, the set $\{\bar{a}'\in M^m:h(\varphi(M^n,\bar{a}'))=(d,\mu)\}$ is $\emptyset$-definable.
\item (Finite additivity) Suppose $X$ is a disjoint union of definable sets $X_1$ and $X_2$, and $h(X_1)=(d_1,\mu_1)$, $h(X_2)=(d_2,\mu_2)$. Then \[h(X)=\begin{cases} (d_1,\mu_1) &\text{ if $d_1>d_2$};\\ (d_2,\mu_2) &\text{ if $d_2>d_1$};\\
(d_1,\mu_1+\mu_2) &\text{ if $d_1=d_2.$}
\end{cases}\]
\item (Fubini property) Suppose $f:X\to Y$ is a definable surjection, $h(Y)=(d,\mu)$ and for all $\bar{b}\in Y,$ we have $h(f^{-1}(\bar{b}))=(e,\nu)$. Then $h(X)=(e+d,\mu\nu)$.
\eenum
\ed
\brmk We now compare the two definitions of measurability. Notice that we are including the empty set as a measurable set by taking $h(\emptyset)=(0,0)$. Also, condition (ii) in \cite{MS} regarding the bound with the $D$-rank is superfluous, as it is in \cite[Corollary 3.6]{EM} that it follows from the definitions. Finally, condition (iv) of Definition 5.1 in \cite{MS} implies both conditions (4) and (5) in Definition 5.1 of the present paper, and vice versa.
\ermk

Due to the Definability Condition, it is easy to check that if $\mathcal{M}\equiv \mathcal{N}$ and $\mathcal{M}$ is measurable then $\mathcal{N}$ is also measurable. The main examples of measurable structures correspond to ultraproducts of \emph{1-dimensional asymptotic classes}. These are proved to be supersimple of $SU$-rank 1 (cf. \cite[Lemma 4.1]{MS}), so they are also geometric structures.

In \cite{Zou}, the third author showed that for the ultraproducts of 1-dimensional asymptotic classes, the corresponding $H$-structure is also pseudofinite. Hence, it is natural to ask whether the construction of $H$-structures also preserves measurability.

A preliminary answer is that it is not possible to preserve measurability if the underlying theory is not trivial: a consequence of Definition \ref{MS-measurable} (cf. \cite[Corollary 3.6]{EM}) is that if $\mathcal{M}$ is a measurable structure, then for every definable set $X$ the $D$-rank of $X$ is bounded above by $\dim(X)$. In particular, since $\dim(X)\in\mathbb{N}$, this shows that $\operatorname{Th}(\mathcal{M})$ must be a supersimple theory of finite $SU$-rank. Therefore, since the corresponding $H$-structure of an $SU$-rank 1 non-trivial theory will have $SU$-rank $\omega$, it would not be measurable.

However, there is a variation of measurability, intended to include some examples of supersimple theories of infinite rank, the so-called \emph{generalized measurable structures}. In these structures, the dimension and measure of definable sets are taken to be in an arbitrary ordered semi-ring, rather than in $(\mathbb{N}\times \mathbb{R}^{>0})\cup\{(0,0)\}$. This definition and the main properties of these structures, as well as several interesting examples and connections with asymptotic classes, will appear in the work \cite{AMSW} by S. Anscombe, D. Macpherson, C. Steinhorn and D. Wolf, currently in preparation. Even though we will not include their definition here, the main results of this section will essentially show that $H$-structures of measurable geometric structures are generalized measurable.

In other words, we will restrict to the study of $H$-structures of $SU$-rank 1 measurable structures, and show that there is an appropriate notion of dimension and measure for definable sets (taking values in $(\mathbb{N}\times \mathbb{N})\times \mathbb{R}$) that satisfies the appropriate analogs of the conditions given in Definition \ref{MS-measurable}: they are uniformly definable in terms of the parameters of the formulas and satisfy both finite additivity and the Fubini property. Moreover, there is a strong connection between the $SU$-rank and the dimension of definable sets.

The main idea is that dimension will be defined using data associated to the $SU$-rank (which in turn is related to large dimension and the size of an $H$-basis), while the measure will be induced by the measure on $\Le$-definable sets in the underlying $SU$-rank 1 structure.

Since we are dealing with structures that are both measurable and geometric, there are two possible notions of dimension that we can use. The first one comes from the fact that the structure $\mathcal{M}$ is measurable, and it is defined in terms of the function $h:\operatorname{Def}(M)\to \mathbb{N}\times \mathbb{R}^{>0}$. The second one comes from the fact that $\mathcal{M}$ has $SU$-rank 1, so it is geometric and there is a dimension given in terms of algebraic independence. 

\begin{definition}\label{Def-coherent} We say that an $\mathcal{L}$-structure $M$ is \emph{coherent measurable} if:
\benum
\item $\operatorname{Th}(M)$ is nowhere trivial,
\item $\operatorname{Th}(M)$ is measurable, and
\item For every $\Le$-formula $\varphi(\bar{x},\bar{y})$ and every $\bar{b}\in M^{|\bar{y}|}$, the dimension of $\varphi(M^{|\bar{x}|},\bar{b})$ (as a measurable structure) coincides with $SU(\varphi(M^{|\bar{x}|},\bar{b}))$.
\eenum
We call $M$ a \emph{one-dimensional coherent measurable} if $h(M)=(1,\mu)$ for some $\mu>0$ and $\operatorname{Th}(M)$ is nowhere trivial.
\end{definition}

Note that a one-dimensional coherent measurable structure is of $SU$-rank 1. It is coherent measurable by additivity of both the dimension and the $\operatorname{SU}$-rank. Moreover, the theory is geometric, and the geometric dimension of a definable set $X$ coincides with both the $SU$-rank and the dimension coming from the measurable structure. We will abuse notation and denote both of them as $\dim(X)$. Later on we will define also a measure for $\Le_H$-formulas that will coincide with the original measure when restricted to $\Le$-definable sets. We will refer to the later notion of measure as $\Le$-measure, when the distinction is necessary. Finally, we will need the technical assumption of being nowhere trivial in order to apply the results from Section \ref{section:dimension}.

 \begin{example} Every ultraproduct of a one-dimensional asymptotic class is a measurable structure of $SU$-rank 1, and the dimension of definable sets (as a measurable structure) coincides with the $SU$-rank. Moreover, if the resulting structure is a group, then it is also nowhere trivial. Hence, every ultraproduct of a 1-dimensional asymptotic class of groups or rings is a one-dimensional coherent measurable structure. These includes key examples such as pseudofinite fields or infinite vector spaces over finite fields. (cf. Example \ref{Example-H-VS2})
\end{example}

\begin{example} If $T$ is the theory of a random graph, then $T$ is the theory of an ultraproduct of a 1-dimensional asymptotic class. It is measurable of $SU$-rank 1, but the theory is trivial. So the random graph is not coherent measurable.  
\end{example}

\begin{example} In the language $\Le=\{P_n:n<\omega\}$ where each $P_n$ is a unary predicate, we can consider the finite $\Le$-structures $M_k$ with universe $\{1,2,\ldots,k^2\}$, where we interpret $P_n(M_k)$ to be $\{kn+1,\ldots,k (n+1)\}$ if $n< k$, and $P_n(M_k)=\emptyset$ otherwise. Any infinite ultraproduct $M=\prod_\mathcal{U}M_k$ models the theory $T$ of countably many infinite disjoint predicates. This theory has $SU$-rank 1 (although it is $\omega$-stable of Morley rank 2) so the theory is geometric with dimension 1.

We can also see $M$ as a measurable structure: for instance, for each $n<\omega$ we can define $\dim(P_n(M))=1$ and $\operatorname{meas}(P_n(M))=1$. Finally, define $\dim(x=x)=2$ and $\operatorname{meas}(x=x)=1$. Note that the dimension corresponds precisely with the coarse pseudofinite dimension of $M$ with respect to the size of $P_1(M)$ and agrees with Morley rank. Also, for a definable set, its measure agrees with its Morley degree. However, since each of the predicates has positive dimension and the predicates are disjoint, it is not possible to see $M$ as a measurable structure in such a way that $\dim(M)=\operatorname{SU}(T)=1$. Hence, $M$ is not coherent measurable.
\end{example}

From now on we will work with $H$-structures of one-dimensional coherent measurable structures. First we define a semiring $\mathcal{R}$; later we assign to each $\mathcal{L}_H$-definable set a dimension and measure in $\mathcal{R}$.
\bd Let $\mathcal{R}$ be the set $\mathbb{N}\times\mathbb{N}^{>0}\times\mathbb{R}^{> 0}\cup\{(0,0,k):k\in\mathbb{N}\}$. We will define operations $\oplus,\odot$ and a relation $\leq$ that will endow $\mathcal{R}$ with an ordered semiring structure. First, define $\oplus$ on $\mathcal{R}$ as: 

\[(x_1,y_1,r_1)\oplus (x_2,y_2,r_2)=\begin{cases} (x_1,y_1,r_1) &\text{ if $(x_2,y_2)<_{\text{lex}}(x_1,y_1)$};\\ (x_2,y_2,r_2) &\text{ if $(x_1,y_1)<_{\text{lex}}(x_2,y_2)$};\\
(x_1,y_1,r_1+r_2) &\text{ if $(x_1,y_1)=(x_2,y_2).$}
\end{cases}\]
Note that $(0,0,0)$ is the neutral element for $\oplus$.\\

Define the product $\odot$ on $\mathcal{R}$ as: $(x_1,y_1,r_1)\odot(x_2,y_2,r_2)=(x_1+x_2,y_1+y_2,r_1r_2)$, for $r_1,r_2\neq 0$. And $(0,0,0)\odot(x,y,r)=(x,y,r)\cdot(0,0,0)=(0,0,0)$. The neutral element for $\odot$ is $(0,0,1)$. Finally, we take $\leq$ to be the lexicographic order on $\mathbb{N}\times\mathbb{N}^{>0}\times\mathbb{R}^{> 0}\cup\{(0,0,k):k\in\mathbb{N}\}$.\\

For a triple $(x,y,r)\in\mathcal{R}$, the pair $(x,y)$ is called the \emph{dimension}, and $r$ is called the \emph{measure}.
\ed

Given an $\mathcal{L}_H$-definable set $X$, we will approximate $X$ by an existential $H$-formula $\exists \bar{z}\in H^{|\bar{z}|}\varphi(\bar{x},\bar{z})$ and determine the measure of $X$ by the $\mathcal{L}$-measure of $\varphi(\bar{x},\bar{z})$. For this, we will use the construction in Claim \ref{claim-main}.

\begin{definition} \label{def-measuring-formula}
Let $(M,H)$ be a sufficiently saturated $H$-structure. Let $X\subseteq M^m$ be a set definable over $\bar{c}$ of $SU$-rank $\omega\cdot n+k$, and let $\varphi(\bar{x},\bar{z})$ be an $\Le$-formula with $|\bar{x}|=m,|\bar{z}|=k$. We say that $\varphi(\bar{x},\bar{z})$ is a \emph{measuring formula for} $X$, if 
\benum
\item[(i)]
$\varphi$ is defined over parameters in $\bar{c},\operatorname{HB}(\bar{c})$;
\item[(ii)]
$\varphi(\bar{a},H^k)$ is finite for any $\bar{a}\in M^{|\bar{x}|}$ and $\exists\bar{z}\in H^k\varphi(\bar{x},\bar{z})$ defines a set $Y$, such that $\operatorname{SU}(X\triangle Y/\bar{c},\operatorname{HB}(\bar{c}))<\omega\cdot n+k$;
\item[(iii)]
For all $\bar{a}\in X$ with $\operatorname{SU}(\bar{a}/\bar{c},\operatorname{HB}(\bar{c}))=\omega\cdot n+k$ and $\bar{d}\in H^k$, we have $M\models \varphi(\bar{a},\bar{d})$ if and only if
$\bar{d}$ enumerates $\operatorname{HB}(\bar{a}/\bar{c},\operatorname{HB}(\bar{c}))$;
\item[(iv)]
$\dim(\varphi)=n+k$ and for all $\bar{e},\bar{b}$ with $M\models \varphi(\bar{e},\bar{b})$ and $\dim(\bar{e},\bar{b}/\bar{c},\operatorname{HB}(\bar{c}))=n+k$, there are $\bar{a},\bar{d}$ such that 
$$\tp_{\mathcal{L}}(\bar{a},\bar{d}/\bar{c},\operatorname{HB}(\bar{c}))=\tp_{\mathcal{L}}(\bar{e},\bar{b}/\bar{c},\operatorname{HB}(\bar{c})),$$ $\operatorname{SU}(\bar{a}/\bar{c},\operatorname{HB}(\bar{c}))=\omega\cdot n+k$ and $\bar{d}$ an enumeration of $\operatorname{HB}(\bar{a}/\bar{c},\operatorname{HB}(\bar{c}))$.
\eenum
\end{definition}

\brmk
By Claim \ref{claim-main}, for any $\Le_H$-definable set $X$ there is a measuring formula for $X$.
\ermk
\begin{definition}\label{def-dimMeas} Suppose $M$ is a one-dimensional coherent measurable structure, and $(M,H)$ is a sufficiently saturated $H$-structure. Let $X$ be a definable set defined over $\bar{c}$ with $SU$-rank $\omega\cdot n+k$ and let $\varphi(\bar{x},\bar{z})$ be a measuring formula for $X$. Suppose $\varphi$ has measure $\mu$. Then, we call the pair $(n,k)$ the \emph{dimension} of $X$ and we call $\mu/|\mathbb{S}_k|$ the \emph{measure of $X$}. The triple $(n,k,\mu/|\mathbb{S}_k|)\in\mathcal{R}$ will be called the \emph{dimension and measure} of $X$.
\end{definition}

We will prove that the dimension and measure for a definable set $X$ is well-defined, i.e., that the dimension and measure of $X$ do not depend on the particular choice of a measuring formula.

\begin{lemma}\label{lem-equiMeasure}
Suppose $M$ is a one-dimensional coherent measurable structure, $(M,H)$ is a sufficiently saturated $H$-structure and $X$ is a definable set. Suppose $\varphi_1(\bar{x},\bar{z})$ and $\varphi_2(\bar{x},\bar{z})$ are two measuring formulas of $X$. Then $\varphi_1$ and $\varphi_2$ have the same dimension and measure as $\mathcal{L}$-formulas.
\end{lemma}
\begin{proof}
Suppose $X$ is defined over $\bar{c}$ and $\operatorname{SU}(X/\bar{c},\operatorname{HB}(\bar{c}))=\omega\cdot n+k$. 
By the definition of measuring formulas, 
$\dim(\varphi_1)=\dim(\varphi_2)=n+k$.
It suffices to show that $\dim(\varphi_1\triangle \varphi_2)<n+k$, as this will imply $\mu(\varphi_1)=\mu(\varphi_1\land\varphi_2)=\mu(\varphi_2)$ by the finite additivity property. Suppose, in order to obtain a contradiction, that  $\dim(\varphi_1\triangle \varphi_2)=n+k$, then there is a tuple $\bar{e},\bar{b}$ such that $M\models \varphi_1(\bar{e},\bar{b})\triangle \varphi_2(\bar{e},\bar{b})$ 
and $\dim(\bar{e},\bar{b}/\bar{c},\operatorname{HB}(\bar{c}))=n+k$.
We may assume that $M\models \varphi_1(\bar{e},\bar{b})$ and $M\models \neg\varphi_2(\bar{e},\bar{b})$.
Since $\varphi_1$ is a measuring formula, by the clause (iv) of the definition there are $\bar{a},\bar{d}$ such that 
$$\tp_{\mathcal{L}}(\bar{a},\bar{d}/\bar{c},\operatorname{HB}(\bar{c}))=\tp_{\mathcal{L}}(\bar{e},\bar{b}/\bar{c},\operatorname{HB}(\bar{c})) \text{ and } \operatorname{SU}(\bar{a}/\bar{c},\operatorname{HB}(\bar{c}))=\omega\cdot n+k,$$ where $\bar{d}$ an enumeration of $\operatorname{HB}(\bar{a}/\bar{c},\operatorname{HB}(\bar{c}))$, and thus it is a tuple in $H^k$. Let $Y_1$ be the set defined by $\exists \bar{z}\in H^k\,\varphi_1(\bar{x},\bar{z})$. Then, by clause (ii) of Definition \ref{def-measuring-formula}, we have $\operatorname{SU}(X\triangle Y_1/\bar{c},\operatorname{HB}(\bar{c}))<\omega\cdot n+k$, and thus $\bar{a}\in X$.
Since $\varphi_2$ is also a measuring formula for $X$, and $\bar{d}$ is an enumeration of $\operatorname{HB}(\bar{a}/\bar{c},\operatorname{HB}(\bar{c}))$, by clause (iii) of the definition we must have $M\models \varphi_2(\bar{a},\bar{d})$.
Since $\tp_{\mathcal{L}}(\bar{a},\bar{d}/\bar{c},\operatorname{HB}(\bar{c}))=\tp_{\mathcal{L}}(\bar{e},\bar{b}/\bar{c},\operatorname{HB}(\bar{c}))$, 
we have $M\models \varphi_2(\bar{e},\bar{b})$, a contradiction.
\end{proof}

We now show some examples of dimensions and measures.

\begin{example} \label{Example-H-VS2} (Example \ref{example-H-vs} revisited)
Let $(\mathbb{V},H)=\prod_{\mathcal{U}}(\mathbb{F}_p^{2n},H_n)$ be an infinite ultraproduct of the finite vector spaces $\mathbb{F}_p^{2n}$ together with a predicate $H_n$ for $n$ linearly independent vectors. Note that the reduct $\mathbb{V}$ is nowhere trivial as it is a group. It is also strongly minimal and pseudofinite, so by \cite[Lemma 2.4]{MS} it is the ultraproduct of a 1-dimensional asymptotic class, and hence it is measurable of $SU$-rank 1. By the arguments in \cite[Section 2]{Pi-str}, it is not difficult to see that measure of a definable set in $\mathbb{V}$ (as a measurable structure) coincides with its Morley degree. The key point in this argument is that $\mathbb{V}$ has the definable multiplicity property.

Now, since $H$ is an infinite collection of linearly independent elements from $\mathbb{V}$ and the dimension of the quotient $\mathbb{V}/\operatorname{span}(H)$ is again infinite, the pair $(\mathbb{V},H)$ is an $H$-structure. From now on in this example, let us assume $p$ is a prime different from $2$ and $3$. Consider the definable sets $X=H+H=\{h_1+h_2:h_1,h_2\in H\}$ and $Y=H+2H=\{h_1+2h_2:h_1,h_2\in H\}$.
Clearly $\operatorname{SU}(H+H)=2$ and $\operatorname{SU}(H+2H)=2$. Notice that $|H+2H|=|H\times H|$ (as non-standard sizes), while $2\cdot |H+H|=|H\times H|+|H|$ 
because the definable function that maps $(h_1,h_2)$ to $h_1+h_2$ is generically 2-to-1. 

Let $\varphi_1(x,z_1,z_2)$ be the formula $x=z_1+z_2$, so $X$ is defined by 
$\exists (z_1,z_2)\in H^2 \varphi_1(x,z_1,z_2)$ and $\varphi_1(x,z_1,z_2)$
is a measuring formula for $X$. Note that the $\mathcal{L}$-measure of $\varphi_1(x,z_1,z_2)$ is $1$ (it has Morley degree $1$) and the measure of $X$ is $1/(2!)=1/2$.

Let $\varphi_2(x,z_1,z_2)$ be the formula $(x=z_1+2z_2\vee x=2z_1+z_2)$. The set $Y$ is defined by 
$\exists (z_1,z_2)\in H^2 \varphi_2(x,z_1,z_2)$ and $\varphi_2(x,z_1,z_2)$
is a measuring formula for $Y$. Note that we had to use $\varphi_2(x,z_1,z_2)$ instead of the formula $x=z_1+2z_2$ to guarantee that clause (iii) holds in the definition of a measuring formula, that is, the formula should be invariant under permutations of $\bar{z}=(z_1,z_2)$. Since the Morley degree of $\varphi_2(x,z_1,z_2)$ is $2$, the $\mathcal{L}$-measure of $\varphi_2(x,z_1,z_2)$ is also $2$ and so the measure of $Y$ is $2/(2!)=1$.
 
\end{example}

\begin{lemma}\label{lem-defMeas}
Suppose $M$ is a one-dimensional coherent measurable structure and that $(M,H)$ is a sufficiently saturated $H$-structure. Then the dimension and measure defined in \emph{Definition} \ref{def-dimMeas} are $\Le_H$-definable. 
\end{lemma}
\begin{proof}
By Theorem \ref{thm-def-SU} and Claim \ref{claim-main},  given an $\mathcal{L}_H$-formula $\varphi(\bar{x};\bar{y})$ there are finitely many formulas $$\{\psi_i(\bar{y}),\psi_i'(\bar{v}_i,\bar{y}),\eta_i(\bar{x},\bar{z}_i,\bar{v}_i,\bar{y}):i\leq N\}$$ and a finite set of pairs $D_{\varphi}=\{(n_i,k_i):i\leq N\}\subseteq \mathbb{N}\times\mathbb{N}$
such that 
\begin{itemize}
\item the collection $\{\psi_i(\bar{y}):i\leq N\}$ forms a cover of $M^{|\bar{y}|}$;
\item
$\psi_i(\bar{y})=\exists\bar{v}_i\in H^{|\bar{v}_i|}\psi_i'(\bar{v}_i,\bar{y})$;
\item for every $i\leq N$ and $\bar{a}\in M^{|\bar{y}|}$, if $(M,H)\models \psi_i(\bar{a})$ holds then the set $X_{\bar{a}}$ defined by $\varphi(\bar{x};\bar{a})$ has $SU$-rank $\omega\cdot n_i+k_i$ over $\bar{a},\operatorname{HB}(\bar{a})$;
\item for every $i\leq N$ and  $\bar {a}\in M^{|\bar{y}|}$, if $\bar{h}_i\in H^{|\bar{v}_i|}$ satisfies $(M,H)\models \psi_i'(\bar{h}_i,\bar{a})$, then $\bar{h}_i\subseteq \operatorname{HB}(\bar{a})$ and $\eta_i(\bar{x},\bar{z}_i;\bar{h}_i,\bar{a})$ is a measuring formula for $X_{\bar{a}}$.
\end{itemize}

Since $M$ is a measurable structure, for every $i\leq N$ there is a finite set  $E_i\subseteq \mathbb{N}\times\mathbb{R}^{>0}\cup\{(0,0)\}, i\leq N$ and $\Le$-formulas $\{\zeta_{i,(d,\mu)}(\bar{v}_i,\bar{y}):i\leq N, (d,\mu)\in E_i\}$ such that $M\models \zeta_{i,(d,\mu)}(\bar{b}_i,\bar{a})$ implies that the $\mathcal{L}$-definable set defined by $\eta_i(\bar{x},\bar{z}_i;\bar{b}_i,\bar{a})$ has dimension $d$ measure $\mu$. Thus, we know that $X_{\bar{a}}$ has dimension and measure $(n,k,\mu/|\mathbb{S}_{k}|)$ if and only if 
$$(M,H)\models \bigvee_{\{i\leq N:(n_i,k_i)=(n,k)\}}\exists\bar{v}_i\in H^{|\bar{v}_i|}\left(\psi_i'(\bar{v}_i,\bar{a})\land \zeta_{i,(n+k,\mu)}(\bar{v}_i,\bar{a})\right).\qedhere$$ 
\end{proof}

\begin{lemma}\label{lem-finAdd}
Suppose $M$ is a one-dimensional coherent measurable structure and that $(M,H)$ is a sufficiently saturated $H$-structure. Then the dimension and measure are finitely additive, that is, if $X_1,X_2\subseteq M^m$ are disjoint definable sets whose dimensions and measures are $(n_1,k_1,\mu_1)$, $(n_2,k_2,\mu_2)$ respectively, then
the dimension and measure of $X_1\cup X_2$ is $(n_1,k_1,\mu_1)\oplus(n_2,k_2,\mu_2)$.
\end{lemma}
\begin{proof}

Let $\varphi_1(\bar{x},\bar{z}_1)$ and $\varphi_2(\bar{x},\bar{z}_2)$ be the measuring formulas for $X_1$ and $X_2$ respectively.
Suppose $(n_1,k_1)>(n_2,k_2)$. Then it is easy to see that $\varphi_1(\bar{x},\bar{z}_1)$ is also a measuring formula for the set $X_1\cup X_2$.
Therefore, the dimension and measure for $X_1\cup X_2$ is $(n_1,k_1,\mu_1)$, which equals $(n_1,k_1,\mu_1)\oplus (n_2,k_2,\mu_2)$.
The case $(n_1,k_1)<(n_2,k_2)$ is analogous.

It remains to consider the case $(n_1,k_1)=(n_2,k_2)$. We can then write $n$ for $n_1=n_2$ and $k$ for $k_1=k_2$. Moreover, since $\bar{z}_1$ and $\bar{z}_2$ have the same length, we can rename both of them simply as $\bar{z}$. It is easy to check that $\varphi_1(\bar{x},\bar{z})\vee \varphi_2(\bar{x},\bar{z})$ is a measuring formula for $X_1 \cup X_2$.

\textbf{Claim:} $\dim(\varphi_1(\bar{x},\bar{z})\wedge \varphi_2(\bar{x},\bar{z}))<n+k$.

We may assume that $X_1$ and $X_2$ are defined over $\bar{c}$. 
Suppose, towards a contradiction, that there are $\bar{e},\bar{b}$ such that $\varphi_1(\bar{e},\bar{b})\wedge \varphi_2(\bar{e},\bar{b})$ holds 
and that $\dim(\bar{e},\bar{b}/\bar{c},\operatorname{HB}(\bar{c}))=n+k$.
By the definition of measuring formulas, there are $\bar{a},\bar{d}$ such that $\bar{d}\in H^{k}$, $\operatorname{SU}(\bar{a}/\bar{c},\operatorname{HB}(\bar{c}))=\omega\cdot n+k$ and $\tp_{\mathcal{L}}(\bar{e},\bar{b}/\bar{c},\operatorname{HB}(\bar{c})=\tp_{\mathcal{L}}(\bar{a},\bar{d}/\bar{c},\operatorname{HB}(\bar{c}))$.
In particular $\varphi_1(\bar{a},\bar{d})\land\varphi_2(\bar{a},\bar{d})$ holds.
Therefore, $\bar{a}\in X_1$ and $\bar{a}\in X_2$, contradicting that $X_1\cap X_2=\emptyset$.

We know that both $\varphi_1$ and $\varphi_2$ have dimension $n+k$, the $\Le$-measure of $\varphi_1$ is $\mu_1\cdot|\mathbb{S}_{k}|$
and the $\Le$-measure of  $\varphi_2$ is $\mu_2\cdot|\mathbb{S}_{k}|$.
Since $\dim(\varphi_1(\bar{x},\bar{z})\wedge \varphi_2(\bar{x},\bar{z}))<n+k$, we get that 
the $\Le$-measure of $\varphi_1\vee\varphi_2$ is $(\mu_1+\mu_2)\cdot|\mathbb{S}_{k}|$.
Therefore, the measure of $X_1\cup X_2$ is $\mu_1+\mu_2$ as desired.
\end{proof}

\begin{lemma}\label{lem-funct}
Let $f:X\to Y$ be an $\mathcal{L}_H$-definable surjective function and suppose $\varphi(\bar{x},\bar{y},\bar{z})$ is a measuring formula for $G(f)$, the graph of $f$. Then:
\benum
\item[(a)] The formula $\varphi'(\bar{x},\bar{y},\bar{z}):=\varphi(\bar{x},\bar{y},\bar{z})\land \forall\bar{y}'\left(\varphi(\bar{x},\bar{y}',\bar{z})\to\bar{y}=\bar{y}'\right)$ is also a measuring formula for $G(f)$.
\item[(b)] Assume that $\rho(\bar{x},\bar{z})$ is a measuring formula for $X$. Then the formula
$\varphi''(\bar{x},\bar{y},\bar{z}):=\varphi'(\bar{x},\bar{y},\bar{z})\wedge \rho(\bar{x},\bar{z})$ is also a measuring formula for $G(f)$ and $\exists \bar{y}\,\varphi''(\bar{x},\bar{y},\bar{z})$ is a measuring formula for $X$.
\eenum
\end{lemma}

\begin{proof} To simplify the notation along this proof, we will assume that $f$ is defined over $\emptyset$. In the general case one can rewrite the proof by relativizing all the notions to $\bar{c},\operatorname{HB}(\bar{c})$.

(a) Suppose $\operatorname{SU}(G(f))=\omega\cdot n+k$. 
Since $\varphi'(\bar{x},\bar{y},\bar{z})$ implies $\varphi(\bar{x},\bar{y},\bar{z})$, it is enough to show that if $\bar{a},\bar{b},\bar{d}$ are such that $\varphi(\bar{a},\bar{b},\bar{d})$ holds and $\dim(\bar{a},\bar{b},\bar{d})=n+k$, then for all $\bar{b}'$ such that $\varphi(\bar{a},\bar{b}',\bar{d})$ also holds, we have $\bar{b}=\bar{b}'$.

Since $\dim(\bar{a},\bar{b},\bar{d})=n+k$, by the definition of measuring formula, there are $\bar{a}_1,\bar{b}_1,\bar{d}_1$ with $\tp_{\mathcal{L}}(\bar{a}_1,\bar{b}_1,\bar{d}_1)=\tp_{\mathcal{L}}(\bar{a},\bar{b},\bar{d})$ such that ($\bar{a}_1,\bar{b}_1)\in G(f)$, $\operatorname{SU}(\bar{a}_1,\bar{b}_1)=\omega\cdot n+k$ and $\bar{d}_1$ is an enumeration of $\operatorname{HB}(\bar{a}_1,\bar{b}_1)$. 
Note that $\bar{b}_1\in\dcl_{\mathcal{L}_H}(\bar{a}_1)$, hence, by Proposition \ref{basicprop} (5), we have $\bar{b}_1\in\acl(\bar{a}_1,\operatorname{HB}(\bar{a}_1))$. Since $\bar{a}_1 \ind_{\operatorname{HB}(\bar{a}_1)}H$, this implies that $\bar{a}_1 \bar{b}_1\ind_{\operatorname{HB}(\bar{a}_1)}H$, we have 
$\operatorname{HB}(\bar{a}_1\bar{b}_1)=\operatorname{HB}(\bar{a}_1)$ and $\bar{d_1}$ is also an enumeration of $\operatorname{HB}(\bar{a}_1)$.

\textbf{Claim:} $\bar{b}_1\in \dcl_{\mathcal{L}}(\bar{a}_1,\bar{d}_1)$.

Let $\bar{b}_1'$ be such that $\tp_{\mathcal{L}}(\bar{a}_1,\bar{b}_1,\bar{d}_1)
=\tp_{\mathcal{L}}(\bar{a}_1,\bar{b}_1',\bar{d}_1)$. 
As $\bar{b}_1'\in\acl(\bar{a}_1,\bar{d}_1)$ and $(\bar{a}_1,\bar{d}_1)$ is $H$-independent, 
so is $(\bar{a}_1,\bar{b}_1',\bar{d}_1)$.
Therefore, by Lemma \ref{fundlem1}, we have $$\tp_{\mathcal{L}_H}(\bar{a}_1,\bar{b}_1,\bar{d}_1)=\tp_{\mathcal{L}_H}(\bar{a}_1,\bar{b}_1',\bar{d}_1).$$
Since $\bar{b}_1\in\dcl_{\mathcal{L}_H}(\bar{a}_1)$, we get $\bar{b}_1=\bar{b}_1'$ and the claim holds. 



Let $g$ be the $\Le$-definable function over $\emptyset$ such that $\bar{b}_1=g(\bar{a}_1,\bar{d}_1)$. Since $\tp_{\mathcal{L}}(\bar{a}_1,\bar{b}_1,\bar{d}_1)=\tp_{\mathcal{L}}(\bar{a},\bar{b},\bar{d})$, we also have $\bar{b}=g(\bar{a},\bar{d})$ and $\bar{b}\in\dcl(\bar{a},\bar{d})$. Thus, $\dim(\bar{a},\bar{d})=\dim(\bar{a},\bar{b},\bar{d})=n+k$. Suppose $M\models \varphi(\bar{a},\bar{b}',\bar{d})$, then $\dim(\bar{a},\bar{b}',\bar{d})=n+k$ (as $\dim(\varphi)=n+k=\dim(\bar{a},\bar{d})$). Since $\varphi$ is a measuring formula for $G(f)$, there are $(\bar{a}_2,\bar{b}_2,\bar{d}_2)$ with $\tp_{\mathcal{L}}(\bar{a}_2,\bar{b}_2,\bar{d}_2)=\tp_{\mathcal{L}}(\bar{a},\bar{b}',\bar{d})$ such that $(\bar{a}_2,\bar{b}_2)\in G(f)$, $\operatorname{SU}(\bar{a}_2,\bar{b}_2)=\omega\cdot n+k$ and $\bar{d}_2$ is an enumeration of $\operatorname{HB}(\bar{a}_2,\bar{b}_2)$. Since $\tp_{\Le}(\bar{a}_2,\bar{d}_2)=\tp_{\Le}(\bar{a},\bar{d})=\tp_{\Le}(\bar{a}_1,\bar{d}_1)$ and both $(\bar{a}_1,\bar{d}_1)$ and $(\bar{a}_2,\bar{d}_2)$ are $H$-independent, by Lemma \ref{fundlem1}, $\tp_{\Le_H}(\bar{a}_2,\bar{d}_2)=\tp_{\Le_H}(\bar{a}_1,\bar{d}_1)$. Since $(\bar{a}_1,\bar{b}_1)\in G(f)$ and $(\bar{a}_2,\bar{b}_2)\in G(f)$, we get $\tp_{\Le_H}(\bar{a}_2,\bar{b}_2,\bar{d}_2)=\tp_{\Le_H}(\bar{a}_1,\bar{b}_1,\bar{d}_1)$. Now that $\bar{b}_1=g(\bar{a}_1,\bar{d_1})$, we must have $\bar{b}_2=g(\bar{a}_2,\bar{d}_2)$, hence $\bar{b}'=g(\bar{a},\bar{d})=\bar{b}$ as desired.

(b)\ The proof that $\varphi''(\bar{x},\bar{y},\bar{z})$ is a measuring formula for $G(f)$ is similar to the proof of (a), and we leave it to the reader. We will now prove that $\exists \bar{y}\,\varphi''(\bar{x},\bar{y},\bar{z})$ is a measuring formula for $X$. First choose $\bar{a}\in X$ with maximal $SU$-rank and let $\bar{b}:=f(\bar{a})$. Just as in the proof of part (a), we know that $\bar{a} \ind_{\operatorname{HB}(\bar{a})}H$, and this implies that $\bar{a} \bar{b}\ind_{\operatorname{HB}(\bar{a})}H$ and thus $\operatorname{HB}(\bar{a}\bar{b})=\operatorname{HB}(\bar{a})$, so $k=|\bar{z}|=|\operatorname{HB}(\bar{a})|$. So we can take a defining formula for $X$ with the same number of extra variables as the one we used for $G(f)$.

We will show that $\exists \bar{y}\,\varphi''(\bar{x},\bar{y},\bar{z})$ is a measuring formula for $X$, so we check properties (i) to (iv) in Definition \ref{def-measuring-formula}.

(i) Just to emphasize the use of parameters in this condition, let us assume that $G(f)$ is defined over parameters $\bar{c}$ (rather than over $\emptyset$ as it was assumed at the beginning of the proof), and so $X$ is also definable over $\bar{c}$. Since $\varphi(\bar{x},\bar{y},\bar{z})$ is a measuring formula for $G(f)$, it is defined over 
$\bar{c},\operatorname{HB}(\bar{c})$, and so the formula $\exists y\,\varphi''(\bar{x},\bar{y},\bar{z})$ is also defined over $\bar{c},\operatorname{HB}(\bar{c})$. 

(ii) Consider $\bar{a}$ and the collection of tuples $\bar{d}\in H^k$ such that $\exists \bar{y}\,\varphi'(\bar{a},\bar{y},\bar{d})\wedge \rho(\bar{x},\bar{d})$ holds. This collection is finite since it is a subset of $\rho(\bar{a}, H^k)$ which has finitely many solutions because $\rho(\bar{x},\bar{z})$ is a measuring formula for $X$. 

Let $X_1$ be the set defined by  $\exists \bar{z} \in H^k\exists \bar{y}\,\varphi''(\bar{x},\bar{y},\bar{z})$.
Assume that $\bar{a}\in X$ is such that $SU(\bar{a})=SU(X)$.
Let $\bar{b}=f(\bar{a})$, then $SU(\bar{a}\bar{b})=SU(X)=SU(G(f))$.
As in the proof of (a), we have that $\operatorname{HB}(\bar{a}\bar{b})=\operatorname{HB}(\bar{a})$. Let $\bar{d}\in H^k$ be an enumeration
of this common $H$-basis. Since $\varphi'(\bar{x},\bar{y},\bar{z})$ is a measuring formula for $G(f)$, 
$\varphi'(\bar{a},\bar{b},\bar{d})$ holds. Similarly, since  $\rho(\bar{x}, \bar{z})$ is a measuring formula
for $X$,  $\rho(\bar{a},\bar{d})$ also holds, thus $\bar{a}\in X_1$. 

Now assume that
$\bar{a}\in X_1$ is such that $SU(\bar{a})=SU(X_1)$ and choose $\bar{d}\in H^k$ such that $\exists \bar{y}\, \varphi''(\bar{a},\bar{y},\bar{d})$ holds.
Let $\bar{b}$ be such that 
$\varphi''(\bar{a},\bar{b},\bar{d})=\varphi'(\bar{a},\bar{b},\bar{d})\wedge \rho(\bar{a},\bar{d})$ holds. 
Since  $\rho(\bar{x},\bar{z})$ is a measuring formula for $X$ and the rank of the tuple $\bar{a}\in X_1$ coincides with $SU(X_1)=SU(X)$, then $\bar{a}\in X$.


(iii) Assume that $SU(X)=\omega \cdot n+k$ and choose $\bar{a}\in X$ such that $SU(\tp(\bar{a}))=\omega \cdot n+k$. Let $\bar{d}\in H^k$. We have to show that
$\exists \bar{y}\,\varphi''(\bar{a},\bar{y},\bar{d})$ holds if and only if $\bar{d}$ is an enumeration of $\operatorname{HB}(\bar{a})$. So assume that $\exists \bar{y}\,\varphi''(\bar{a},\bar{y},\bar{d})$ holds, which implies that $\rho(\bar{a},\bar{d})$ also holds and we have that $\bar{d}$ is an enumeration of  $\operatorname{HB}(\bar{a})$.
Conversely, let $\bar{d}$ be an enumeration of $\operatorname{HB}(\bar{a})$. Let $\bar{b}=f(\bar{a})$. Then
$\bar{d}$ is an enumeration of $\operatorname{HB}(\bar{a}\bar{b})=\operatorname{HB}(\bar{a})$.
Since $\varphi'$ is a measuring formula for $G(f)$, we have that $\varphi'(\bar{a},\bar{b},\bar{d})$ holds. Also, since $\rho$ is a measuring formula for $X$ then $\rho(\bar{a},\bar{d})$ also holds. Thus, the formula $\exists \bar{y}\,\varphi''(\bar{a},\bar{y},\bar{d})$ also holds.

(iv) Finally, choose $\bar{e},\bar{d}_0$ such that $\dim(\bar{e},\bar{d}_0)=n+k$ and $\exists \bar{y}\,\varphi''(\bar{e},\bar{y},\bar{d}_0)$ holds. Since $\rho(\bar{x},\bar{z})$ is a measuring formula for $X$, we can find
$\bar{a},\bar{d}$ such that
$$\tp_{\mathcal{L}}(\bar{a},\bar{d})=\tp_{\mathcal{L}}(\bar{e},\bar{d}_0)$$
with $\operatorname{SU}(\bar{a})=\omega\cdot n+k$ and $\bar{d}$ an enumeration of $\operatorname{HB}(\bar{a})$
as we wanted.
\end{proof}

\begin{theorem} \label{th-Fubini} Let $M$ be a one-dimensional coherent measurable structure and let $(M,H)$ be a sufficiently saturated $H$-structure. Then, the dimension and measure of definable sets in $(M,H)$ satisfy the Fubini condition:
If $f:X\to Y$ is an $\Le_H$-definable surjective function such that for any $\bar{b}\in Y$ the definable set $f^{-1}(\bar{b})$ has dimension and measure $(n_1,k_1,\mu_1)$, and $Y$ has dimension and measure $(n_2,k_2,\mu_2)$. Then $X$ has dimension and measure $$(n_1,k_1,\mu_1)\odot (n_2,k_2,\mu_2).$$
\end{theorem}
\begin{proof}
We may assume $f$ is defined over $\emptyset$. 
For every $\bar{b}\in Y$, let $X_{\bar{b}}:=f^{-1}(\bar{b})$. 
Consider the definable family $\{X_{\bar{b}}:\bar{b}\in Y\}$. 
We apply compactness, Theorem \ref{thm-def-SU} and Claim \ref{claim-main} to the parameter space $Y\cap\{\bar
{b}:\operatorname{SU}(\bar{b})=\omega\cdot n_2+k_2\}$ (which is a closed set in the space of $\Le_H$-types by continuity of $SU$-rank) to get
formulas 
$$\{\psi_i(\bar{y}),\psi_i'(\bar{y},\bar{v}),\eta_i'(\bar{x},\bar{z},\bar{y},\bar{v}):i\leq N\}$$ where $\{\eta_i'(\bar{x},\bar{z},\bar{y},\bar{v}):i\leq N\}$ are $\Le$-formulas, and the following conditions hold: 
\begin{enumerate}

\item
$|\bar{v}|=k_2$ and
$\psi_i'(\bar{b},H^{k_2})$ is finite for any $\bar{b}\in Y$. 
Also, if $\bar{b}\in Y$ and $\operatorname{SU}(\bar{b})=\omega\cdot n_2+k_2$ then there exists $i\leq N$ such that for any $\bar{h}\in H^{k_2}$,  the  formula $\psi_i'(\bar{b},\bar{h})$ holds if and only if $\bar{h}$ is an enumeration of $\operatorname{HB}(\bar{b})$. 
\item
$\psi_i(\bar{y}):=\exists\bar{v}\in H^{k_2}\psi_i'(\bar{y},\bar{v})$;
\item The set
$Y\cap \{\bar
{b}:\operatorname{SU}(\bar{b})=\omega\cdot n_2+k_2\}$ is covered by $\bigcup_{i\leq N}\psi_i(M^{|\bar{y}|})$. 

\item Whenever $\bar{b}\in Y$ and $\bar{h}\in H^{|\bar{v}|}$ is such that $\psi_i'(\bar{b},\bar{h})$ holds, then $\eta_i'(\bar{x},\bar{z};\bar{b},\bar{h})$ is a measuring formula of $X_{\bar{b}}$ with $|\bar{z}|=k_1$. In particular, $\eta_i'(\bar{x},\bar{z};\bar{b},\bar{h})$ has dimension $n_1+k_1$ over $\bar{b},\bar{h}$ and measure $|\mathbb{S}_{k_1}|\cdot\mu_1$. Furthermore, as in Lemma \ref{lem-almostExistential}, we have $$\forall\bar{x}\forall\bar{z}\left(\eta_i'(\bar{x},\bar{z};\bar{b},\bar{h})\to\bigvee_{j\leq N_i,\sigma\in\mathbb{S}_{k_2}}\Psi_{p_{i,j}}(\bar{x},\bar{z};\bar{b},\sigma(\bar{h}))\right)$$ where $p_{ij}$ are complete $\mathcal{L}_H$-types of $SU$-rank $\omega\cdot n_1+k_1$ over parameters $\bar{h},\bar{b}$.
\item Both formulas $\psi_i'(\bar{y},\bar{v})$ and $\eta_i'(\bar{x},\bar{z};\bar{y},\bar{v})$ are invariant under permutations of $\bar{v}$, and $\eta_i'(\bar{x},\bar{z};\bar{y},\bar{v})$ is invariant under permutations of $\bar{z}$.
\end{enumerate}

By finite additivity, we only need to show the Fubini Property for each of the sets $$Y_i:=Y\cap\psi_i(M^{|\bar{y}|})\setminus\left(\bigcup_{j<i}\psi_j(M^{|\bar{y}|})\right)$$ with $\operatorname{SU}(Y_i)=\omega\cdot n_2+k_2$ and $i\leq N$. 
Therefore, we may assume that $N=1$ and that there are formulas $\{\psi(\bar{y}),\psi'(\bar{y},\bar{v}),\eta'(\bar{x},\bar{z},\bar{y},\bar{v})\}$ and $\{\Psi_{p_j}(\bar{x},\bar{z};\bar{y},\bar{v}):j\leq N_0\}$ as claimed before. In particular, $\psi(M^{|\bar{y}|})$ contains all the generics of $Y$.

Let $\varphi'(\bar{y},\bar{v})$ be a measuring formula for $Y$ and let $\xi(\bar{y},\bar{v})$ be the $\Le$-formula saying that $\eta'(\bar{x},\bar{z};\bar{y},\bar{v})$ has dimension $n_1+k_1$ and measure $|\mathbb{S}_{k_1}|\cdot\mu_1$. Since $\varphi'(\bar{y},\bar{v})$ is a measuring formula for $Y$, it is invariant under permutations of $\bar{v}$. Also, since $\eta'$ is invariant under permutations of $\bar{v}$, we may take $\xi$ which is also invariant under permutations of $\bar{v}$. Define the $\mathcal{L}$-formula $$\varphi(\bar{y},\bar{v}):=\varphi'(\bar{y},\bar{v})\land \xi(\bar{y},\bar{v})\land\forall\bar{x}\forall\bar{z}\left(\eta'(\bar{x},\bar{z};\bar{y},\bar{v})\to\bigvee_{j\leq N_0,\sigma\in\mathbb{S}_{k_2}}\Psi_{p_{j}}(\bar{x},\bar{z};\bar{y},\sigma(\bar{v}))\right).$$ Note $\varphi$ is invariant under permutations of $\bar{v}$.
Since $\varphi'$ is a measuring formula for $Y$, if $\bar{b},\bar{h}$ is a tuple of dimension $n_2+k_2$ such that $\varphi'(\bar{b},\bar{h})$ holds,  there are $(\bar{b}',\bar{h}')\equiv_{\mathcal{L}}(\bar{b},\bar{h})$ 
such that $\bar{h}'$ is an enumeration of $\operatorname{HB}(\bar{b'})$ 
and $\bar{b}'\in Y$ with $\operatorname{SU}(\bar{b}')=\omega\cdot n_2+k_2$. Then $\psi(\bar{b}')$ holds (since $\psi(M^{|\bar{y}|})$ contains all the generics of $Y$) and there is $\bar{h}''\in H^{k_2}$ with $M\models \psi'(\bar{b}',\bar{h}'')$. Therefore, by condition (1), $\bar{h}''$ is a permutation of $\bar{h}'$. By the fact that $\psi'$ is invariant under permutations of $\bar{v}$, we get $\psi'(\bar{b}',\bar{h}')$ holds.
Thus, $\varphi(\bar{b}',\bar{h}')$ holds by condition (4), and so we conclude that $\varphi(\bar{b},\bar{h})$ holds because $\varphi(\bar{y},\bar{v})$ is an $\Le$-formula.
Hence, $\dim(\varphi(\bar{y},\bar{v})\triangle \varphi'(\bar{y},\bar{v}))<n_2+k_2$ and thus
$\varphi(\bar{y},\bar{v})$ is also a measuring formula for $Y$. 

Let $\displaystyle{\tau(\bar{x},\bar{y},\bar{z},\bar{v}):=\eta'(\bar{x},\bar{z};\bar{y},\bar{v})\land\varphi(\bar{y},\bar{v})}$ and  $\displaystyle{\tau'(\bar{x},\bar{y},\bar{z},\bar{v}):=\bigvee_{\sigma\in \mathbb{S}_{k_1+k_2}}\tau(\bar{x},\bar{y},\sigma(\bar{z},\bar{v})).}$ Notice that both $\tau(\bar{x},\bar{y},\bar{z},\bar{v})$ and $\tau'(\bar{x},\bar{y},\bar{z},\bar{v})$ are $\Le$-formulas.

\begin{claim} \label{tau-prime-measuring} The formula $\tau'(\bar{x},\bar{y},\bar{z},\bar{v})$ is a measuring formula for the definable set $G(f):=\{(\bar{a},f(\bar{a})):\bar{a}\in X\}$.
\end{claim}

We will first finish the proof of Theorem \ref{th-Fubini} assuming the Claim \ref{tau-prime-measuring}.

By the additivity of $\operatorname{SU}$-rank in $T^{ind}$ (Lemma \ref{lem-addSU}), $\operatorname{SU}(G(f))=\omega\cdot (n_1+n_2)+(k_1+k_2)$, and by the claim above the formula $\tau'(\bar{x},\bar{y},\bar{z},\bar{v})$ has dimension $(n_1+n_2)+(k_1+k_2)$. Suppose it has measure $\mu$, then the measure of $G(f)$ is $\mu/|\mathbb{S}_{k_1+k_2}|$. 
By Lemma \ref{lem-funct} part (a), $$\tilde{\tau}(\bar{x},\bar{y},\bar{z},\bar{v}):=\tau'(\bar{x},\bar{y},\bar{z},\bar{v})\land\forall\bar{y}'\left(\tau'(\bar{x},\bar{y}',\bar{z},\bar{v})\to\bar{y}=\bar{y}'\right)$$ is also a measuring formula for $G(f)$. 
Hence the measure of $\tilde{\tau}$ is also $\mu$.



Assume that $\rho(\bar{x},\bar{z}')$ is a measuring formula for $X$. Then by part (b) of 
Lemma \ref{lem-funct}, $|\bar{z}'|=|(\bar{z},\bar{v})|$ and we may assume $\bar{z}'=(\bar{z},\bar{v})$. Moreover, $\tau''(\bar{x},\bar{y},\bar{z},\bar{v}):=\tilde{\tau}(\bar{x},\bar{y},\bar{z},\bar{v})\wedge \rho(\bar{x},\bar{z},\bar{v})$ is also a measuring formula for $G(f)$, again with measure $\mu$, as $\tau''$ and $\tau'$ are measuring formulas for the same definable set. And $\exists \bar{y}\,\tau''(\bar{x},\bar{y},\bar{z},\bar{v})$ is a measuring formula for $X$, also with measure $\mu$, since \[M\models\forall\bar{x}\forall\bar{z}\forall\bar{v}\forall\bar{y}\forall\bar{y}'\left(\tau''(\bar{x},\bar{y},\bar{z},\bar{v})\land\tau''(\bar{x},\bar{y'},\bar{z},\bar{v})\to \bar{y}=\bar{y}' \right),\] by the definition of $\tilde{\tau}$. Therefore, the dimension and measure of $X$ is \[(n_1+n_2,k_1+k_2,\mu/|\mathbb{S}_{k_1+k_2}|).\]

To complete the proof, we only need to show that $$\frac{\mu}{|\mathbb{S}_{k_1+k_2}|}=\mu_1 \mu_2.$$
Note that by assumption, $\varphi(\bar{y},\bar{v})$ has dimension $n_2+k_2$ and measure $\mu_2\cdot |\mathbb{S}_{k_2}|$
and $\eta'(\bar{x},\bar{z};\bar{b},\bar{h})$ has dimension $n_1+k_1$ and measure $\mu_1\cdot |\mathbb{S}_{k_1}|$ over any $\bar{b},\bar{h}$
such that $\varphi(\bar{b},\bar{h})$ holds. 
Therefore, by the Fubini condition for $\Le$-formulas in $M$, we have that $\tau(\bar{x},\bar{y},\bar{z},\bar{v})$ has dimension $(n_1+n_2)+(k_1+k_2)$ and measure $|\mathbb{S}_{k_1}|\cdot|\mathbb{S}_{k_2}|\cdot\mu_1\mu_2$.

Recall that $\tau'(\bar{x},\bar{y},\bar{z},\bar{v})=\bigvee_{\sigma\in \mathbb{S}_{k_1+k_2}}\tau(\bar{x},\bar{y},\sigma(\bar{z},\bar{v}))$. 
By construction, $\tau$ is invariant under permutations in each of the tuples $\bar{z}$ and $\bar{v}$. That is, for $\sigma_1\in \mathbb{S}_{k_1}$ and $\sigma_2\in \mathbb{S}_{k_2}$ we have that 
$\tau(\bar{x},\bar{y},\sigma_1(\bar{z}),\sigma_2(\bar{v}))$ holds if and only if $\tau(\bar{x},\bar{y},\bar{z},\bar{v})$ holds.
Consider the collection $D$ of definable sets given by the realizations of formulas of the form $\tau(\bar{x},\bar{y},\sigma(\bar{z},\bar{v}))$ with $\sigma\in \mathbb{S}_{k_1+k_2}$. Now, we consider the action $\Omega$ from $\mathbb{S}_{k_1+k_2}\times D$ to $D$ given by $\sigma\boldsymbol{\cdot} \tau(\bar{x},\bar{y},\sigma_1(\bar{z},\bar{v}))=\tau(\bar{x},\bar{y},\sigma \circ \sigma_1(\bar{z},\bar{v}))$. \\

\textbf{Claim:} For all $\bar{a},\bar{b},\bar{d},\bar{e}$ with $\dim(\bar{a},\bar{b},\bar{d},\bar{e})=n_1+n_2+k_1+k_2$, for all $\sigma,\sigma'\in\mathbb{S}_{k_1+k_2}$, if $M\models \tau(\bar{a},\bar{b},\sigma(\bar{d},\bar{e}))\land\tau(\bar{a},\bar{b},\sigma'(\bar{d},\bar{e}))$, then  $\sigma'\circ\sigma^{-1}\in 
\mathbb{S}_{k_1}\times\mathbb{S}_{k_2}$.\\

Suppose $M\models \tau(\bar{a},\bar{b},\sigma(\bar{d},\bar{e}))$, by Claim \ref{tau-prime-measuring} and by the definition of measuring formula, there are $(\bar{a}',\bar{b}',\bar{d}',\bar{e}')$ with the same $\Le$-type of $(\bar{a},\bar{b},\bar{d},\bar{e})$ such that $(\bar{a}',\bar{b}')\in G(f)$, $\operatorname{SU}(\bar{a}',\bar{b}')=\omega\cdot(n_1+n_2)+k_1+k_2$ and $(\bar{d}',\bar{e}')$ is an enumeration of $\operatorname{HB}(\bar{a}',\bar{b}')$. We write $\sigma(\bar{d}',\bar{e}')$ as $(\bar{h}_1,\bar{h}_2)$ where $|\bar{h}_1|=k_1$ and $|\bar{h}_2|=k_2$. Then $M\models\varphi'(\bar{b}',\bar{h}_2)$ and $M\models \eta'(\bar{a}',\bar{h}_1,\bar{b}',\bar{h}_2)$. Since $\operatorname{SU}(\bar{a}'/\bar{b}',\bar{h}_2)\leq \omega\cdot n_1+k_1$ and $\bar{h}_2\subseteq\operatorname{HB}(\bar{b}')$ (as $\varphi'(\bar{b}',H^{k_2})$ is a finite set by the definition of measuring formula), and $\operatorname{SU}(\bar{b}')\leq \omega\cdot n_2+k_2$, we must have $\operatorname{SU}(\bar{b}')= \omega\cdot n_2+k_2$ and $\bar{h}_2$ is an enumeration of $\operatorname{HB}(\bar{b}')$ and $\operatorname{SU}(\bar{a}'/\bar{b}',\bar{h}_2)= \omega\cdot n_1+k_1$ and $\bar{h}_1$ is an enumeration of $\operatorname{HB}(\bar{a}'/\bar{b}',\operatorname{HB}(\bar{b}'))$. Suppose we also have $M\models \tau(\bar{a},\bar{b},\sigma'(\bar{d},\bar{e}))$. Since $\tau$ is an $\Le$-formula, we obtain $M\models \tau(\bar{a}',\bar{b}',\sigma'(\bar{d}',\bar{e}'))$. Write $\sigma'(\bar{d}',\bar{e}')$ as $(\bar{h}_1',
\bar{h}_2')$ with $|\bar{h}_1'|=k_1$ and $\bar{h}_2'=k_2$. By the same argument as before, $\bar{h}_2'$ is an enumeration of $\operatorname{HB}(\bar{b}')$ and $\bar{h}_1'$ is an enumeration of $\operatorname{HB}(\bar{a}'/\bar{b}',\operatorname{HB}(\bar{b}'))$. Since $\bar{h}_2$ and $\bar{h}_1$ are also enumerations of the same $H$-basis respectively, there are $\sigma_1\in\mathbb{S}_{k_1}$ and $\sigma_2\in\mathbb{S}_{k_2}$ such that $\sigma_1(\bar{h}_1)=\bar{h}_1'$ and $\sigma_2(\bar{h}_2)=\bar{h}_2'$. Therefore $(\sigma_1,\sigma_2)\circ\sigma(\bar{d}',\bar{e}')=\sigma'(\bar{d}',\bar{e}')$. As $(\bar{d}',\bar{e}')$ is a tuple of distinct elements, we must have $(\sigma_1,\sigma_2)\circ\sigma=\sigma'$ and $\sigma'\circ\sigma^{-1}=(\sigma_1,\sigma_2)\in\mathbb{S}_{k_1}\times\mathbb{S}_{k_2}$.\\

Therefore, by the claim above, the sets defined by the formulas $\tau(\bar{x},\bar{y},\sigma(\bar{z},\bar{v}))$ and $\tau(\bar{x},\bar{y},\sigma'(\bar{z},\bar{v}))$ intersect generically if and only if they are equal, and since $\tau$ is invariant under permutations of $\bar{z}$ and of $\bar{v}$, this occurs precisely when $\sigma'\circ\sigma^{-1}\in\mathbb{S}_{k_1}\times \mathbb{S}_{k_2}$. For all other pairs $(\sigma,\sigma')$, the sets defined are generically disjoint.

Hence, the stabilizer $\text{Stab}(\Omega)$ of a set in the collection $D$ is isomorphic to $\mathbb{S}_{k_1}\times \mathbb{S}_{k_2}$

and thus by the Orbit-Stabilizer Theorem there are $|\mathbb{S}_{k_1+k_2}/\text{Stab}(\Omega)|=\frac{(k_1+k_2)!}{k_1!k_2!}=\binom{k_1+k_2}{k_1}$ many generically disjoint sets, all of which have measure $\mu_1\mu_2\cdot|\mathbb{S}_{k_1}||\mathbb{S}_{k_2}|$.


Hence, $$\mu=\binom{k_1+k_2}{k_1}\cdot |\mathbb{S}_{k_1}|\cdot |\mathbb{S}_{k_2}|\cdot\mu_1\mu_2=(k_1+k_2)!\cdot\mu_1\mu_2=\mu_1\mu_2\cdot|\mathbb{S}_{k_1+k_2}|,$$ and we conclude that the dimension and measure of the set $X$ is \[\left(n_1+n_2,k_1+k_2,\mu_1\mu_2\right)=(n_1,k_1,\mu_1)\odot (n_2,k_2,\mu_2).\] This finishes the proof of Theorem \ref{th-Fubini}.
\end{proof}

\begin{proof}[Proof of Claim \ref{tau-prime-measuring}]
We will show that conditions (i)-(iv) from Definition \ref{def-measuring-formula} hold for $\tau'(\bar{x},\bar{y},\bar{z},\bar{v})$ and $G(f)$.

\textbf{(i)} Note that $G(f)$ is defined over $\emptyset$, and so is the formula $\tau'(\bar{x},\bar{y},\bar{z},\bar{v})$.

\textbf{(ii)} By Lemma \ref{lem-addSU}, $\operatorname{SU}(G(f))=\omega\cdot(n_1+n_2)+(k_1+k_2)$. Also, note that $\bar{z},\bar{v}$ is a tuple of variables of length $k_1+k_2$. By definition of measuring formulas, $\varphi(\bar{b},H^{k_2})$ is finite for any $\bar{b}$, and $\eta'(\bar{a},H^{k_1};\bar{b},\bar{h})$ is finite for any $\bar{a},\bar{b},\bar{h}$. Thus, $\tau'(\bar{a},\bar{b},H^{k_1+k_2})$ is finite for any $\bar{a},\bar{b}$.

Let $Z$ be the set defined by $\exists\bar{z},\bar{v}\in H^{k_1+k_2}\tau'(\bar{x},\bar{y},\bar{z},\bar{v})$. We need to show that $\operatorname{SU}(G(f)\triangle Z)<\omega\cdot(n_1+n_2)+(k_1+k_2)$, that is, $G(f)$ and $Z$ have the same set of generics. 

Let $(\bar{a},\bar{b})\in G(f)$ with $\operatorname{SU}(\bar{a},\bar{b})=\omega\cdot(n_1+n_2)+(k_1+k_2)$. Then $\bar{b}\in Y$ and $\operatorname{SU}(\bar{b})=\omega\cdot n_2+k_2$. 
Let $\bar{h}\in H^{k_2}$ be an enumeration of $\operatorname{HB}(\bar{b})$. Since $\varphi(\bar{y},\bar{v})$ is a measuring formula for $Y$, we have that $M\models \varphi(\bar{b},\bar{h})$.
Also, we have $\operatorname{SU}(\bar{a}/\bar{b},\bar{h})=\omega\cdot n_1+k_1$ and $\bar{a}\in X_{\bar{b}}$. By item (4) of the listed properties in the proof of Theorem \ref{th-Fubini} the formula $\eta'(\bar{x},\bar{y};\bar{b},\bar{h})$ is a measuring formula for $X_{\bar{b}}$.
Let $\bar{h}_1\in H^{k_1}$ be an enumeration of $\operatorname{HB}(\bar{a}/\bar{b},\bar{h})$. 
Then we must have $M\models \eta'(\bar{a},\bar{h}_1;\bar{b},\bar{h})$. Therefore, we have $M\models \eta'(\bar{a},\bar{h}_1;\bar{b},\bar{h})\wedge \varphi(\bar{b},\bar{h})$, and in particular, $(\bar{a},\bar{b})\in Z$.

Now suppose that $(\bar{a},\bar{b})\in Z$ is a tuple of maximal $SU$-rank. Then there are $\bar{h}_1\in H^{k_1}$ and $\bar{h}\in H^{k_2}$ such that $M\models \eta'(\bar{a},\bar{h}_1;\bar{b},\bar{h})\land\varphi(\bar{b},\bar{h})$.
Since $(M,H)\models\exists\bar{v}\in H^{k_2}\varphi(\bar{b},\bar{v})$, we get $\operatorname{SU}(\bar{b})\leq \omega\cdot n_2+k_2$.
On the other hand, $\varphi(\bar{b},\bar{h})$ implies 
$$\forall\bar{x}\forall\bar{z}\left(\eta'(\bar{x},\bar{z};\bar{b},\bar{h})\to\bigvee_{j\leq N_0, \sigma'\in\mathbb{S}_{k_2}}\Psi_{p_{j}}(\bar{x},\bar{z};\bar{b},\sigma'(\bar{h}))\right)$$
for some types $p_j$ of $SU$-rank $\omega\cdot n_1+k_1$. Thus, $\Psi_{p_{j}}(\bar{a},\bar{h}_1;\bar{b},\sigma'(\bar{h}))$ holds for some $j\leq N_0$ and $\sigma'\in\mathbb{S}_{k_2}$. So, we get $\operatorname{SU}(\bar{a}/\bar{b},\bar{h})\leq\omega\cdot n_1+k_1$. Since $\varphi(\bar{y},\bar{v})$ is a measuring formula for $Y$, $\varphi(\bar{b},H^{k_2})$ is finite, and therefore $\bar{h}\subseteq \operatorname{HB}(\bar{b})$. Hence, $$\operatorname{SU}(\bar{a},\bar{b})=\operatorname{SU}(\bar{a}/\bar{b})\oplus \operatorname{SU}(\bar{b})=\operatorname{SU}(\bar{a}/\bar{b},\bar{h})\oplus \operatorname{SU}(\bar{b})\leq \omega\cdot(n_1+n_2)+(k_1+k_2).$$

Since we proved already that $Z$ contains the generics of $G(f)$, and $(\bar{a},\bar{b})$ is a tuple of maximal rank in $Z$, we must have that $\operatorname{SU}(\bar{a}/\bar{b})=\omega\cdot n_1+k_1$ and $\operatorname{SU}(\bar{b})=\omega\cdot n_2+k_2$. Again, since $\varphi(\bar{y},\bar{v})$ is a measuring formula for $Y$, we have that $\bar{b}\in Y$ and $\bar{h}$ is an enumeration of $\operatorname{HB}(\bar{b})$.
Therefore, $\eta'(\bar{x},\bar{z};\bar{b},\bar{h})$ is a measuring formula for $X_{\bar{b}}$. 
Since $\operatorname{SU}(\bar{a}/\bar{b})=\omega\cdot n_1+k_1$ and $\operatorname{SU}(X_{\bar{b}}\triangle \exists \bar{z}\,\eta'(\bar{x},\bar{z};\bar{b},\bar{h}))<\omega\cdot n_1+k_1$, we get that
$\bar{a}\in X_{\bar{b}}$, that is, $(\bar{a},\bar{b})\in G(f)$.

\textbf{(iii)} We will show that if $(\bar{a},\bar{b})\in G(\bar{f})$ and $\operatorname{SU}(\bar{a},\bar{b})=\omega\cdot(n_1+n_2)+(k_1+k_2)$, then for any $\bar{h}_0\in H^{k_1+k_2}$, $M\models \tau'(\bar{a},\bar{b},\bar{h}_0)$  if and only if $\bar{h}_0$ is an enumeration of $\operatorname{HB}(\bar{a},\bar{b})$. 

Suppose $\bar{h}_0$ is an enumeration of $\operatorname{HB}(\bar{a},\bar{b})$.
By additivity of $\operatorname{H}$-basis (Proposition \ref{additivityHB}), we have $\operatorname{HB}(\bar{a},\bar{b})=\operatorname{HB}(\bar{a}/\bar{b},\operatorname{HB}(\bar{b}))\cup \operatorname{HB}(\bar{b})$. Thus,
there is $\sigma\in \mathbb{S}_{k_1+k_2}$ such that $\sigma(\bar{h}_0)=(\bar{h}_1,\bar{h})$ where $\bar{h}$ is an enumeration of $\operatorname{HB}(\bar{b})$ and $\bar{h}_1$ an enumeration of $\operatorname{HB}(\bar{a}/\bar{b},\operatorname{HB}(\bar{b}))$. Since $\varphi(\bar{y},\bar{v})$ is a measuring formula for $Y$, the value of $\operatorname{SU}(\bar{b})$ is $\omega\cdot n_2+k_2$ and $\bar{h}$ is an enumeration of $\operatorname{HB}(\bar{b})$, we have $M\models \varphi(\bar{b},\bar{h})$. Similarly, $M\models \eta'(\bar{a},\bar{h}_1;\bar{b},\bar{h})$. Hence, $M\models \tau'(\bar{a},\bar{b},\bar{h}_0)$.

For the other direction, if $M\models \tau'(\bar{a},\bar{b},\bar{h}_0)$ for some $\bar{h}_0\in H^{k_1+k_2}$, then there is $\sigma\in\mathbb{S}_{k_1+k_2}$ such that $\sigma(\bar{h}_0)=(\bar{h}_a,\bar{h}_b)$, and $M\models \varphi(\bar{b},\bar{h}_b)\land\eta'(\bar{a},\bar{h}_a,\bar{b},\bar{h}_b)$. Since $\varphi$ is a measuring formula for $Y$, $\bar{h}_b$ is an enumeration of $\operatorname{HB}(\bar{b})$, and since $\eta'$ is a measuring formula for $X_{\bar{b}}$, $\bar{h}_a$ is an enumeration of $\operatorname{HB}(\bar{a}/\bar{b},\operatorname{HB}(\bar{b}))$.

\textbf{(iv)} By assumption, we have $\dim(\varphi(\bar{y},\bar{v}))=n_2+k_2$ and $\dim(\eta'(\bar{x},\bar{z};\bar{b},\bar{h}))=n_1+k_1$ for any parameter $\bar{b},\bar{h}$ satisfying $\varphi(\bar{b},\bar{h})$.
Therefore, since the dimension is additive, $$\dim(\tau'(\bar{x},\bar{y},\bar{z},\bar{v}))=(n_1+n_2)+(k_1+k_2).$$

Suppose now that $M\models \tau'(\bar{a},\bar{b},\bar{h}_0)$ and that $\dim(\bar{a},\bar{b},\bar{h}_0)=(n_1+n_2)+(k_1+k_2)$. Then there is $\sigma\in \mathbb{S}_{k_1+k_2}$ such that $\sigma(\bar{h}_0)=(\bar{h}_1,\bar{h})$ and $M\models \eta'(\bar{a},\bar{h}_1;\bar{b},\bar{h})\land\varphi(\bar{b},\bar{h})$.

Since $\dim(\eta'(\bar{x},\bar{z};\bar{b},\bar{h}))=n_1+k_1$ and $\dim(\varphi(\bar{y},\bar{v}))=n_2+k_2$, we must have 
$\dim(\bar{b},\bar{h})=n_2+k_2$ and
$\dim(\bar{a},\bar{h}_1/\bar{b},\bar{h})=n_1+k_1$. Also, since $\varphi$ is a measuring formula for $Y$, there are $(\bar{b}',\bar{h}')$ with the same $\mathcal{L}$-type as $(\bar{b},\bar{h})$ such that $\operatorname{SU}(\bar{b}')=\omega\cdot n_2+k_2$ and $\bar{h}'$ is an enumeration of $\operatorname{HB}(\bar{b}')$.

Let $p(\bar{x},\bar{y};\bar{b},\bar{h})=\tp_{\mathcal{L}}(\bar{a},\bar{h}_1/\bar{b},\bar{h})$
and consider the type $p(\bar{x},\bar{y};\bar{b}',\bar{h}')$. Let $(\bar{a}',\bar{h}_1')$ be a realization of $p(\bar{x},\bar{y};\bar{b}',\bar{h}')$. Then $M\models \eta'(\bar{a}',\bar{h}_1';\bar{b}',\bar{h}')$ and we also have $\dim(\bar{a}',\bar{h}_1'/\bar{b}',\bar{h}')=n_1+k_1$. Since $\bar{b}'\in Y$ and  $\eta'(\bar{x},\bar{z};\bar{b}',\bar{h}')$ is a measuring formula for $X_{\bar{b}'}$,
there are $\bar{a}'',\bar{h}_1''$ realizing $p(\bar{x},\bar{z};\bar{b}',\bar{h}')$ such that $\operatorname{SU}(\bar{a}''/\bar{b}',\bar{h}')=\omega\cdot n_1+k_1$ and $\bar{h}_1''$ is an enumeration of $\operatorname{HB}(\bar{a}''/\bar{b}',\bar{h}')$. Now the tuples $\bar{a}'',\bar{b}'$ and $\bar{h}_0'=\sigma^{-1}(\bar{h}_1'',\bar{h}')$ satisfy the following:
\begin{itemize}
\item $\tp_{\Le}(\bar{a}'',\bar{b}',\bar{h}_1'',\bar{h}')=\tp_{\Le}(\bar{a}',\bar{b}',\bar{h}_1',\bar{h}')=\tp_{\Le}(\bar{a},\bar{b},\bar{h}_1,\bar{h})$, so we can conclude that $\tp_{\Le}(\bar{a}'',\bar{b}',\bar{h}_0')=\tp_{\Le}(\bar{a},\bar{b},\bar{h}_0)$.
\item By additivity of $\operatorname{SU}$-rank, $\operatorname{SU}(\bar{a}'',\bar{b}')=\omega\cdot (n_1+n_2)+(k_1+k_2)$ and by the additivity of the $\operatorname{HB}$-basis, $\bar{h}_0'$ is an enumeration of $\operatorname{HB}(\bar{a}'',\bar{b}')$.\qedhere
\end{itemize}

\end{proof}

The results of this section can be summarized as:

\begin{theorem} \label{main-theorem}
Suppose $T=\operatorname{Th}(M)$ is the theory of a one-dimensional coherent measurable structure $M$. Then $T^{ind}$ is the theory of a generalized measurable structure with measure semi-ring \[\mathcal{R}=((\mathbb{N}\times \mathbb{N}^{>0}\times \mathbb{R}^{>0})\cup\{(0,0,k):k\in\mathbb{N}\},\oplus,\odot)\]
\end{theorem}
\begin{remark}\label{remark-FiniteSet}
It is easy to see that if a definable set $A$ in $(M,H)$ is $\mathcal{L}$-definable by some $\mathcal{L}$-formula $\varphi(x)$, then $\varphi(x)$ is a measuring formula for $A$, and $A$ has dimension $(\dim(\varphi(x),0)$ and measure $\mu(A)/0!=\mu(A)$.
In particular a finite set $A$ has dimension $(0,0)$ and measure $|A|$. 

But the calculation of dimension and measure of a finite set also follows from the Fubini property and finite additivity in general. Suppose $\{a\}$ has dimension and measure $\beta\in\mathcal{R}$, it is clear $\beta\neq (0,0,0)$. Now consider the definable map $id:\{a\}\to\{a\}$, then $\beta\odot\beta=\beta$ by Theorem \ref{th-Fubini}, hence $\beta=(0,0,1)$. Then finite additivity implies that the dimension and measure of a finite set is indeed $(0,0,|A|)$. 
\end{remark}

Theorem \ref{main-theorem} has some nice consequences.

\begin{remark}
Assume $(M,\cdot)$ is a one-dimensional coherent measurable group. Therefore, the (normalized) measure makes
the structure $(M,\cdot)$ a definably amenable group (see Definition 3.1 in \cite{CheSi}). Its expansion $(M,\cdot,H)$ is generalized measurable, which will induce a probability measure $\mu$ on definable subsets $X\subseteq M$ defined as \[\mu(X)=\begin{cases}\operatorname{measure}(X) &\text{if $\operatorname{dimension}(X)=(1,0)$ (i.e., if $\operatorname{SU}(X)=\omega$)}\\ 0 &\text{if $\operatorname{dimension}(X)=(0,k)$ (i.e., $\operatorname{SU}(X)<\omega$)}.\end{cases}\]

We will now see that this measure is invariant under translations. Let $a\in M$, and consider the map that sends $x$ to 
$a\cdot x$. If $X\subseteq M$ is $\mathcal{L}_H$-definable then $Y=a\cdot X$ is also $\Le_H$-definable and there a definable bijection $f_a:X\to Y$ defined by $f_a(x)=a\cdot x$. Moreover, for every $b\in Y$ we have $h(f_a^{-1}(b))=h(\{a^{-1}b\})=((0,0)\ ,\ 1)$. So, by the Fubini property in the $\Le_H$-structure $(M,\cdot,H)$ (Theorem \ref{th-Fubini}), we have that $h(X)=(\operatorname{dimension}(Y)\oplus(0,0),\operatorname{measure}(Y)\cdot 1)=h(Y)$. Therefore, if $SU(X)=\omega$,  \[\mu(X)=\operatorname{measure}(X)=\operatorname{measure}(Y)=\mu(Y),\] and $\mu(X)=0=\mu(Y)$ in all other cases because the $\operatorname{SU}$-rank is invariant under definable bijections. Now, if $X_1,\ldots,X_k$ are disjoint $\Le_H$-definable subsets of $M$ (at least one of $\operatorname{SU}$-rank $\omega$) then by the finite additivity property (Lemma \ref{lem-finAdd}) we will have \[\mu\left(\bigcup_{i=1}^k X_i\right)=\operatorname{measure}\left(\bigcup_{i=1}^k X_i\right)=\sum_{1\leq i\leq k, \operatorname{SU}(X_i)=\omega} \operatorname{measure}(X_i)=\sum_{1\leq i\leq k}\mu(X_i).\] This shows that $(M,\cdot,H)$ is still a definably amenable group.
\end{remark}

\brmk Let us consider the special case of measures of $\mathcal{L}_H$-definable subsets of $H(M)^k$ of $SU$-rank $k$. By Lemma \ref{fundlem2} a definable subset $Y$ of $H(M)^k$ is the set of solutions of a formula of the form $(\bar x\in H^k)\wedge \theta(\bar x,\bar c)$, where $\bar c$ is a tuple from $M$ and $\theta(\bar x,\bar y)$ is an $\mathcal{L}$-formula. Recall that $H$ with the induced structure from $M$ is a \emph{generic trivialization} of $M$ (see \cite{BeVa3}) and that there is a nice correspondence between its theory $T^*$ and the theory $T$. For example $T$ is supersimple of $SU$-rank 1 or strongly minimal if and only if $T^*$ is supersimple of $SU$-rank 1 or strongly minimal (see also \cite{BeVa3}). If $Y$ is a definable subset of $H(M)^k$ of dimension $k$, the density property implies that generic properties of $\theta(\bar x,\bar c)$ hold for $Y$. In our setting, this translates as the measure of $Y$ being equal to $\operatorname{meas}(\theta(\bar x,\bar c))$, where $\operatorname{meas}(\theta(\bar x,\bar c))$ is the $\mathcal{L}$-measure of the formula $\theta(\bar x,\bar c)$, as the following shows:
\bc\label{maxdimH}
Let $Y\subset H^k$ be a definable set of $SU$-rank $k$ and assume that $\theta(\bar x,\bar y)$ is an $\mathcal{L}$-formula and $\bar c$ is a tuple in $M$ such that $Y$ is defined by the expression 
$(\bar x\in H^k)\wedge \theta(\bar x,\bar c)$. Then the measure of $Y$ is the $\mathcal{L}$-measure of the formula $\theta(\bar x,\bar c)$. 
\ec

\begin{proof}
Let $\bar x$, $\bar z$ be tuples of variables of length $k$ and let $\varphi(\bar x,\bar z,\bar c)$ be the formula
\[\displaystyle{\bigvee_{\sigma \in \mathbb{S}_k} \theta(\sigma(\bar z),\bar c)\wedge (\bar x=\sigma(\bar z))}.\] Then
 $\varphi(\bar x,\bar z,\bar c)$ is a measuring formula for $Y$. Conditions (i), (ii), (iv) in Definition \ref{def-measuring-formula} are easy to verify. For (iii) note that since $\operatorname{SU}(Y)=k$, any tuple $\bar{h}:=(h_1,\ldots,h_k)$ in $Y$ with $h_i\neq h_j$ for $i\neq j$ and $\bar{h}\cap \operatorname{HB}(\bar{c})=\emptyset$ is of the maximal $\operatorname{SU}$-rank, and any $\sigma(\bar{h})$ for $\sigma\in\mathbb{S}_k$ is an enumeration of $\operatorname{HB}(\bar{h}/\bar{c},\operatorname{HB}(\bar{c}))$. Since $\varphi(\bar x,\bar z,\bar c)$ is invariant under permutations of $\bar{z}$, we get $\varphi(\bar{h},\sigma(\bar{h}),\bar{c})$ holds. Note that the set defined by $\varphi(\bar{x},\bar{z},\bar{c})$ is of the form $\bigcup_{\sigma\in \mathbb{S}_k}\{(\bar{a},\sigma^{-1}(\bar{a})):M\models \theta(\bar{a},\bar{c})\}$, which is generically disjoint. Hence, $\operatorname{meas}(\varphi(\bar x,\bar z,\bar c))=k!\cdot \operatorname{meas}(\theta(\bar x,\bar c))$.
Thus, the measure of $Y$ is $\operatorname{meas}(\varphi(\bar x,\bar z,\bar c))/k!=\operatorname{meas}(\theta(\bar x,\bar c))$.
\end{proof}
\ermk

\begin{definition}[Definition 3.1 \cite{BeVa3}]\label{def-generic-trivialzation} Let $(M,H)$ be an $H$-structure and consider the new structure $H^*(M)$ whose universe is $H$ and whose definable sets are generated by traces of the form $\varphi(H(M)^n)$, where $\varphi(\bar x)$ is an $\emptyset$-definable $\mathcal{L}$-formula. We call such structure a \emph{generic trivialization of $T$}. 
\end{definition}

It is shown in \cite{BeVa3} that the theory $T^*=Th(H^*(M))$ is complete, has quantifier elimination and depends only on $T$ and not on the specific $H$-structure $(M,H)$.

\bc If $M$ is a one-dimensional coherent measurable structure and $(M,H)$ is an $H$-structure, then the generic trivialization $H^*(M)$ is a measurable structure of dimension 1, namely the dimension and measure of the universe is $(1,1)$.
\ec
\bdem  
By quantifier elimination in trivializations, if $A\subseteq H^*(M)^n$ is definable in the generic trivialization, $A$ can also be seen as a definable subset in the structure $(M,H)$ (definable with parameters in $H$) and thus it has a measure and a dimension in the pair. We will show that the restriction of the dimension and measure from the pair $(M,H)$ provides a measure and a dimension for $H^*(M)$ satisfying Definition \ref{MS-measurable}. 

First note that $SU(H)=1$ and so $\operatorname{dimension}(H)=(0,1)$ and by Corollary \ref{maxdimH}, $\operatorname{measure}(H)=1$.
More generally, if $A\subseteq H^n$ is definable, then $SU(A)\leq n$ in $(M,H)$ and $\operatorname{dimension}(A)=(0,SU(A))$, hence the restriction of the dimension function will always be finite and bounded by the dimension of the underlying power of $H$.

Hence we have a function $h:\operatorname{Def}(H^*(M))\to \mathbb{N}\times \mathbb{R}^{>0}\cup\{(0,0)\}$ by sending a definable set $A\in \operatorname{Def}(H^*(M))$ of dimension $(0,k)$ and measure $\mu$ in $(M,H)$ (as an $\mathcal{L}_H$-definable set) to the value $(k,\mu)$.
By Remark \ref{remark-FiniteSet}, if $A$ is finite, then $h(A)=(0,|A|)$. 

The only delicate property left is definability of the function. Let $\varphi(\bar x,\bar y)$ be a formula over $\emptyset$ in $H^*(M)$. Then there is an $\mathcal{L}_H$-formula $\varphi'(\bar x,\bar y)$ over $\emptyset$ that defines the same subset of $H^{|\bar x|+|\bar y|}$  in $(M,H)$. For each value $(0,d,\mu)\in D_{\varphi'}$ choose an $\mathcal{L}_H$-formula $\psi'_{(0,d,\mu)}(\bar y)$ over $\emptyset$ such that
for any tuple $\bar a$, we have 
$h(\varphi'(\bar x,\bar a))=(d,\mu)$ if and only if
$(M,H) \models \psi'_{(0,d,\mu)}(\bar a)$. Since $\varphi'(M^{{|\bar x|+|\bar y|}})\subseteq H^{|\bar x|+|\bar y|}$, we may assume that 
$\psi'_{(0,d,\mu)}(M^{|\bar y|})\subseteq H^{|\bar y|}$ and thus we may assume that $\psi'_{(0,d,\mu)}(\bar y)=\theta_{(0,d,\mu)} (\bar y)\land (\bar{y}\in H^{|\bar y|})$ where $\theta_{(0,d,\mu)}(\bar y)$ is an $\mathcal{L}$-formula over $\emptyset$. We can use the trace of the $\mathcal{L}$-formula $\theta_{(0, d,\mu)}(\bar y)$ in $H^{|\bar y|}$ as the corresponding defining formula in the trivialization. Note that this new formula in the trivialization is again over $\emptyset$.

The other properties just follow from the corresponding properties of $h$ in the pair. For example, since $h$ attains finitely many values on each formula of the pair, it also attains finitely many values on uniformly definable subsets of $H^*(M)$.

\edem
\section{Questions}




In \cite{AMSW,Wolf} the authors define the notion of \emph{multidimensional asymptotic classes}, which are classes of finite structures whose ultraproducts are generalized measurable structures. If $M$ is an ultraproduct of a one-dimensional asymptotic class, then by Theorem \ref{main-theorem}, $(M,H)$ is a generalized measurable structure. By the results in \cite{Zou}, $(M,H)$ is also pseudofinite. 

\bpreg Is $(M,H)$ elementary equivalent to an ultraproduct of finite structures in a multidimensional asymptotic class?
\epreg
 The answer to this question might shed some light on whether the extra hypothesis that we used in Section \ref{section:comparing-dimensions} holds in general.

\end{document}